\documentclass[a4paper]{amsart}

\usepackage{amsrefs}
\usepackage[all]{xy}

\let\mod=\undefined
\DeclareMathOperator{\Id}{Id}
\DeclareMathOperator{\ql}{ql}
\DeclareMathOperator{\tr}{tr}

\DeclareMathOperator{\Aut}{Aut}
\DeclareMathOperator{\End}{End}
\DeclareMathOperator{\Hom}{Hom}
\DeclareMathOperator{\mod}{mod}
\DeclareMathOperator{\cone}{cone}
\DeclareMathOperator{\ldeg}{ldeg}
\DeclareMathOperator{\proj}{proj}
\DeclareMathOperator{\gldim}{gl.dim}
\DeclareMathOperator{\bdeg}{\mathbf{deg}}
\DeclareMathOperator{\bdim}{\mathbf{dim}}

\newcommand{\bbA}{\mathbb{A}}
\newcommand{\bbF}{\mathbb{F}}
\newcommand{\bbN}{\mathbb{N}}
\newcommand{\bbQ}{\mathbb{Q}}
\newcommand{\bbZ}{\mathbb{Z}}

\newcommand{\cA}{\mathcal{A}}
\newcommand{\cB}{\mathcal{B}}
\newcommand{\cC}{\mathcal{C}}
\newcommand{\cD}{\mathcal{D}}
\newcommand{\cH}{\mathcal{H}}
\newcommand{\cI}{\mathcal{I}}
\newcommand{\cK}{\mathcal{K}}
\newcommand{\cT}{\mathcal{T}}
\newcommand{\cX}{\mathcal{X}}

\newcommand{\cZ}{\mathcal{Z}}

\newcommand{\bd}{\mathbf{d}}

\newcommand{\bs}{\mathbf{s}}
\newcommand{\bt}{\mathbf{t}}

\newcommand{\frakq}{\mathfrak{q}}

\newcounter{claim}[section]

\newtheorem{coro}[claim]{Corollary}
\newtheorem{lemm}[claim]{Lemma}
\newtheorem{prop}[claim]{Proposition}
\newtheorem{theo}[claim]{Theorem}

\numberwithin{equation}{section}
\renewcommand{\theequation}{\thesection.\arabic{equation}}

\newcommand{\ol}{\overline}

\newcommand{\vertexD}[1]{\bullet \save*+!D{\scriptstyle #1} \restore}
\newcommand{\vertexL}[1]{\bullet \save*+!L{\scriptstyle #1} \restore}
\newcommand{\vertexR}[1]{\bullet \save*+!R{\scriptstyle #1} \restore}
\newcommand{\vertexU}[1]{\bullet \save*+!U{\scriptstyle #1} \restore}

\title[Derived Hall algebras of one-cycle gentle algebras]{Derived Hall algebras of one-cycle gentle algebras: The infinite global dimension case}

\author{Grzegorz Bobi\'nski}

\address{Faculty of Mathematics and Computer Science \newline Nicolaus Copernicus University \newline ul.\ Chopina 12/18 \newline 87-100 Toru\'n \newline Poland}

\email{gregbob@mat.umk.pl}

\author{Janusz Schmude}

\address{Institute of Informatics \newline University of Warsaw \newline ul.\ Banacha 2 \newline 02-097 Warsaw \newline Poland}

\email{j.schmude@mimuw.edu.pl}



\begin{document}

\begin{abstract}
We describe, in terms of generators and relations, the derived Hall algebras associated to the one-cycle gentle algebras of infinite global dimension.
\end{abstract}

\maketitle

\section{Introduction}

Given an abelian category $\cA$ satisfying some finitary conditions, one associates to $\cA$ the Hall algebra $\cH (\cA)$, whose basis is formed by the isomorphism classes of objects in $\cA$ and the multiplication is given by counting the exact sequences with given end terms (see~\cite{Macdonald} for a classical situation of modules over a discrete valuation ring). In the context of finite dimensional algebras the Hall algebras have been first studied by Ringel~\cite{Ringel}. Since then they have attracted a lot of interest and showed connections to different topics, including quantum group (see the survey article~\cite{DengXiao} and the references therein).

The Hall algebras do not generalize trivially to triangulated categories. A proper way to construct such a generalization was found by To\"{e}n~\cite{Toen}. An explicit form of such algebras have been calculated in only few cases, including the derived categories of hereditary algebras~\cite{Toen} and the triangulated categories generated by spherical objects~\cite{KellerYangZhou}. The aim of this paper is to calculate the Hall algebras associated to the derived categories of the one-cycle gentle algebras of infinite global dimension.

The gentle algebras have been introduced by Assem and Skowro\'nski~\cite{AssemSkowronski}. They are the algebras Morita equivalent to the path algebras of bound quivers satisfying some combinatorial conditions (see Section~\ref{sect gentle} for details). A gentle algebra is called one-cycle if there is exactly one cycle in the corresponding quiver. We say that a one-cycle gentle algebra satisfies the clock condition if the number of the clockwise oriented relations on the cycle coincides with the number of the anticlockwise oriented ones. Assem and Skowro\'nski~\cite{AssemSkowronski} have proved that the one-cycle gentle algebras satisfying the clock condition are precisely the algebras derived equivalent to the hereditary algebras of Euclidean type $\tilde{\bbA}$. One should also mention an earlier result of Assem and Happel~\cite{AssemHappel} stating that the tree (the corresponding quiver is a tree) gentle algebras are the algebras derived equivalent to the hereditary algebras of Dynkin type $\bbA$.

According to a result of Vossieck~\cite{Vossieck}, over an algebraically closed base field the one-cycle gentle algebras not satisfying the clock condition are precisely the derived discrete algebras, which are not derived representation finite. Recall that an algebra $\Lambda$ is derived discrete if for each homology dimension vector there are only finitely many isomorphism classes of complexes in the derived category of $\Lambda$. Moreover, an algebra $\Lambda$ is derived representation finite if up to the suspension functor there are only finitely many isomorphism classes of the indecomposable complexes in the derived category of $\Lambda$. Vossieck's results has been an inspiration for studying the derived categories of the one-cycle gentle algebras not satisfying the clock condition~\cites{Bobinski, BobinskiKrause, Broomhead, BroomheadPauksztelloPloogI, BroomheadPauksztelloPloogII}. In this paper we continue this line of research by describing the Hall algebras associated to the derived categories of the one-cycle gentle algebras of infinite global dimension. One should be aware that if a gentle one-cycle algebras has an infinite global dimension, then it does not satisfy the clock condition.

The paper is organized as follows. In Section~\ref{sect prelim} we introduce tools we use throughout the paper, including a method of calculating cones. Next, in Section~\ref{sect category} we describe the category for which we calculate the Hall algebra. In Sections~\ref{sect res1} and~\ref{sect res2} the actual calculations of Hall algebras are performed. Finally, in Section~\ref{sect gentle} we explain how these calculations apply to a description of the Hall algebras associated to the derived categories of the one-cycle gentle algebras of infinite global dimension.

Throughout the paper $\bbF$ is a fixed finite field. All considered categories are $\bbF$-linear. By $\bbZ$, $\bbN$ and $\bbN_+$ we denote the sets of integers, nonnegative integers and positive integers, respectively. If $i, j \in \bbZ$, then $[i, j] := \{ k \in \bbZ : \text{$i \leq k \leq j$} \}$. Finally, if $I$ is a set, then $\bbN^{(I)}$ denotes the set of sequences of nonnegative integers indexed by $I$, whose all but finite number of elements are $0$.

The first named author acknowledges the support of the National Science Center grant no.\ 2015/17/B/ST1/01731. 

\section{Preliminaries} \label{sect prelim}

Throughout this section $\cT$ is a triangulated category with a suspension functor $\Sigma$.

\subsection{Cones} \label{subsect cones}

In this subsection we assume that $\cT$ has split idempotents.

Let $X$ be an object of $\cT$. Following~\cite[Section~I.2]{AuslanderReitenSmalo} we denote by $X \backslash \cT$ the class of morphisms $f \colon X \to Y$ with $Y \in \cT$. If $f \colon X \to Y$ and $f' \colon X \to Y'$ are two objects of $X \backslash \cT$, then by $\Hom (f, f')$ we denote the set of $g \colon Y \to Y'$ such that $f' = g \circ f$, i.e.\ the following diagram commutes
\[
\vcenter{\xymatrix{X \ar[r]^f \ar[rd]_{f'} & Y \ar[d]^g \\ & Y'}}.
\]
For $f, f' \in X \backslash \cT$, we write $f \sim f'$ if $\Hom (f, f') \neq \emptyset \neq \Hom (f', f)$. This is an equivalence relation. For $f \in X \backslash \cT$, we denote by $[f]$ the equivalence class of $f$ in the relation $\sim$. For $f, f' \in X \backslash \cT$, we write $[f] \leq [f']$ if $\Hom (f, f') \neq \emptyset$. Obviously, $\leq$ is an order relation.

Now let $f \colon X \to Y$ be a morphism. By $f \backslash \cT$ we denote the subcategory of $Y \backslash \cT$ consisting of $g$ such that $g \circ f = 0$. Let $[f \backslash \cT]$ denote the set of the corresponding equivalence classes. We have the following observation.

\begin{lemm} \label{lemm minimal}
If
\[
X \xrightarrow{f} Y \xrightarrow{g} Z \to \Sigma X
\]
is an exact triangle, then $[g]$ is the least element of $[f \backslash \cT]$.
\end{lemm}

\begin{proof}
This follows from the exact sequence
\[
\Hom_\cT (Z, -) \xrightarrow{\Hom_\cT (g, -)} \Hom_\cT (Y, -) \xrightarrow{\Hom_\cT (f, -)} \Hom_\cT (X, -). \qedhere
\]
\end{proof}

We explore now the above observation in order to identify the cone of a map in some cases.

Recall that a map $f \colon X \to Y$ is called left minimal if $\Hom (f, f)$ (in $X \backslash \cT$) consists of automorphisms of $Y$, i.e.\ if $g \circ f = f$, then $g$ is an automorphism.

The following is the main result of this section.

\begin{prop} \label{prop cone}
Let $f \colon X \to Y$ be a nonzero map with $X$ indecomposable. Assume $g \colon Y \to Z$ is a left minimal map such that $[g]$ is the least element of $[f \backslash \cT]$. Then there exists $h \colon Z \to \Sigma X$ such that
\[
X \xrightarrow{f} Y \xrightarrow{g} Z \xrightarrow{h} \Sigma X
\]
is an exact triangle. In particular, $\cone (f) \simeq Z$.
\end{prop}

We need additional observations before we give the proof of the proposition. The first one is the following.

\begin{lemm} \label{lemm cone}
Let
\[
X \xrightarrow{f} Y \xrightarrow{g} Z \xrightarrow{h} \Sigma X
\]
be an exact triangle. If $M$ is an object of $\cT$ and $f' :=  \left[
\begin{smallmatrix}
f \\ 0
\end{smallmatrix}
\right] \colon X \to Y \oplus M$, then $\cone (f') \simeq Z \oplus M$.
\end{lemm}

\begin{proof}
This follows from~\cite[Lemma~2.5]{PengXiao}, however we include a short proof for completeness. Observe that $f' = \mu_1 \circ f$, where $\mu_1 := \left[
\begin{smallmatrix}
\Id_Y \\ 0
\end{smallmatrix}
\right] \colon Y \to Y \oplus M$. If
\[
X \xrightarrow{f'} Y \oplus M \xrightarrow{g'} Z' \xrightarrow{h'} \Sigma X
\]
is an exact triangle, then the octahedral axiom implies that we have a commutative diagram
\[
\xymatrix{
\Sigma^{-1} M \ar[r]^-0 \ar[d]^{\Id} & Y \ar[r]^-{\mu_1} \ar[d]^g & Y \oplus M \ar[r] \ar[d]^{g'} & M \ar[d]^{\Id}
\\
\Sigma^{-1} M \ar[r]^-i & Z \ar[r] & Z' \ar[r] & M}
\]
for some $i \colon \Sigma^{-1} M \to Z$, whose rows are exact triangles. The commutativity of the first square implies $i = 0$, hence $\cone (f') \simeq Z' \simeq Z \oplus M$.
\end{proof}

The second one is the following variant of~\cite[Theorem~2.2]{AuslanderReitenSmalo}.

\begin{lemm} \label{lemm leftminimal}
If $f \colon X \to Y$, $f' \colon X \to Y'$, $[f] = [f']$ and $f$ is left minimal, then, up to isomorphism, $Y' = Y \oplus M$, for some $M$, and $f' = \left[
\begin{smallmatrix}
f \\ 0
\end{smallmatrix}
\right]$.
\end{lemm}


\begin{proof}
Since $[f] = [f']$, there exist maps $g \colon Y \to Y'$ and $g' \colon Y' \to Y$ such that $g \circ f = f'$ and $g' \circ f' = f$. Then $g' \circ g \in \Hom (f, f)$, hence $g' \circ g$ is an automorphism. Since the idempotents split in $\cT$, the claim follows by standard arguments.
\end{proof}

\begin{proof}[Proof of Proposition~\ref{prop cone}]
Let
\[
X \xrightarrow{f} Y \xrightarrow{g'} Z' \xrightarrow{h'} \Sigma X
\]
be an exact triangle. Our assumptions and Lemma~\ref{lemm minimal} imply that $[g] = [g']$. Now Lemma~\ref{lemm leftminimal} implies that (up to isomorphism) $Z' = Z \oplus M$, for some $M$, and $g' = \left[
\begin{smallmatrix}
g \\ 0
\end{smallmatrix}
\right]$. Consequently, $\Sigma X \simeq \cone (g') \simeq \cone (g) \oplus M$ by Lemma~\ref{lemm cone}. Since $X$ is indecomposable, either $\cone (g) = 0$ or $M = 0$. However, $\cone (g) = 0$ means $g$ is an isomorphism, hence $f = 0$, a contradiction. Thus $M = 0$, i.e.\ $g' = g$ and $Z' = Z$, hence the claim follows if we take $h := h'$.
\end{proof}

We conclude this subsection with the following well-known easy observation.

\begin{lemm} \label{lemm not minimal}
Let $f \colon M \to N_0 \oplus N_1 \oplus N_2$ with $N_1 \neq 0 \neq N_2$. Let $f_i \colon M \to N_i$, $i = 0, 1, 2$, be the induced maps. If $[f_1] \leq [f_2]$, then $[f] = [f']$, where $f' := \left[
\begin{smallmatrix}
f_0 \\ f_1
\end{smallmatrix}
\right] \colon M \to N_0 \oplus N_1$.
\end{lemm}

\begin{proof}
Put $N' := N_0 \oplus N_1$ and write $f_2 = g \circ f_1$. Then we have $f' = h \circ f$ and $f = h' \circ f'$, where
\[
h := \left[
\begin{smallmatrix}
\Id_{N_0} & 0 & 0
\\
0 & \Id_{N_1} & 0
\end{smallmatrix}
\right] \colon N \to N' \qquad \text{and} \qquad h' := \left[
\begin{smallmatrix}
\Id_{N_0} & 0
\\
0 & \Id_{N_1}
\\
0 & g
\end{smallmatrix}
\right] \colon N' \to N. \qedhere
\]
\end{proof}

\subsection{Hall algebras} \label{subsect Hall}

In this subsection we assume that $\Hom (M, N)$ is finite, for all objects $M$ and $N$ of $\cT$, $\End (X)$ is local, for each indecomposable object $X$ of $\cT$, the isomorphism classes of objects in $\cT$ form a set, and if $M$ and $N$ are objects of $\cT$, then $\Hom (M, \Sigma^{-n} N) = 0$, for $n \gg 0$. Note that the first two conditions imply that $\cT$ is a Krull--Schmidt category. We denote by $\cT / \simeq$ a fixed set of representatives of the isomorphisms classes of objects in $\cT$. We may (and do) assume that $\cT / \simeq$ is closed under the action of $\Sigma$.

If $M$, $N$ and $L$ are objects of $\cT$, then we denote by $\Hom (M, L)_N$ the set of $f \in \Hom (M, L)$ such that $\cone (f) \simeq N$, and put
\[
F_{M, N}^L := \tfrac{|\Hom (M, L)_N|}{|\Aut (M)|} \cdot \tfrac{\{ M, L \}}{\{ M, M \}},
\]
where
\[
\{ X, Y \} := \prod_{n \in \bbN_+} |\Hom (X, \Sigma^{-n} Y)|^{(-1)^n},
\]
for objects $X$ and $Y$ of $\cT$. We define the Hall algebra $\cH := \cH (\cT)$ associated to $\cT$ in the following way: $\cH$ is the $\bbQ$-space with basis $\cT / \simeq$ and the multiplication is given by the formula
\begin{equation}
\renewcommand{\theequation}{$*$}
M \cdot N := \sum_{L \in \cT / \simeq} F_{M, N}^L L.
\end{equation}
It has been proved in~\cites{Toen, XiaoXu} that $\cH$ is an associative algebra with the identity element given by the zero object. If $\cT$ is the bounded derived category of an algebra $\Lambda$, then we call $\cH$ the derived Hall algebra of $\Lambda$.

We present now a slight modification of the definition of $F_{M, N}^L$. Let
\[
[X, Y] := \prod_{n \in \bbN} |\Hom (X, \Sigma^{-n} Y)|^{(-1)^n}.
\]
Moreover, recall that~\cite[Proposition 2.5']{XiaoXu} says
\[
F_{M, N}^L = \tfrac{|\Hom (L, N)_{\Sigma M}|}{|\Aut (N)|} \cdot \tfrac{\{ L, N \}}{\{ N, N \}}.
\]
We have the following easy observation.

\begin{lemm} \label{lemm F}
If $M, N, L \in \cT$, then
\[
F_{M, N}^L = \tfrac{|\End (M)|}{|\Aut (M)|} \cdot \tfrac{|\Hom (M, L)_N|}{|\Hom (M, L)|} \cdot \tfrac{[M, L]}{[M, M]} = \tfrac{|\End (N)|}{|\Aut (N)|} \cdot \tfrac{|\Hom (L, N)_{\Sigma M}|}{|\Hom (L, N)|} \cdot \tfrac{[L, N]}{[N, N]}.
\]
Moreover, if $[M, L] = v \cdot [M, M] \cdot [M, N]$, for some $v \in \bbQ$, then
\[
F_{M, N}^L = v \cdot \tfrac{|\End (M)|}{|\Aut (M)|} \cdot \tfrac{|\Hom (M, L)_N|}{|\Hom (M, L)|} \cdot [M, N].
\]
Similarly, if $[L, N] = v \cdot [M, N] \cdot [N, N]$, for some $v \in \bbQ$, then
\[
F_{M, N}^L = v \cdot \tfrac{|\End (N)|}{|\Aut (N)|} \cdot \tfrac{|\Hom (L, N)_{\Sigma M}|}{|\Hom (L, N)|} \cdot [M, N].
\]
\end{lemm}

\begin{proof}
This follows immediately from the obvious formula
\[
\{ X, Y \} = \tfrac{[X, Y]}{|\Hom (X, Y)|}. \qedhere
\]
\end{proof}

We formulate now some conditions for the multiplicative behavior of $[M, L]$.

\begin{lemm} \label{lemm mult}
Let $M, N, L \in \cT$.
\begin{enumerate}

\item \label{point one}
If $L = M \oplus N$, then
\[
[M, L] = [M, M] \cdot [M, N] \qquad \text{and} \qquad [L, N] = [M, N] \cdot [N, N].
\]

\item \label{point two}
If $F_{M, N}^L \neq 0$ and $\Hom (M, \Sigma M) = 0$, then
\[
[M, L] = [M, M] \cdot [M, N] \qquad \text{and} \qquad [L, M] = \tfrac{|\Hom (L, \Sigma M)|}{|\Hom (N, \Sigma M)|} \cdot [M, M] \cdot [M, N].
\]

\item \label{point three}
If $F_{M, N}^L \neq 0$ and $\Hom (N, \Sigma N) = 0$, then
\[
[N, L] = \tfrac{|\Hom (N, \Sigma L)|}{|\Hom (N, \Sigma M)|} \cdot [N, M] \cdot [N, N] \qquad \text{and} \qquad [L, N] = [M, N] \cdot [N, N].
\]

\end{enumerate}
\end{lemm}

\begin{proof}
The formulas in~\eqref{point one} are clear. For~\eqref{point two} we fix an exact triangle $M \to L \to N \to \Sigma M$. Applying the functor $\Hom (M, -)$ to this triangle, we get a long exact sequence
\begin{multline*}
\cdots \to \Hom (M, \Sigma^{-n} M) \to \Hom (M, \Sigma^{-n} L) \to \Hom (M, \Sigma^{-n} N) \to \cdots
\\
\to \Hom (M, M) \to \Hom (M, L) \to \Hom (M, N) \to \Hom (M, \Sigma M) = 0,
\end{multline*}
which immediately implies the former formula. Similarly, we get the latter formula applying the functor $\Hom (-, M)$ and using that $\Hom (\Sigma^n Z, M) = \Hom (Z, \Sigma^{-n} M)$. In particular, $\Hom (\Sigma^{-1} M, M) = \Hom (M, \Sigma M) = 0$.
\end{proof}

We formulate the following consequence.

\begin{coro} \label{coro direct sum}
If $X \not \simeq Y$ are indecomposable, then
\[
F_{X, Y}^{X \oplus Y} = [X, Y].
\]
\end{coro}

\begin{proof}
Let $X \xrightarrow{f} X \oplus Y \to Y \to \Sigma X$ be an exact triangle. Since $\Hom (M, N)$ is finite, for all $M, N \in \cT$, it follows that the induced sequence
\[
0 \to \Hom (Y, X) \to \Hom (X \oplus Y, X) \to \Hom (X, X) \to 0
\]
is exact, thus $f$ is a section. This implies $\Hom (X, X \oplus Y)_Y = \Aut (X) \times \Hom (X, Y)$, since $X \not \simeq Y$ (we use here that $\End (X)$ is a local ring). Now the claim follows from and Lemmas~\ref{lemm F} and~\ref{lemm mult}\eqref{point one}.
\end{proof}

\subsection{Dimension vectors} \label{subsect dimvect}

Let $I$ and $J$ be sets. In each of the sets $\bbN^{(I)}$ and $\bbN^{(J)}$ we compare the vectors componentwise. Next, if $\bd_1, \bd_1' \in \bbN^{(I)}$ and $\bd_2, \bd_2' \in \bbN^{(J)}$, then we write $(\bd_1, \bd_2) \leq (\bd_1', \bd_2')$ if either $\bd_1 < \bd_1'$ or $\bd_1 = \bd_1'$ and $\bd_2 \leq \bd_2'$.

Let $\cT$ be a triangulated category and assume we have functions $\bdim_I \colon \cT \to \bbN^{(I)}$ and $\bdim_J \colon \cT \to \bbN^{(J)}$ such that $\bdim M = 0$ if and only if $M = 0$, where $\bdim M := (\bdim_I M, \bdim_J M)$. Our aim is to give conditions, which guarantee that for each exact triangle $M \to L \to N \to \Sigma M$,
\[
\bdim L \leq \bdim M + \bdim N.
\]
In such a situation we call $(\bdim_I, \bdim_J)$ a subadditive pair of functions. If additionally $I = \emptyset$, then we call $\bdim_J$ a subadditive function. Observe that if $(\bdim_I, \bdim_J)$ is a subadditive pair of functions, then $\bdim M = \bdim N$ provided $M \simeq N$.

Let
\begin{gather*}
\ql_I (M) := |\bdim_I M|, \qquad \ql_J (M) := |\bdim_J M|,
\\
\intertext{and}
\ql (M) := \ql_I (M) + \ql_J (M).
\end{gather*}
Here, if $\bd \in \bbN^{(K)}$, for a set $K$, then $|\bd| := \sum_{k \in K} \bd (k)$. Observe that $\ql (M) = 0$ if and only if $M = 0$.

We have the following.

\begin{prop} \label{prop sub}
Assume that the following conditions hold:
\begin{enumerate}

\addtocounter{enumi}{-1}

\item \label{cond sub0}
If $M \simeq N$, then $\bdim M = \bdim N$.

\item \label{cond sub1}
If $\ql (M) = 1$, then for each exact triangle $M \to L \to N \to \Sigma M$,
\[
\bdim L \leq \bdim M + \bdim N.
\]

\item \label{cond sub2}
If $\ql (M) > 1$, then there exists an exact triangle $M' \to M \to M'' \to \Sigma M'$ such that $M' \neq 0 \neq M''$ and
\[
\bdim M = \bdim M' + \bdim M''.
\]

\end{enumerate}
Then $(\bdim_I, \bdim_J)$ is a subadditive pair of functions.
\end{prop}

\begin{proof}
Let $M \xrightarrow{f} L \to N \to \Sigma M$ be an exact triangle. We prove by induction on $\ql (M)$ that
\[
\bdim L \leq \bdim M + \bdim  N.
\]

If $\ql (M) = 0$, then $M = 0$, thus $L \simeq N$, and the claim follows from condition~\eqref{cond sub0}. Next, if $\ql (M) = 1$, then the claim follows from condition~\eqref{cond sub1}.

Now assume $\ql (M) > 1$. Condition~\eqref{cond sub2} implies that there exists an exact triangle $M' \xrightarrow{g} M \to M'' \to \Sigma N'$ such that $M' \neq 0 \neq M''$ and
\[
\bdim M = \bdim M' + \bdim M''.
\]
Using the octahedral axiom for the maps $g$, $f$ and $f \circ g$, we obtain exact triangles
\[
M' \to L \to L' \to \Sigma M' \qquad \text{and} \qquad M'' \to L' \to N \to \Sigma M''.
\]
Since $\ql (M'), \ql (M'') < \ql (M)$, by the inductive hypothesis used for the above triangles, we get
\[
\bdim L \leq \bdim M' + \bdim L' \leq \bdim M' + \bdim M'' + \bdim N = \bdim M + \bdim N. \qedhere
\]
\end{proof}

We formulate the following consequence.

\begin{coro} \label{coro sub prim}
Let $(\bdim_I, \bdim_J)$ be a subadditive pair of functions. Let $\cT'$ be the full subcategory of $\cT$ formed by the objects $M$ such that $\bdim_I M = 0$. If $\cT'$ is closed with respect to $\Sigma$, then $\cT'$ is a triangulated subcategory of $\cT$ and $\bdim_J |_{\cT'}$ is a subadditive function on $\cT'$.
\end{coro}

\begin{proof}
We have to show that if $M \to L \to N \to \Sigma M$ is an exact triangle with $M, N \in \cT'$, then $L \in \cT'$. By assumption, $\bdim_I M = 0 = \bdim_I N$. Since $\bdim L \leq \bdim M + \bdim  N$, $0 \leq \bdim_I L \leq \bdim_I M + \bdim_I N = 0$. Consequently, $\bdim_I L = 0$, thus in particular $L \in \cT'$. Moreover, since $\bdim_I L = \bdim_I M + \bdim_I N$, $\bdim_J L \leq \bdim_J M + \bdim_J N$.
\end{proof}

\subsection{Polynomials}

Let $\cT$ be a Krull--Schmidt triangulated category. Suppose we have a subadditive pair of functions $\bdim = (\bdim_I, \bdim_J) \colon \cT \to \bbN^{(I)} \times \bbN^{(J)}$ such that $\bdim (M \oplus N) = \bdim M + \bdim N$, for any $M, N \in \cT$. Moreover, we assume that for each $i \in I$, there exists a unique (up to isomorphism) object $Z_i$ such that $(\bdim_I Z_i) (i') = \delta_{i, i'}$ and $\bdim_J Z_i = 0$, where $\delta_{x, y}$ is the Kronecker delta. Similarly, for each $j \in J$, there exists a unique (up to isomorphism) object $X_j$ such that $(\bdim_J X_j) (j') = \delta_{j, j'}$ and $\bdim_I X_j = 0$.

Let $\cA$ be the free algebra in the generators $z_i$, $i \in I$, and $x_j$, $j \in J$. If $\Phi \in \cA$, then we define the degree vector $\bdeg \Phi \in \bbN^{(I)} \times \bbN^{(J)}$ by the formula
\[
\bdeg \Phi := (\bdeg_I \Phi, \bdeg_J \Phi),
\]
where $\bdeg_I \Phi \in \bbN^{(I)}$ and $\bdeg_J \Phi \in \bbN^{(J)}$ are given by
\begin{gather*}
(\bdeg_I \Phi) (i) := \deg_{z_i} \Phi \qquad \text{and} \qquad (\bdeg_J \Phi) (j) := \deg_{x_j} \Phi.
\end{gather*}

Assume now that $\cT$ satisfies the conditions from subsection~\ref{subsect Hall} and let $\cH$ be the Hall algebra of $\cT$. Without loss of generality we may assume that, for each $i \in I$, $Z_i \in \cT / \simeq$ and, for each $j \in J$, $X_j \in \cT / \simeq$. Consequently, if $\Phi \in \cA$, then we view $\Phi ((Z_i), (X_j))$ as an element of $\cH$.

\begin{prop} \label{prop poly}
Assume that for each indecomposable object $M \in \cT / \simeq$ there exists a polynomial $\Phi_M \in \cA$ such that
\[
\Phi_M ((Z_i), (X_j)) = M \qquad \text{and} \qquad \bdeg \Phi_M \leq \bdim M.
\]
Then for each object $M \in \cT / \simeq$ there exists a polynomial $\Phi_M \in \cA$ with the above properties.
\end{prop}

\begin{proof}
We prove the claim by induction on $|\End (M)|$. If $|\End (M)| = 1$, then $M = 0$, and we take $\Phi_M = 0$. Next, if $M$ is indecomposable, then the claim follows from the assumption.

Now assume $M \simeq M_1 \oplus M_2$ with $M_1 \neq 0 \neq M_2$ and $M_1, M_2 \in \cT / \simeq$. If
\[
M_1 M_2 = \sum_{L \in \cT / \simeq} v_L L,
\]
then $v_M \neq 0$. By the inductive hypothesis we know
\[
M_1 = \Phi_{M_1} ((Z_i), (X_j)) \qquad \text{and} \qquad M_2 = \Phi_{M_2} ((Z_i), (X_j)),
\]
for $\Phi_{M_1}, \Phi_{M_2} \in \cA$ such that $\bdeg \Phi_{M_1} \leq \bdim M_1$ and $\bdeg \Phi_{M_2} \leq \bdim M_2$, since $|\End (M_1)|, |\End (M_2)| < |\End (M)|$. Moreover, standard arguments using exact Hom-sequences induced by an exact triangle $M_1 \to L \to M_2 \to \Sigma M_1$ imply that if $v_L \neq 0$ and $L \neq M$, then $|\End (L)| < |\End (M)|$, thus again by the inductive hypothesis $L = \Phi_L ((Z_i), (X_j))$, for $\Phi_L \in \cA$ such that $\bdim \Phi_L \leq \bdeg L$. Observe that for such $L$,
\[
\bdeg \Phi_L \leq \bdim L \leq \bdim M_1 + \bdim M_2 = \bdim M
\]
by assumptions of $\bdim$. Moreover,
\[
\bdeg (\Phi_{M_1} \Phi_{M_2}) \leq \bdeg \Phi_{M_1} + \bdim \Phi_{M_2} \leq  \bdim M_1 + \bdim M_2 = \bdim M.
\]
Consequently, the claim follows if we take
\[
\Phi_M = \tfrac{1}{v_M} \Bigl( \Phi_{M_1} \Phi_{M_2} - \sum_{\substack{L \neq M \\ v_L \neq 0}} v_L \Phi_L \Bigr). \qedhere
\]
\end{proof}

\section{The category} \label{sect category}

In this section we introduce the category, which we are interested in.

\subsection{Description}

Throughout this section we fix $r, m \in \bbN_+$. We describe a category $\cC := \cC (r, m)$, which is the path category of a quiver $\Gamma$ modulo some relations, i.e.\ the objects of $\cC$ are formal direct sums of vertices of $\Gamma$, and the homomorphism spaces between indecomposable objects are formed by the $\bbF$-linear combinations of paths with appropriate starting and terminating vertices modulo the ideal generated by the relations.

The vertices of the quiver $\Gamma$ are:
\[
X_{i, j}^{(k)}, \; k \in [1, r], \; i, j \in \bbZ, \; i \leq j, \qquad \text{and} \qquad  Z_i^{(k)}, \; k \in [1, r], \; i \in \bbZ.
\]
In order to shorten notation we put $X_i^{(k)} := X_{i, i}^{(k)}$, for $k \in [1, r]$ and $i \in \bbZ$. We also put $X_{i, i - 1}^{(k)} := 0$. We always compute the upper index modulo $r$, i.e.\ for example $r + 1 = 1$. Moreover, $k + 1 = k$ if $r = 1$.

Now we describe the arrows in $\Gamma$. Moreover, in order to easier describe relations, we associate to each arrow a number called degree.

We begin with describing the arrows starting at $X_{i, j}^{(k)}$, for $k \in [1, r]$ and $i \leq j$. First, if $s \in [i, j]$ and $t \geq j$, then there is an arrow
\[
\alpha_{(s, t), (i, j)}^{(k)} \colon X_{i, j}^{(k)} \to X_{s, t}^{(k)}
\]
of degree $0$. Next, if $s \in [i, j]$, then there is an arrow
\[
\beta^{(k)}_{s, (i, j)} \colon X_{i, j}^{(k)} \to Z_s^{(k)}
\]
of degree $1$. Finally, if $s \leq i + \delta_{k, r} m - 1$ and $t \in [i + \delta_{k, r} m - 1, j + \delta_{k, r} m - 1]$, then there is an arrow
\[
\gamma^{(k)}_{(s, t), (i, j)} \colon X_{i, j}^{(k)} \to X_{s, t}^{(k + 1)}
\]
of degree $2$.

Now we describe the arrows starting at $Z_i^{(k)}$, for $k \in [1, r]$ and $i \in \bbZ$. If $s \geq i$, then there is an arrow
\[
\alpha_{s, i}'^{(k)} \colon Z_i^{(k)} \to Z_s^{(k)}
\]
of degree $0$. Next, if $s \leq i + \delta_{k, r} m - 1 \leq t$, then there is an arrow
\[
\beta'^{(k)}_{(s, t), i} \colon Z_i^{(k)} \to X_{s, t}^{(k + 1)}
\]
o degree $1$.

Finally we describe the relations. First, if $k \in [1, r]$, $i, j \in \bbZ$, and $i \leq j$, then
\[
\alpha^{(k)}_{(i, j), (i, j)} = \Id_{X_{i, j}^{(k)}}.
\]
Similarly, if $k \in [1, r]$ and $i \in \bbZ$, then
\[
\alpha_{i, i}'^{(k)} = \Id_{Z_i^{(k)}}.
\]
Finally, let $f \colon X \to Y$ and $g \colon Y \to Z$ be arrows of degree $d$ and $d'$, respectively. If there exists an arrow $h \colon X \to Z$ of degree $d + d'$, then $g \circ f = h$. Otherwise, $g \circ f = 0$. For example, $\beta_{s, (i, j)}^{(k)} \circ \alpha_{(i, j), (i, i)}^{(k)}$ equals $\beta_{s, (i, i)}^{(k)}$, if $s = i$, and equals $0$ otherwise.

The degrees which we have attached to the arrows of $\Gamma$ extend to the maps in the path category of $\Gamma$. This means that a map $\bigoplus_i X_i \to \bigoplus Y_j$ in the path category of $\Gamma$ is homogenous of degree $d$ provided every component map is a linear combinations of maps of the form $\lambda \alpha_1 \cdots \alpha_n$, where $\lambda \in \bbF$ and $\alpha_1$, \ldots, $\alpha_n$ are arrows of degrees $d_1$, \ldots, $d_n$, respectively, such that $d_1 + \cdots + d_n = d$. Moreover, every map is a sum $\sum_d f_d$, where $f_d$ is a map of degree $d$. Finally, if $f$ and $g$ are maps of degrees $d$ and $d'$, respectively, then $g \circ f$ (if defined) is a map of degree $d + d'$.

Since the relations are homogeneous with respect to the above degree function, the degree function descends to the maps in $\cC$. Note that if $f$ is a homogenous map in $\cC$ of degree $d$, then $d \leq 2$. Moreover, if $X$ and $Y$ are indecomposable, then the space of maps $X \to Y$ of degree $d$ is at most one dimensional and, if nonzero, is spanned by an arrow. Moreover, the space of all maps $X \to Y$ is at most two dimensional. We use these observations freely.

For the rest of the section we assume that $\cC$ has a triangulated structure such that
\[
\Sigma X_{i, j}^{(k)} = X_{i + \delta_{k, r} m, j + \delta_{k, r} m}^{(k + 1)} \qquad \text{and} \qquad \Sigma Z_i^{(k)} = Z_{i + \delta_{k, r} m}^{(k + 1)}.
\]
One easily checks that $\cC$ satisfies the conditions of subsection~\ref{subsect Hall}, hence one may study the Hall algebra $\cH (\cC)$.

\subsection{Cones} \label{subsect conesC}

We want to use results of subsection~\ref{subsect cones} in order to describe some cones in $\cC$. For this we need to describe some left minimal maps. We first formulate a criterion, which we use later on.

Let $f \colon M \to N$ be a nonzero map and write $f = f_0 + f_1 + f_2$, where $f_d$ is homogeneous of degree $d$. We call the minimal $d$ such that $f_d \neq 0$ the lower degree of $f$ and denote it $\ldeg f$.

\begin{prop} \label{prop min crit}
Let $Y_1$, \ldots, $Y_n$ be indecomposable objects and
\[
f = \left[
\begin{smallmatrix}
f_1 & \cdots & f_n
\end{smallmatrix}
\right]^{\tr} \colon M \to Y_1 \oplus \cdots \oplus Y_n
\]
be such that $f_i \neq 0$, for each $i \in [1, n]$. Assume that, for $i \neq j$, $\ldeg h > \ldeg f_i - \ldeg f_j$, for each nonzero $h \in \Hom (Y_j, Y_i)$. Then $f$ is left minimal.
\end{prop}

\begin{proof}
For each $i \in [1, n]$, let $f_i'$ be the homogenous component of $f_i$ of degree $\ldeg f_i$. Let $g \in \End (Y_1 \oplus \cdots \oplus Y_n)$ and, for each $i \in [1, n]$, let $g_i$ be the homogenous component of degree $0$ of the induced map $Y_i \to Y_i$. Note that each $g_i$ is a multiplicity of $\Id_{Y_i}$. If $g \circ f = f$, then our assumptions imply that the homogenous component of degree $\ldeg f_i$ of the map $M \to Y_i$ induced by $g \circ f$ is $g_i \circ f_i'$. Thus $g_i \circ f_i' = f_i'$. Since $f_i' \neq 0$, this implies $g_i = \Id_{Y_i}$, for each $i \in [1, n]$.

Without loss of generality we may assume $\ldeg f_1 \leq \cdots \leq \ldeg f_n$. This implies the homogenous component of degree $0$ of $g$ is upper triangular. Since it has identities on the diagonal, $g$ is an automorphism, and the claim follows.
\end{proof}

As a first application of Proposition~\ref{prop min crit} we get the following.

\begin{lemm} \label{lemm arrowmin}
If $\alpha \colon X \to Y$ is an arrow in $\Gamma$, then $\alpha$ is a left minimal map in $\cC$. \qed
\end{lemm}

Now we concentrate on the case $X = X_i^{(k)}$. Observe that if $\alpha_1 \colon X \to Y_1$ and $\alpha_2 \colon X \to Y_2$ are arrows, then either there exists an arrow $\beta \colon Y_1 \to Y_2$ such that $\alpha_2 = \beta \circ \alpha_1$, or there exists an arrow $\beta \colon Y_2 \to Y_1$ such that $\alpha_1 = \beta \circ \alpha_2$. The above extends to maps.

\begin{lemm} \label{lemm factor}
Let $X := X_i^{(k)}$, for $k \in [1, r]$ and $i \in \bbZ$. If $Y_1$ and $Y_2$ are indecomposable, and $f_1 \colon X \to Y_1$ and $f_2 \colon X \to Y_2$, then either $[f_1] \leq [f_2]$ or $[f_2] \leq [f_1]$.
\end{lemm}

\begin{proof}
We may assume $f_1 \neq 0$ and $f_2 \neq 0$.

If $\dim \Hom (X, Y_1) = 1 = \dim \Hom (X, Y_2)$, then $f_1 = \lambda_1 \alpha_1$ and $f_2 = \lambda_2 \alpha_2$, for $\lambda_1, \lambda_2 \in \bbF^\times$ and arrows $\alpha_1 \colon X \to Y_1$ and $\alpha_2 \colon X \to Y_2$. By symmetry, $\alpha_2 = \beta \circ \alpha_1$, for an arrow $\beta \colon Y_1 \to Y_2$, hence $[f_1] \leq [f_2]$.

Now assume $\dim \Hom (X, Y_1) = 1$, but $\dim \Hom (X, Y_2) = 2$. In this case again $f_1 = \lambda_1 \alpha_1$, but $f_2 = \lambda_2 \alpha_2 + \mu (\gamma \circ \alpha_2)$, for $\lambda_1, \lambda_2, \mu \in \bbF$ and arrows $\alpha_1 \colon X \to Y_1$, $\alpha_2 \colon X \to Y_2$ and $\gamma \colon Y_2 \to Y_2$, such that $\lambda_1 \neq 0$, $\deg \alpha_2 = 0$ and $\deg \gamma = 2$. If $\lambda_2 = 0$ or $\mu = 0$, then we proceed as in the previous case, thus we assume $\lambda_2 \neq 0 \neq \mu$. Moreover, if $\alpha_2 = \beta \circ \alpha_1$, for an arrow $\beta \colon Y_1 \to Y_2$, then again $[f_1] \leq [f_2]$, thus assume $\alpha_1 = \beta \circ \alpha_2$, for an arrow $\beta \colon Y_2 \to Y_1$. Since $Y_1 \neq Y_2$, one checks that $\beta \circ \gamma \circ \alpha_2 = 0$. Consequently, $[f_2] \leq [f_1]$.

Finally assume $\dim \Hom (X, Y_1) = 2 = \dim \Hom (X, Y_2)$. In this case one checks that $Y_1 = Y_2$, thus $f_1 = \lambda_1 \alpha + \mu_1 (\gamma \circ \alpha)$ and $f_2 = \lambda_2 \alpha + \mu_2 (\gamma \circ \alpha)$, for $\lambda_1, \lambda_2, \mu_1, \mu_2 \in \bbF$ and arrows $\alpha \colon X \to Y_1$ and $\gamma \colon Y_1 \to Y_1$, such that $\deg \alpha = 0$ and $\deg \gamma = 2$. If either $\lambda_1 = 0$ or $\lambda_2 = 0$, then we proceed as in the previous cases. If $\lambda_1 \neq 0 \neq \lambda_2$, then
\[
f_2 = \bigl( \tfrac{\lambda_2}{\lambda_1} \cdot \Id_{Y_2} + \tfrac{\lambda_1 \mu_2 - \lambda_2 \mu_1}{\lambda_1^2} \cdot \gamma \bigr) \circ f_1,
\]
i.e.\ $[f_1] \leq [f_2]$.
\end{proof}

\begin{lemm} \label{lemm leftminX}
Let $X := X_i^{(k)}$, for $k \in [1, r]$ and $i \in \bbZ$. If $f \colon X \to M$ is nonzero, then there exists an arrow $\alpha \colon X \to Y$ in $\Gamma$ such that $[f] = [\alpha]$.
\end{lemm}

\begin{proof}
Write $M = \bigoplus_{j = 1}^n Y_j$ for $Y_1$, \ldots, $Y_n$ indecomposable. Then $f = \left[
\begin{smallmatrix}
f_1 & \ldots & f_n
\end{smallmatrix}
\right]^{\tr}$. It follows from Lemma~\ref{lemm factor}, that we may assume that $[f_1] \leq [f_j]$, for each $j \in [1, n]$, hence $[f] = [f_1]$ by Lemma~\ref{lemm not minimal}. Thus we may assume $M$ is indecomposable. Since $f \neq 0$, $f = g \circ \alpha$, for $g \in \Aut (M)$ and an arrow $\alpha \colon X \to M$, hence the claim follows.
\end{proof}

According to Lemmas~\ref{lemm cone} and~\ref{lemm leftminimal}, Lemma~\ref{lemm leftminX} together with the following give a description of the cones of the maps starting at $X_i^{(k)}$.

\begin{prop} \label{prop coneX}
Let $k \in [1, r]$ and $i \in \bbZ$.
\begin{enumerate}

\item
If $j \geq i$, then $\cone (\alpha_{(i, j), (i, i)}^{(k)}) = X_{i + 1, j}^{(k)}$.

\item
$\cone (\beta_{i, (i, i)}^{(k)}) = Z_{i + 1}^{(k)}$.

\item
If $j \leq i + \delta_{k, r} m - 1$, then $\cone (\gamma_{(j, i + \delta_{k, r} m - 1), (i, i)}^{(k)}) = X_{j, i + \delta_{k, r} m}^{(k + 1)}$.

\end{enumerate}
\end{prop}

\begin{proof}
We only prove the first formula, the remaining are proved similarly. Using Proposition~\ref{prop cone} and Lemma~\ref{lemm arrowmin} is suffices to show that $[\alpha']$, where $\alpha' := \alpha_{(i + 1, j), (i, j)}^{(k)}$, is the least element in $[\alpha \backslash \cC]$, where $\alpha := \alpha_{(i, j), (i, i)}^{(k)}$. Let $X' := X_{i, j}^{(k)}$ and fix $f \colon X' \to Y$ such that $f \circ \alpha = 0$. We may assume $Y$ is indecomposable, and $f$ is nonzero and homogeneous.

Assume first $f$ is of degree $0$. Then $Y = X_{s, t}^{(k)}$, for $s \in [i, j]$ and $t \geq j$, and without loss of generality $f = \alpha_{(s, t), (i, j)}^{(k)}$. If $s = i$, then $f \circ \alpha = \alpha_{(s, t), (i, i)}^{(k)} \neq 0$. Thus, $s \geq i + 1$ and
\[
f = \alpha_{(s, t), (i + 1, j)}^{(k)} \circ \alpha_{(i + 1, j), (i, j)}^{(k)},
\]
hence the claim follows in this case.

The cases, when $f$ is of degree $1$ and $2$ are treated similarly.
\end{proof}

We need a similar description of the cones of the maps starting at $Z_i^{(k)}$. We first introduce some notation. Assume $k \in [1, r]$ and $i \in \bbZ$. Choose
\[
\bs = (s_n < \cdots < s_1) \qquad \text{and} \qquad \bt = (t_1 < \cdots < t_n)
\]
such that $s_1 \leq i + \delta_{k, r} m - 1 \leq t_1$. Let
\[
f_{\bs, \bt, i}^{(k)} \colon Z_i^{(k)} \to \bigoplus_{p = 1}^n X_{s_p, t_p}^{(k + 1)}
\]
be the map, whose $p$'s component is given by $\beta_{(s_p, t_p), i}'^{(k)}$. If, additionally, $j \geq i$ is such that $t_n < j + \delta_{k, r} m - 1$ (in the case $n > 0$), then we extend $f_{\bs, \bt, i}^{(k)}$ to the map
\[
f_{\bs, \bt, j, i}'^{(k)} \colon Z_i^{(k)} \to \bigl( \bigoplus_{p = 1}^n X_{s_p, t_p}^{(k + 1)} \bigr) \oplus Z_j^{(k)},
\]
such that the induced map $Z_i^{(k)} \to Z_j^{(k)}$ is given by $\alpha_{j, i}'^{(k)}$.

We start with the following.

\begin{lemm}
In the above situation the maps $f_{\bs, \bt, i}^{(k)}$ and $f_{\bs, \bt, j, i}'^{(k)}$ are left minimal.
\end{lemm}

\begin{proof}
Observe that there are no nonzero maps of degree $0$ between different indecomposable direct summands of the codomain. Moreover, in the case of $f_{\bs, \bt, j, i}'^{(k)}$ there are no nonzero maps of degree $1$ from $Z_j^{(k)}$ to the other indecomposable direct summands of the codomain. Thus the claim follows from Proposition~\ref{prop min crit}.
\end{proof}

A more difficult observation is the following.

\begin{lemm}
Let $k \in [1, r]$ and $i \in \bbZ$. If $f \colon Z_i^{(k)} \to M$ is nonzero, then there exist $\bs$ and $\bt$ as above such that either $[f] = [f_{\bs, \bt, i}^{(k)}]$ or $[f] = [f_{\bs, \bt, j, i}'^{(k)}]$, for some $j \geq i$ such that $t_n < j + \delta_{k, r} m - 1$.
\end{lemm}

\begin{proof}
Put $Z := Z_i^{(k)}$. Let $f \colon Z \to M$. We prove the claim by induction on the number of indecomposable direct summands of $M$. First, if $M = M' \oplus M''$ and $f = \left[
\begin{smallmatrix}
f' \\ 0
\end{smallmatrix}
\right]$ with respect to this decomposition, then $[f] = [f']$ by Lemma~\ref{lemm not minimal}. Consequently, we may assume that for each indecomposable direct summand $Y$ of $M$ the induced map $Z \to Y$ is nonzero. Then
\[
M = \bigl( \bigoplus_{p = 1}^n X_{s_p, t_p}^{(k + 1)} \bigr) \oplus \bigl(  \bigoplus_{q = 1}^l Z_{j_q}^{(k)} \bigr),
\]
for some $s_1, \ldots, s_n, t_1, \ldots, t_n, j_1, \ldots, j_l \in \bbZ$ such that $s_p \leq i + \delta_{k, r} m - 1 \leq t_p$ and $j_q \geq i$. We show that we may assume $s_n < \cdots < s_1$, $t_1 < \cdots < t_n$, $l \leq 1$, $j_1 + \delta_{k, r} m - 1 > t_n$ (if $l = 1$ and $n > 0$), and, up to isomorphism, the induced maps are given by the corresponding arrows.

First assume $l > 1$. Without loss of generality we may assume $j_1 \leq j_2$. If $f_1 \colon Z \to Z_{j_1}^{(k)}$ and $f_2 \colon Z \to Z_{j_2}^{(k)}$ are the induced maps, then $[f_1] \leq [f_2]$ and we may remove the summand $Z_{j_2}^{(k)}$ by Lemma~\ref{lemm not minimal}. Thus by induction we may assume $l = 1$. Next assume $s_p \leq s_q$ and $t_p \leq t_q$, for some $p \neq q$. If $f_1 \colon Z \to X_{s_p, t_p}^{(k + 1)}$ and $f_2 \colon Z \to X_{s_q, t_q}^{(k + 1)}$ are the induced maps, then again $[f_1] \leq [f_2]$ and we use Lemma~\ref{lemm not minimal}. Thus we may assume $s_n < \cdots < s_1$ and $t_n > \cdots > t_1$. We argue similarly, if $l = 1$, $n > 0$ and $t_n \geq j_1 + \delta_{k, r} m - 1$. Finally, by using an appropriate automorphism of $M$, we may assume that the induced maps $Z \to X_{s_p, t_p}^{(k + 1)}$ and $Z \to Z_{j_1}^{(k)}$ are given by arrows.
\end{proof}

In order to describe the cones of the maps $f_{\bs, \bt, i}^{(k)}$ and $f_{\bs, \bt, j, i}'^{(k)}$ we introduce additional notation. Let $g_{\bs, \bt, i}^{(k)}$ be the map
\[
\bigoplus_{p = 1}^n X_{s_p, t_p}^{(k + 1)} \to \bigl( \bigoplus_{p = 1}^n X_{s_{p - 1}, t_p}^{(k + 1)} \bigr) \oplus Z_{s_n}^{(k + 1)}
\]
given by the matrix
\[
\left[
\begin{smallmatrix}
\alpha_{(s_0, t_1), (s_1, t_1)}^{(k + 1)} & 0 & \cdots & 0 & 0
\\
- \alpha_{(s_1, t_2), (s_1, t_1)}^{(k + 1)} & \alpha_{(s_1, t_2), (s_2, t_2)}^{(k + 1)} & \ddots & \vdots & \vdots
\\
0 & - \alpha_{(s_2, t_3), (s_2, t_2)}^{(k + 1)} & \ddots & 0 & 0
\\
\vdots & \ddots & \ddots & \alpha_{(s_{n - 2}, t_{n - 1}), (s_{n - 1}, t_{n - 1})}^{(k + 1)} & 0
\\
0 & \cdots & 0 &  - \alpha_{(s_{n - 1}, t_n), (s_{n - 1}, t_{n - 1})}^{(k + 1)} & \alpha_{(s_{n - 1}, t_n), (s_n, t_n)}^{(k + 1)}
\\
0 & \cdots & 0 & 0 & - \beta_{s_n, (s_n, t_n)}^{(k + 1)}
\end{smallmatrix}
\right],
\]
where $s_0 := i + \delta_{k, r} m$. Recall that if $t_1 = i + \delta_{k, r} m - 1$, then $X_{s_0, t_1}^{(k + 1)} = 0$. Moreover, in this case we put $\alpha_{(s_0, t_1), (s_1, t_1)}^{(k + 1)} := 0$. Similarly, $g_{\bs, \bt, j, i}'^{(k)}$ is the map
\[
\bigl( \bigoplus_{p = 1}^n X_{s_p, t_p}^{(k + 1)} \bigr) \oplus Z_j^{(k)} \to \bigoplus_{p = 1}^{n + 1} X_{s_{p - 1}, t_p}
\]
given by the matrix
\[
\left[
\begin{smallmatrix}
\alpha_{(s_0, t_1), (s_1, t_1)}^{(k + 1)} & 0 & \cdots & 0 & 0
\\
- \alpha_{(s_1, t_2), (s_1, t_1)}^{(k + 1)} & \alpha_{(s_1, t_2), (s_2, t_2)}^{(k + 1)} & \ddots & \vdots & \vdots
\\
0 & - \alpha_{(s_2, t_3), (s_2, t_2)}^{(k + 1)} & \ddots & 0 & 0
\\
\vdots & \ddots & \ddots & \alpha_{(s_{n - 1}, t_n), (s_n, t_n)}^{(k + 1)} & 0
\\
0 & \cdots & 0 &  - \alpha_{(s_n, t_{n + 1}), (s_n, t_n)}^{(k + 1)} & \beta_{(s_n, t_{n + 1}), j}'^{(k + 1)}
\end{smallmatrix}
\right],
\]
where $t_{n + 1} := j + \delta_{k, r} m - 1$. In particular, if $n = 0$, then $g_{\bs, \bt, j, i}'^{(k)} \colon Z_j^{(k)} \to X_{i + \delta_{k, r} m, j + \delta_{k, r} m - 1}^{(k + 1)}$ is given by $\beta_{(i + \delta_{k, r} m, j + \delta_{k, r} m) - 1, j}'^{(k + 1)}$.

\begin{lemm}
In the above situation the maps $g_{\bs, \bt, i}^{(k)}$ and $g_{\bs, \bt, j, i}'^{(k)}$ are left minimal.
\end{lemm}

\begin{proof}
The claim follows again from Proposition~\ref{prop min crit}. We leave details to the reader.
\end{proof}

The following observation will be crucial.

\begin{lemm} \label{lemm least}
\begin{enumerate}

\item
$[g_{\bs, \bt, i}^{(k)}]$ is the least element in $[f_{\bs, \bt, i}^{(k)} \backslash \cC]$.

\item
$[g_{\bs, \bt, j, i}'^{(k)}]$ is the least element in $[f_{\bs, \bt, j, i}'^{(k)} \backslash \cC]$.

\end{enumerate}
\end{lemm}

\begin{proof}
We only prove the former statement, the proof of the later one is analogous.

Put
\[
M := \bigoplus_{p = 1}^n X_{s_p, t_p}^{(k + 1)} \qquad \text{and} \qquad N := \bigl( \bigoplus_{p = 1}^n X_{s_{p - 1}, t_p}^{(k + 1)} \bigr) \oplus Z_{s_n}^{(k + 1)}.
\]
One easily verifies that $g_{\bs, \bt, i}^{(k)} \circ f_{\bs, \bt, i}^{(k)} = 0$. We need to show that if $g \circ f_{\bs, \bt, i}^{(k)} = 0$, for some $g \colon M \to X$, then $g$ factors through $g_{\bs, \bt, i}^{(k)}$. Without loss of generality we may assume $X$ is indecomposable, and $g$ is nonzero and homogeneous. Write $g = \left[
\begin{smallmatrix}
g_1 & \cdots & g_n
\end{smallmatrix}
\right]$, for $g_p \colon X_{s_p, t_p}^{(k + 1)} \to X$.

Assume first $g$ is of degree $0$. Then $X = X_{s, t}^{(k + 1)}$, for some $s \leq t$. Let $q$ be the minimal $p$ such that $g_p \neq 0$. Then $s \in [s_q, t_q]$ and $t \geq t_q$. Without loss of generality we may assume $g_q = \alpha_{(s, t), (s_q, t_q)}^{(k + 1)}$.

If either $q = n$ or $q < n$ and $t < t_{q + 1}$, then $g_p = 0$, for $p \neq q$. Consequently,
\[
g \circ f_{\bs, \bt, i}^{(k)} = \alpha_{(s, t), (s_q, t_q)}^{(k + 1)} \circ \beta_{(s_q, t_q), i}'^{(k)},
\]
hence the condition $g \circ f_{\bs, \bt, i}^{(k)} = 0$ implies $s \geq i + \delta_{k, r} m = s_0$, since otherwise $g \circ f_{\bs, \bt, i}^{(k)} = \beta_{(s, t), i}'^{(k)} \neq 0$. Thus
\[
g_q = \alpha_{(s, t), (s_q, t_q)}^{(k + 1)} = \alpha_{(s, t), (s_{q - 1}, t_q)}^{(k + 1)} \circ \alpha_{(s_{q - 1}, t_q), (s_q, t_q)}^{(k + 1)}.
\]
Moreover, if $q'$ is the minimal $p'$ such that $s \leq t_{p'}$, then, for each $p \in [q', q]$,
\[
\alpha_{(s, t), (s_p, t_{p + 1})}^{(k + 1)} \circ \alpha_{(s_p, t_{p + 1}), (s_p, t_p)}^{(k + 1)} = \alpha_{(s, t), (s_p, t_p)}^{(k + 1)} = \alpha_{(s, t), (s_{p - 1}, t_p)}^{(k + 1)} \circ \alpha_{(s_{p - 1}, t_p), (s_p, t_p)}^{(k + 1)}.
\]
Finally, if $q' > 1$, then
\[
\alpha_{(s, t), (s_{q' - 1}, t_{q'})}^{(k + 1)} \circ \alpha_{(s_{q' - 1}, t_{q'}), (s_{q' - 1}, t_{q' - 1})}^{(k + 1)} = 0,
\]
since $s > t_{q' - 1}$. Consequently, if
\[
h := \left[
\begin{smallmatrix}
0 & \cdots & 0 & \alpha_{(s, t), (s_{q' - 1}, t_{q'})}^{(k + 1)} & \cdots & \alpha_{(s, t), (s_{q - 1}, t_q)}^{(k + 1)} & 0 & \cdots & 0
\end{smallmatrix}
\right] \colon N \to X,
\]
then $g = h \circ g_{\bs, \bt, i}^{(k)}$.

Next assume $q < n$ and $t \geq t_{q + 1}$. Then
\[
g_q = \alpha_{(s, t), (s_q, t_q)}^{(k + 1)} = \alpha_{(s, t), (s_q, t_{q + 1})}^{(k + 1)} \circ \alpha_{(s_q, t_{q + 1}), (s_q, t_q)}^{(k + 1)}.
\]
Put
\[
h := \left[
\begin{smallmatrix}
0 & \cdots & 0 & \alpha_{(s, t), (s_q, t_{q + 1})}^{(k + 1)} & 0 & \cdots & 0
\end{smallmatrix}
\right] \colon N \to X,
\]
and $g' := g - h \circ g_{\bs, \bt, i}^{(k)}$. Then $g' \circ f_{\bs, \bt, i}^{(k)} = 0$. Moreover, if $g' = \left[
\begin{smallmatrix}
g_1' & \cdots & g_n'
\end{smallmatrix}
\right]$, then $g_1' = 0$, \ldots, $g_q' = 0$. Consequently, by the induction hypothesis there exists $h' \colon N \to X$ such that $g' = h' \circ g_{\bs, \bt, i}^{(k)}$, hence $g = (h + h') \circ g_{\bs, \bt, i}^{(k)}$.

Now assume $g$ is of degree $1$. Then $X = Z_j^{(k + 1)}$, for some $j$. Let $q$ be the minimal $p$ such that $g_p \neq 0$. Then $s_q \leq j$. Without loss of generality we may assume $g_q = \beta_{j, (s_q, t_q)}^{(k + 1)}$.

If $q = n$, then $g_1 = 0$, \ldots, $g_{n - 1} = 0$. Moreover,
\[
g_n = \beta_{j, (s_n, t_n)}^{(k + 1)} = \alpha_{j, s_n}'^{(k + 1)} \circ \beta_{s_n, (s_n, t_n)}^{(k + 1)},
\]
thus if we put
\[
h := \left[
\begin{smallmatrix}
0 & \cdots & 0 & \alpha_{j, s_n}'^{(k + 1)}
\end{smallmatrix}
\right] \colon N \to X,
\]
then $g = h \circ g_{\bs, \bt, i}^{(k)}$.

Now assume $q < n$. Again $g_1 = 0$, \ldots, $g_{q - 1} = 0$, and
\[
g_q = \beta_{j, (s_q, t_q)}^{(k + 1)} = \beta_{j, (s_q, t_{q + 1)}}^{(k + 1)} \circ \alpha_{(s_q, t_{q + 1}), (s_q, t_q)}^{(k + 1)}.
\]
Put
\[
h := \left[
\begin{smallmatrix}
0 & \cdots & 0 & \beta_{j, (s_q, t_{q + 1})}^{(k + 1)} & 0 & \cdots & 0
\end{smallmatrix}
\right] \colon N \to X
\]
and $g' := g - h \circ g_{\bs, \bt, i}^{(k)}$. Then $g' \circ f_{\bs, \bt, i}^{(k)} = 0$. Moreover, if $g' = \left[
\begin{smallmatrix}
g_1' & \cdots & g_n'
\end{smallmatrix}
\right]$, then $g_1' = 0$, \ldots, $g_q' = 0$. Consequently, by the induction hypothesis there exists $h' \colon N \to X$ such that $g' = h' \circ g_{\bs, \bt, i}^{(k)}$, hence $g = (h + h') \circ g_{\bs, \bt, i}^{(k)}$.

We proceed similarly, when $g$ is of degree $2$.
\end{proof}

As a consequence of the above we obtain.

\begin{prop} \label{prop coneZ}
\begin{enumerate}

\item
$\cone (f_{\bs, \bt, i}^{(k)}) = \bigl( \bigoplus_{p = 1}^n X_{s_{p - 1}, t_p}^{(k + 1)} \bigr) \oplus Z_{s_n}^{(k + 1)}$.

\item
$\cone (f_{\bs, \bt, j, i}'^{(k)}) = \bigoplus_{p = 1}^{n + 1} X_{s_{p - 1}, t_p}^{(k + 1)}$.

\end{enumerate}
\end{prop}

\begin{proof}
In both cases the claim follows from Proposition~\ref{prop cone} and Lemma~\ref{lemm least}.
\end{proof}

\subsection{Dimension vectors}

Now we apply considerations from subsection~\ref{subsect dimvect} to our category. For each $M \in \cC$ we define $\bdim_\cX M \in \bbN^{([1, r] \times \bbZ)}$ and $\bdim_\cZ M \in \bbN^{[1, r]}$ in the following way. First, if $M$ is indecomposable, then
\begin{align*}
(\bdim_\cX M) (k, i) & :=
\begin{cases}
1 & \text{if either $M = X_{s, t}^{(k)}$, for some $s \leq i \leq t$},
\\
& \quad \text{or $i < 0$ and $M = Z_s^{(k)}$, for some $s \leq i$},
\\
& \quad \text{or $i \geq \delta_{k, 1} m$ and $M = Z_s^{(k - 1)}$},
\\
& \qquad \text{for some $s \geq i - \delta_{k, 1} m + 1$},
\\
0 & \text{otherwise},
\end{cases}
\\
\intertext{and}
(\bdim_\cZ M) (k) & :=
\begin{cases}
1 & \text{if $M = Z_s^{(k)}$, for some $s \in \bbZ$},
\\
0 & \text{otherwise}.
\end{cases}
\end{align*}
Furthermore,
\begin{align*}
\bdim_\cX (M_1 \oplus M_2) & := \bdim_\cX M_1 + \bdim_\cX M_2,
\\
\intertext{and}
\bdim_\cZ (M_1 \oplus M_2) & := \bdim_\cZ M_1 + \bdim_\cZ M_2.
\end{align*}
Moreover, we put
\[
\ql_\cX (M) := |\bdim_\cX (M)| \qquad \text{and} \qquad \ql_\cZ (M) := |\bdim_\cZ (M)|.
\]
Finally,
\[
\bdim M := (\bdim_\cZ M, \bdim_\cX M) \qquad \text{and} \qquad \ql (M) := \ql_\cX (M) + \ql_\cZ (M).
\]
Obviously, if $M \simeq N$, then $\bdim M = \bdim N$.

The following lemmas will allow us to apply Proposition~\ref{prop sub}.

\begin{lemm} \label{lemm sub1}
If $\ql (M) = 1$ and $M \xrightarrow{f} L \to N \to \Sigma M$ is an exact triangle in $\cC$, then
\[
\bdim L \leq \bdim M + \bdim N.
\]
\end{lemm}

\begin{proof}
We have two cases to consider: $M = X_i^{(k)}$ and $M = Z_0^{(k)}$.

Assume first that $M = X_i^{(k)}$. If $f = 0$, then $N = L \oplus \Sigma M$, hence the claim is clear. If $f \neq 0$, then using Lemmas~\ref{lemm leftminX}, \ref{lemm arrowmin}, \ref{lemm leftminimal}, and~\ref{lemm cone}, we may assume $L = Y \oplus Z$ and $N = \cone (\alpha) \oplus Z$, for an object $Z$ and an arrow $\alpha \colon M \to Y$. Thus it suffices to show
\begin{equation}
\label{ineq dim}
\bdim Y \leq \bdim M + \bdim \cone (\alpha),
\end{equation}
for each arrow $\alpha \colon M \to Y$. Using Proposition~\ref{prop coneX} we know we are in one of the following cases:
\begin{enumerate}

\item
$Y = X_{i, j}^{(k)}$ and $\cone (\alpha) = X_{i + 1, j}^{(k)}$, for $j \geq i$,

\item
$Y = Z_i^{(k)}$ and $\cone (\alpha) = Z_{i + 1}^{(k)}$,

\item
$Y = X_{j, i + \delta_{k, r} m - 1}^{(k + 1)}$ and $\cone (\alpha) = X_{j, i + \delta_{k, r} m}^{(k + 1)}$, for $j \leq i + \delta_{k, r} m - 1$.

\end{enumerate}
Now one easily verifies the inequality~\eqref{ineq dim} in all the above cases. In fact we have equality in case~(1) and for $i < 0$ in case~(2), while the inequality is strict in the remaining cases.

Now assume $M = Z_0^{(k)}$. Similarly as in the previous case we may assume that either $f = f_{\bs, \bt, 0}^{(k)}$ or $f = f_{\bs, \bt, j, 0}'^{(k)}$, where we use freely notation from subsection~\ref{subsect conesC}. We use Proposition~\ref{prop coneZ},
thus in the former case we get
\[
L = \bigoplus_{p = 1}^n X_{s_p, t_p}^{(k + 1)} \qquad \text{and} \qquad N = \bigl( \bigoplus_{p = 1}^n X_{s_{p - 1}, t_p}^{(k + 1)} \bigr) \oplus Z_{s_n}^{(k + 1)},
\]
and the inequality $\bdim L \leq \bdim M + \bdim N$ follows trivially (since $\bdim_\cZ L = 0$, while $\bdim_\cZ M \neq 0 \neq \bdim_\cZ N$). In the latter case, we have
\[
L = \bigl( \bigoplus_{p = 1}^n X_{s_p, t_p}^{(k + 1)} \bigr) \oplus Z_j^{(k)} \qquad \text{and} \qquad N = \bigoplus_{p = 1}^{n + 1} X_{s_{p - 1}, t_p}^{(k + 1)},
\]
and one verifies that $\bdim L = \bdim M + \bdim N$.
\end{proof}

\begin{lemm} \label{lemm sub2}
If $\ql (M) > 1$, then there exists an exact triangle $M' \to M \to M'' \to \Sigma M$ such that $\bdim M = \bdim M' + \bdim M''$.
\end{lemm}

\begin{proof}
If $M$ is decomposable, then we take any nontrivial decomposition $M = M' \oplus M''$. Next, if $M = X_{i, j}^{(k)}$, then we take $M' = X_i^{(k)}$ and $M'' = X_{i + 1, j}^{(k)}$ ($\ql (M) > 1$ implies $i < j$). Finally, assume $Z = Z_i^{(k)}$. Note that $i \neq 0$, since $\ql (M) > 1$. If $i < 0$, then we take $M' = X_i^{(k)}$ and $M'' := Z_{i + 1}^{(k)}$, while if $i > 0$, then $M' = Z_{i - 1}^{(k)}$ and $M'' = X_{i + \delta_{r, k} m - 1}^{(k + 1)}$.
\end{proof}

We have the following consequences.

\begin{coro} \label{coro dim vect C}
The pair $(\bdim_\cZ, \bdim_\cX)$ is a subadditive pair of functions.
\end{coro}

\begin{proof}
This follows immediately from Proposition~\ref{prop sub} and Lemmas~\ref{lemm sub1} and~\ref{lemm sub2}.
\end{proof}

Let $\cX := \cX (r, m)$ be the full additive subcategory of $\cC$ whose indecomposable objects are $X_{i, j}^{(k)}$, $i \leq j$, $k \in [1, r]$.

\begin{coro} \label{coro dim vect X}
$\cX$ is a triangulated subcategory of $\cC$ and $\bdim_\cX |_\cX$ is a subadditive function.
\end{coro}

\begin{proof}
Observe that $\cX$ consists of the $M$'s such that $\bdim_\cZ M = 0$, hence the claim follows from Corollaries~\ref{coro sub prim} and~\ref{coro dim vect C}.
\end{proof}

\section{Results. I} \label{sect res1}

In this section we describe the Hall algebra $\cH' := \cH (\cX)$ associated to the category $\cX$ introduced in Section~\ref{sect category}. By abuse of notation in this section we write $\bdim M := \bdim_\cX M$, for $M \in \cX$.

\subsection{Relations}

We introduce the following notation:
\[
\lambda_i^{(k, l)} := [X_0^{(k)}, X_i^{(l)}].
\]
The precise formulas for $\lambda_i^{(k, l)}$ depend on $r$ and $m$ and are a little bit cumbersome, but for concrete values of $i$, $k$ and $l$, $\lambda_i^{(k, l)}$ is relatively easy to calculate. For example,
\[
\lambda_i^{(k, k)} = \frakq^{\varepsilon' + \varepsilon''},
\]
where
\[
\varepsilon' :=
\begin{cases}
(-1)^{\tfrac{i}{m} r} & \text{if $\tfrac{i}{m} \in \bbN$},
\\
0 & \text{otherwise},
\end{cases}
\qquad \text{and} \qquad
\varepsilon'' :=
\begin{cases}
(-1)^{\tfrac{i + 1}{m} r - 1} & \text{if $\tfrac{i + 1}{m} \in \bbN_+$},
\\
0 & \text{otherwise}.
\end{cases}
\]
For the rest of the paper we put $\lambda_i := \lambda_i^{(1, 1)}$, $i \in \bbZ$. Observe that $\lambda_i = \lambda_i^{(k, k)}$, for any $k \in [1, r]$. Moreover, $[X_i^{(k)}, X_j^{(l)}] = \lambda_{j - i}^{(k, l)}$, for any $k, l \in [1, r]$ and $i, j \in \bbZ$.

Our first aim is to calculate $X_i^{(k)} X_j^{(l)}$, for $(k, i) \neq (l, j)$. The first step in this direction is the following observation.

\begin{lemm} \label{lemm possL}
Let $k, l \in [1, r]$ and $i, j \in \bbZ$. If $F_{X_i^{(k)}, X_j^{(l)}}^L \neq 0$, then we have the following possibilities:
\begin{enumerate}

\item
$L = X_i^{(k)} \oplus X_j^{(l)}$;

\item
$L = X_{i, i + 1}^{(k)}$, if $k = l$ and $j = i + 1$;

\item
$L = 0$, if $l = k + 1$ and $j = i + \delta_{k, r} m$.

\end{enumerate}
\end{lemm}

\begin{proof}
Note that $L = \Sigma^{-1} (\cone (f))$, where $f \colon X_j^{(l)} \to \Sigma X_i^{(k)} = X_{i + \delta_{k, r} m}^{(k + 1)}$. Obviously, if $f = 0$, then $L = X_i^{(k)} \oplus X_j^{(l)}$. If $f \neq 0$, then (up to an automorphism of $X_j^{(i)}$) we have two possibilities: either $l = k + 1$, $j = i + \delta_{k, r} m$ and $f = \Id_{X_j^{(l)}}$, or $l = k$, $j = i + 1$ and $f = \gamma_{(i + \delta_{k, r} m, i + \delta_{k, r} m), (i + 1, i + 1)}^{(k)}$. In the former case $L = 0$, while in the latter case $L = X_{i, i + 1}^{(k)}$ by Proposition~\ref{prop coneX}.
\end{proof}

The second step in calculating $X_i^{(k)} X_j^{(l)}$ is the following.

\begin{lemm} \label{lemm FXX}
Let $(k, i) \in [1, r] \times \bbZ$. If $(l, j) \neq (k, i)$, then
\[
F_{X_i^{(k)}, X_j^{(l)}}^{X_i^{(k)} \oplus X_j^{(l)}} = \lambda_{j - i}^{(k, l)}.
\]
Moreover,
\[
F_{X_i^{(k)}, X_{i + 1}^{(k)}}^{X_{i, i + 1}^{(k)}} = \lambda_1 \qquad \text{and} \qquad  F_{X_i^{(k)}, X_{i + \delta_{k, r} m}^{(k + 1)}}^{0} = \tfrac{\frakq}{\frakq - 1} \lambda_{\delta_{k, r} m}^{(k, k + 1)}.
\]
\end{lemm}

\begin{proof}
The first formula follows from Corollary~\ref{coro direct sum}. Next, Lemmas~\ref{lemm F} and~\ref{lemm mult}\eqref{point two} imply
\[
F_{X_i^{(k)}, X_{i + 1}^{(k)}}^{X_{i, i + 1}^{(k)}} = \tfrac{|\End (X_i^{(k)})|}{|\Aut (X_i^{(k)})|} \cdot \tfrac{|\Hom (X_i^{(k)}, X_{i, i + 1}^{(k)})_{X_{i + 1}^{(k)}}|}{|\Hom (X_i^{(k)}, X_{i, i + 1}^{(k)})|} \cdot \lambda_1.
\]
Observe that
\begin{align*}
|\End (X_i^{(k)})| & =
\begin{cases}
\frakq^2 & \text{if $(r, m) = (1, 1)$},
\\
\frakq & \text{otherwise},
\end{cases}
\\
\intertext{and}
|\Aut (X_i^{(k)})| & =
\begin{cases}
\frakq^2 - \frakq & \text{if $(r, m) = (1, 1)$},
\\
\frakq - 1 & \text{otherwise}.
\end{cases}
\end{align*}
Similarly
\begin{align*}
|\Hom (X_i^{(k)}, X_{i, i + 1}^{(k)})_{X_{i + 1}^{(k)}}| & =
\begin{cases}
\frakq^2 - \frakq & \text{if $(r, m) = (1, 2)$},
\\
\frakq - 1 & \text{otherwise},
\end{cases}
\\
\intertext{and}
|\Hom (X_i^{(k)}, X_{i, i + 1}^{(k)})| & =
\begin{cases}
\frakq^2 & \text{if $(r, m) = (1, 2)$},
\\
\frakq & \text{otherwise},
\end{cases}
\end{align*}
which gives the second formula. The third formula is proved similarly.
\end{proof}

We get the following consequences.

\begin{coro}
Let $(k, i) \in [1, r] \times \bbZ$.
\begin{enumerate}

\item
If $(l, j) \neq (k, i), (k, i + 1), (k + 1, i + \delta_{k, r} m)$, then
\[
X_i^{(k)} X_j^{(l)} = \lambda_{j - i}^{(k, l)} (X_i^{(k)} \oplus X_j^{(l)}).
\]

\item
If $(r, m) \neq (1, 1)$, then
\begin{align*}
X_i^{(k)} X_{i + 1}^{(k)} & = \lambda_1 (X_i^{(k)} \oplus X_{i + 1}^{(k)}) + \lambda_1 X_{i, i + 1}^{(k)},
\\
\intertext{and}
X_i^{(k)} X_{i + \delta_{k, r} m}^{(k + 1)} & = \lambda_{\delta_{k, r} m}^{(k, k + 1)} (X_i^{(k)} \oplus X_{i + \delta_{k, r} m}^{(k + 1)}) + \tfrac{\frakq}{\frakq - 1} \lambda_{\delta_{k, r} m}^{(k, k + 1)}.
\end{align*}

\item
If $(r, m) = (1, 1)$, then
\[
X_i^{(k)} X_{i + 1}^{(k)} = \lambda_1 (X_i^{(k)} \oplus X_{i + 1}^{(k)}) + \lambda_1 X_{i, i + 1}^{(k)} + \tfrac{\frakq}{\frakq - 1} \lambda_1. \eqno \qed
\]

\end{enumerate}
\end{coro}

As a result of the above formulas we get the first two relations.

\begin{prop} \label{prop rel1}
Let $(k, i) \in [1, r] \times \bbZ$. If $(l, j) \neq (k, i), (k, i + 1), (k, i - 1), (k + 1, i + \delta_{k, r} m), (k - 1, i - \delta_{k, 1} m)$, then
\[
\tfrac{1}{\lambda_{j - i}^{(k, l)}} X_i^{(k)} X_j^{(l)} - \tfrac{1}{\lambda_{i - j}^{(k, l)}} X_j^{(l)} X_i^{(k)} = 0.
\]
Moreover, if $(r, m) \neq (1, 1)$, then
\[
\tfrac{1}{\lambda_{\delta_{k, r} m}^{(k, k + 1)}} X_i^{(k)} X_{i + \delta_{k, r} m}^{(k + 1)} - \tfrac{1}{\lambda_{- \delta_{k, r} m}^{(k + 1, k)}} X_{i + \delta_{k, r} m}^{(k + 1)} X_i^{(k)} = \tfrac{\frakq}{\frakq - 1}. \eqno \qed
\]
\end{prop}

We also observe the following.

\begin{lemm} \label{lemm multiply}
Let $(k, i) \in [1, r] \times \bbZ$. If $(r, m) \neq (1, 1)$, then
\[
\tfrac{1}{\lambda_1} X_i^{(k)} X_{i + 1}^{(k)} - \tfrac{1}{\lambda_{-1}} X_{i + 1}^{(k)} X_i^{(k)} = X_{i, i + 1}^{(k)}.
\]
If $(r, m) = (1, 1)$, then
\[
\tfrac{1}{\lambda_1} X_i^{(k)} X_{i + 1}^{(k)} - \tfrac{1}{\lambda_{-1}} X_{i + 1}^{(k)} X_i^{(k)} = X_{i, i + 1}^{(k)} + \tfrac{\frakq}{\frakq - 1}. \eqno \qed
\]
\end{lemm}

Now we want to calculate $X_i^{(k)} X_{i, i + 1}^{(k)}$, $X_{i, i + 1}^{(k)} X_i^{(k)}$, $X_{i + 1}^{(k)} X_{i, i + 1}^{(k)}$, $X_{i, i + 1}^{(k)} X_{i + 1}^{(k)}$. We start with the following.

\begin{lemm}
Let $(k, i) \in [1, r] \times \bbZ$.
\begin{enumerate}

\item
If $F_{X_i^{(k)}, X_{i, i + 1}^{(k)}}^L \neq 0$, then either $L = X_i^{(k)} \oplus X_{i, i + 1}^{(k)}$ or, if $(r, m) = (1, 1)$, $L = X_i^{(k)}$.

\item
If $F_{X_{i, i + 1}^{(k)}, X_i^{(k)}}^L \neq 0$, then $L = X_i^{(k)} \oplus X_{i, i + 1}^{(k)}$.

\item
If $F_{X_{i + 1}^{(k)}, X_{i, i + 1}^{(k)}}^L \neq 0$, then $L = X_{i + 1}^{(k)} \oplus X_{i, i + 1}^{(k)}$.

\item
If $F_{X_{i, i + 1}^{(k)}, X_{i + 1}^{(k)}}^L \neq 0$, then either $L = X_{i + 1}^{(k)} \oplus X_{i, i + 1}^{(k)}$, or, if $(r, m) = (1, 1)$, $L = X_{i + 1}^{(k)}$.

\end{enumerate}
\end{lemm}

\begin{proof}
We only prove the first statement. The rest is proved analogously. We need to analyze the cones of the maps $f \colon X_{i, i + 1}^{(k)} \to \Sigma X_i^{(k)} = X_{i + \delta_{k, r} m}^{(k + 1)}$. If $(r, m) \neq (1, 1)$, the only possibility is the zero map (we use here that $m \neq 0$), hence $L = X_i^{(k)} \oplus X_{i, i + 1}^{(k)}$. If $(r, m) = (1, 1)$, then there are nonzero maps of the form $\lambda \alpha_{(i + 1, i + 1), (i, i + 1)}^{(k)} $, $\lambda \in \bbF^\times$, whose cone is $X_{i + \delta_{k, r} m}^{(k)}$ by the dual of Proposition~\ref{prop coneX}, thus we get $L = \Sigma^{-1} X_{i + \delta_{k, r} m}^{(k)} = X_i^{(k)}$ ($k - 1 = k$, since $r = 1$).
\end{proof}

According to the above lemma the next step is the following.

\begin{lemm}
Let $(k, i) \in [1, r] \times \bbZ$. Then
\begin{align*}
F_{X_i^{(k)}, X_{i, i + 1}^{(k)}}^{X_i^{(k)} \oplus X_{i, i + 1}^{(k)}} & = \lambda_0 \lambda_1, & F_{X_{i + 1}^{(k)}, X_{i, i + 1}^{(k)}}^{X_{i + 1}^{(k)} \oplus X_{i, i + 1}^{(k)}} & = \tfrac{1}{\frakq} \lambda_0 \lambda_{-1},
\\
F_{X_{i, i + 1}^{(k)}, X_i^{(k)}}^{X_i^{(k)} \oplus X_{i, i + 1}^{(k)}} & = \tfrac{1}{\frakq} \lambda_0 \lambda_{-1}, & F_{X_{i, i + 1}^{(k)}, X_{i + 1}^{(k)}}^{X_{i + 1}^{(k)} \oplus X_{i, i + 1}^{(k)}} & = \lambda_0 \lambda_1.
\end{align*}
Moreover, if $(r, m) = (1, 1)$, then
\[
F_{X_i^{(k)}, X_{i, i + 1}^{(k)}}^{X_i^{(k)}} = \tfrac{1}{\frakq} \qquad \text{and} \qquad F_{X_{i, i + 1}^{(k)}, X_{i + 1}^{(k)}}^{X_{i + 1}^{(k)}} = \tfrac{1}{\frakq}.
\]
\end{lemm}

\begin{proof}
We prove the second formula. Observe that according to Corollary~\ref{coro direct sum}
\[
F_{X_{i + 1}^{(k)}, X_{i, i + 1}^{(k)}}^{X_{i + 1}^{(k)} \oplus X_{i, i + 1}^{(k)}} = [X_{i + 1}^{(k)}, X_{i, i + 1}^{(k)}].
\]
Moreover, by Lemma~\ref{lemm mult}\eqref{point three}
\[
[X_{i + 1}^{(k)}, X_{i, i + 1}^{(k)}] = \tfrac{|\Hom (X_{i + 1}^{(k)}, X_{i + \delta_{k, r} m, i + \delta_{k, r} m + 1}^{(k + 1)})|}{|\Hom (X_{i + 1}^{(k)}, X_{i + \delta_{k, r} m}^{(k + 1)})|} \cdot [X_{i + 1}^{(k)}, X_i^{(k)}] \cdot [X_{i + 1}^{(k)}, X_{i + 1}^{(k)}] = \tfrac{1}{\frakq} \lambda_0 \lambda_{-1}.
\]
The first, third and fourth formulas are proved similarly.

The fifth and sixth formulas are proved by direct calculations. For example,
\begin{gather*}
|\Hom (X_{i, i + 1}^{(k)}, X_{i + 1}^{(k)})_{X_{i + 1}^{(k)}}| = \frakq - 1, \qquad |\Aut (X_{i, i + 1}^{(k)})| = \frakq (\frakq - 1),
\\
\{ X_{i, i + 1}^{(k)}, X_{i + 1}^{(k)} \} = \tfrac{1}{\frakq} \qquad \text{and} \qquad \{ X_{i, i + 1}^{(k)}, X_{i, i + 1}^{(k)} \} = \tfrac{1}{\frakq},
\end{gather*}
hence $F_{X_{i, i + 1}^{(k)}, X_{i + 1}^{(k)}}^{X_{i + 1}^{(k)}} = \tfrac{1}{\frakq}$.
\end{proof}

Consequently we get.

\begin{coro} \label{coro mult}
Let $(k, i) \in [1, r] \times \bbZ$.
\begin{enumerate}

\item
We have
\[
X_{i, i + 1}^{(k)} X_i^{(k)} = \tfrac{1}{\frakq} \lambda_0 \lambda_{-1} (X_i^{(k)} \oplus X_{i, i + 1}^{(k)})
\quad \text{and} \quad
X_{i + 1}^{(k)} X_{i, i + 1}^{(k)} = \tfrac{1}{\frakq} \lambda_0 \lambda_{-1} (X_{i + 1}^{(k)} \oplus X_{i, i + 1}^{(k)}).
\]

\item
If $(r, m) \neq (1, 1)$, then
\[
X_i^{(k)} X_{i, i + 1}^{(k)} = \lambda_0 \lambda_1 (X_i^{(k)} \oplus X_{i, i + 1}^{(k)}),
\quad \text{and} \quad
X_{i, i + 1}^{(k)} X_{i + 1}^{(k)} = \lambda_0 \lambda_1 (X_{i + 1}^{(k)} \oplus X_{i, i + 1}^{(k)}).
\]

\item
If $(r, m) = (1, 1)$, then
\begin{align*}
X_i^{(k)} X_{i, i + 1}^{(k)} & = \lambda_0 \lambda_1 (X_i^{(k)} \oplus X_{i, i + 1}^{(k)}) + \tfrac{1}{\frakq} X_i^{(k)},
\\
\intertext{and}
X_{i, i + 1}^{(k)} X_{i + 1}^{(k)} & = \lambda_0 \lambda_1 (X_{i + 1}^{(k)} \oplus X_{i, i + 1}^{(k)}) + \tfrac{1}{\frakq} X_{i + 1}^{(k)}.
\end{align*}

\end{enumerate}
\end{coro}

As a result we get the following relations.

\begin{prop} \label{prop rel2}
Let $(k, i) \in [1, r] \times \bbZ$.
\begin{enumerate}

\item
If $(r, m) \neq (1, 1)$, then
\begin{gather*}
\tfrac{1}{\lambda_1^2} X_i^{(k)} X_i^{(k)} X_{i + 1}^{(k)} - \tfrac{1}{\lambda_1 \lambda_{-1}} (\frakq + 1) X_i^{(k)} X_{i + 1}^{(k)} X_i^{(k)} + \tfrac{1}{\lambda_{-1}^2} \frakq X_{i + 1}^{(k)} X_i^{(k)} X_{i + 1}^{(k)} = 0
\\
\intertext{and}
\tfrac{1}{\lambda_1^2} X_i^{(k)} X_{i + 1}^{(k)} X_{i + 1}^{(k)} - \tfrac{1}{\lambda_1 \lambda_{-1}} (\frakq + 1) X_{i + 1}^{(k)} X_i^{(k)} X_{i + 1}^{(k)} + \tfrac{1}{\lambda_{-1}^2} \frakq X_{i + 1}^{(k)} X_{i + 1}^{(k)} X_i^{(k)} = 0.
\end{gather*}

\item
If $(r, m) = (1, 1)$, then
\begin{multline*}
\tfrac{1}{\lambda_1^2} X_i^{(k)} X_i^{(k)} X_{i + 1}^{(k)} - \tfrac{1}{\lambda_1 \lambda_{-1}} (\frakq + 1) X_i^{(k)} X_{i + 1}^{(k)} X_i^{(k)} + \tfrac{1}{\lambda_{-1}^2} \frakq X_{i + 1}^{(k)} X_i^{(k)} X_{i + 1}^{(k)}
\\
= \frakq (\frakq + 1) X_i^{(k)}
\end{multline*}
and
\begin{multline*}
\tfrac{1}{\lambda_1^2} X_i^{(k)} X_{i + 1}^{(k)} X_{i + 1}^{(k)} - \tfrac{1}{\lambda_1 \lambda_{-1}} (\frakq + 1) X_{i + 1}^{(k)} X_i^{(k)} X_{i + 1}^{(k)} + \tfrac{1}{\lambda_{-1}^2} \frakq X_{i + 1}^{(k)} X_{i + 1}^{(k)} X_i^{(k)}
\\
= \frakq (\frakq + 1) X_{i + 1}^{(k)}.
\end{multline*}

\end{enumerate}
\end{prop}

\begin{proof}
We prove the first formula. Observe that
\begin{align*}
& \tfrac{1}{\lambda_1^2} X_i^{(k)} X_i^{(k)} X_{i + 1}^{(k)} - \tfrac{1}{\lambda_1 \lambda_{-1}} (\frakq + 1) X_i^{(k)} X_{i + 1}^{(k)} X_i^{(k)} + \tfrac{1}{\lambda_{-1}^2} \frakq X_{i + 1}^{(k)} X_i^{(k)} X_i^{(k)}
\\
& \qquad = \tfrac{1}{\lambda_1} X_i^{(k)} (\tfrac{1}{\lambda_1} X_i^{(k)} X_{i + 1}^{(k)} - \tfrac{1}{\lambda_{-1}} X_{i + 1}^{(k)} X_i^{(k)})
\\
& \qquad \qquad - \tfrac{1}{\lambda_{-1}} \frakq (\tfrac{1}{\lambda_1} X_i^{(k)} X_{i + 1}^{(k)} - \tfrac{1}{\lambda_{-1}} X_{i + 1}^{(k)} X_i^{(k)}) X_i^{(k)}
\\
& \qquad = \tfrac{1}{\lambda_1} X_i^{(k)} X_{i, i + 1}^{(k)} - \tfrac{1}{\lambda_{-1}} \frakq X_{i, i + 1}^{(k)} X_i^{(k)},
\end{align*}
where the latter equality follows from Lemma~\ref{lemm multiply}. Now the formula follows from Corollary~\ref{coro mult}. The remaining formulas are proved analogously. For the third and forth formulas one also has to use that if $(r, m) = (1, 1)$, then $\lambda_1 = \tfrac{1}{\frakq^2}$ and $\lambda_{-1} = 1$.
\end{proof}

\subsection{Generators}

Let $\cA'$ be the algebra generated by $x_i^{(k)}$, $k \in [1, r]$, $i \in \bbZ$. Observe that, for $\Phi \in \cA'$, we have the degree vector $\bdeg \Phi \in \bbN^{([1, r] \times \bbZ)}$ defined by
\[
(\bdeg \Phi) (i, k) := \deg_{x_i^{(k)}} \Phi.
\]
We need the following.

\begin{lemm} \label{lemm poly}
There exist polynomials $\Phi_{i, j}^{(k)} \in \cA'$ such that $\bdeg \Phi_{i, j}^{(k)} \leq \bdim X_{i, j}^{(k)}$ and
\[
\Phi_{i, j}^{(k)} ((X_s^{(l)})) = X_{i, j}^{(k)},
\]
for all $k$ and $i \leq j$.
\end{lemm}

\begin{proof}
We prove the claim by induction of $j - i$. If $j = i$, the claim is obvious ($\Phi_{i, i}^{(k)} = x_i^{(k)}$). Similarly, if $j = i + 1$, then by Lemma~\ref{lemm multiply} we may take
\[
\Phi_{i, i + 1}^{(k)} := \tfrac{1}{\lambda_1} x_i^{(k)} x_{i + 1}^{(k)} - \tfrac{1}{\lambda_{-1}} x_{i + 1}^{(k)} x_i^{(k)} - \delta_{r m, 1} \tfrac{\frakq}{\frakq - 1}.
\]

Now assume $j > i + 1$. Similarly as in the previous subsection one shows there exist $v_0, v_1, v_2, u_0 \in \bbQ$ such that
\begin{gather*}
X_i^{(k)} X_{i + 1, j}^{(k)} = v_0 (X_i^{(k)} \oplus X_{i + 1, j}^{(k)}) + v_1 X_{i, j}^{(k)} + v_2 X_{i + 1, j - 1}^{(k)}
\\
\intertext{and}
X_{i + 1, j}^{(k)} X_i^{(k)} = u_0 (X_i^{(k)} \oplus X_{i + 1, j}^{(k)}),
\end{gather*}
where $v_0, v_1, u_0 \neq 0$ (and $v_2 \neq 0$ if and only if $r = 1$ and $m = j - i$). By the inductive hypothesis
\[
X_{i + 1, j}^{(k)} = \Phi_{i + 1, j}^{(k)} ((X_s^{(l)})) \qquad \text{and} \qquad X_{i + 1, j - 1}^{(k)} = \Phi_{i + 1, j - 1}^{(k)} ((X_s^{(l)}))
\]
for $\Phi_{i + 1, j}^{(k)}, \Phi_{i + 1, j - 1}^{(k)} \in \cA'$ such that $\bdeg \Phi_{i + 1, j}^{(k)} \leq \bdim X_{i + 1, j}^{(k)}$ and $\bdeg \Phi_{i + 1, j - 1}^{(k)} \leq \bdim \Phi_{i + 1, j - 1}^{(k)}$. Now, if we take
\[
\Phi_{i, j}^{(k)} := \tfrac{1}{v_1} x_i^{(k)} \Phi_{i + 1, j}^{(k)} - \tfrac{v_0}{v_1 u_0} \Phi_{i + 1, j}^{(k)} x_i^{(k)} - \tfrac{v_2}{v_1} \Phi_{i + 1, j - 1}^{(k)},
\]
the claim follows.
\end{proof}

Using Lemma~\ref{lemm poly} and Proposition~\ref{prop poly} we get the following.

\begin{prop} \label{prop epi}
For each $M \in \cX$ there exits a polynomial $\Phi_M \in \cA'$ such that $\bdeg \Phi_M \leq \bdim M$ and
\[
\Phi_M ((X_i^{(k)})) = M. \eqno \qed
\]
\end{prop}

Consequently, we we obtain the following.

\begin{coro} \label{coro gen}
The algebra $\cH'$ is generated by $X_i^{(k)}$, $k \in [1, r]$, $i \in \bbZ$. \qed
\end{coro}

\subsection{Presentation} \label{sub pres X}

This subsection is devoted to the proof of the following.

\begin{theo} \label{theo main}
Let $\cA'$ be the algebra generated by $x_i^{(k)}$, $k \in [1, r]$, $i \in \bbZ$, and $\cI'$ be the ideal in $\cA'$ generated by the following elements:
\begin{enumerate}

\item
$\tfrac{1}{\lambda_{j - i}^{(k, l)}} x_i^{(k)} x_j^{(l)} - \tfrac{1}{\lambda_{i - j}^{(l, k)}} x_j^{(l)} x_i^{(k)}$, $(k, i), (l, j) \in [1, r] \times \bbZ$, $(l, j) \neq (k, i), (k, i + 1), (k, i - 1), (k + 1, i + \delta_{k, r} m), (k - 1, i - \delta_{k, 1} m)$,

\addtocounter{enumi}{-1}
\renewcommand{\theenumi}{\arabic{enumi}'}

\item
$\tfrac{1}{\lambda_{\delta_{k, r} m}^{(k, k + 1)}} x_i^{(k)} x_{i + \delta_{k, r} m}^{(k + 1)} - \tfrac{1}{\lambda_{- \delta_{k, r} m}^{(k + 1, k)}} x_{i + \delta_{k, r} m}^{(k + 1)} x_i^{(k)} - \tfrac{\frakq}{\frakq - 1}$, for $(k, i) \in [1, r] \times \bbZ$, if $(r, m) \neq (1, 1)$,

\renewcommand{\theenumi}{\arabic{enumi}}

\item
$\tfrac{1}{\lambda_1^2} x_i^{(k)} x_i^{(k)} x_{i + 1}^{(k)} - \tfrac{1}{\lambda_1 \lambda_{-1}} (\frakq + 1) x_i^{(k)} x_{i + 1}^{(k)} x_i^{(k)} + \tfrac{1}{\lambda_{-1}^2} \frakq x_{i + 1}^{(k)} x_i^{(k)} x_i^{(k)} - \delta_{r m, 1} \frakq (\frakq + 1) x_i^{(k)}$, for $(k, i) \in [1, r] \times \bbZ$,

\item
$\tfrac{1}{\lambda_1^2} x_i^{(k)} x_{i + 1}^{(k)} x_{i + 1}^{(k)} - \tfrac{1}{\lambda_1 \lambda_{-1}} (\frakq + 1) x_{i + 1}^{(k)} x_i^{(k)} x_{i + 1}^{(k)} + \tfrac{1}{\lambda_{-1}^2} \frakq x_{i + 1}^{(k)} x_{i + 1}^{(k)} x_i^{(k)} - \delta_{r m, 1} \frakq (\frakq + 1) x_{i + 1}^{(k)}$, for $(k, i) \in [1, r] \times \bbZ$.

\end{enumerate}
If $\cB' := \cA' / \cI'$, then the homomorphism $\xi \colon \cA' \to \cH'$ given by $\xi (x_i^{(k)}) := X_i^{(k)}$, $k \in [1, r]$, $i \in \bbZ$, induces an isomorphism $\ol{\xi} \colon \cB' \to \cH'$.
\end{theo}

Propositions~\ref{prop rel1} and~\ref{prop rel2} imply that $\ol{\xi}$ is a well-defined homomorphism. Next, Corollary~\ref{coro gen} implies that $\ol{\xi}$ is an epimorphism. It remains to show that $\ol{\xi}$ is a monomorphism.

Let $\bd \in \bbN^{([1, r] \times \bbZ)}$. Denote by $\cA_\bd'$ the subspace of $\cA'$ consisting of $\Phi \in \cA'$ such that $\bdeg \Phi \leq \bd$. Let $\cB_\bd'$ be the image of $\cA_\bd'$ under the canonical epimorphism $\cA' \to \cB'$. Finally, let $\cH_\bd'$ be the subspace of $\cH'$ spanned by $M$ such that $\bdim M \leq \bd$. Corollary~\ref{coro dim vect X} implies that $\xi$ induces a map $\ol{\xi}_\bd \colon \cB_\bd' \to \cH_\bd'$.

Since $\cB' = \bigcup_\bd \cB_\bd'$, in order to show that $\ol{\xi}$ is a monomorphism, it suffices to show that $\ol{\xi}_\bd$ is a monomorphism, for each $\bd$. Note that $\ol{\xi}_\bd$ is an epimorphism by Proposition~\ref{prop epi}. Since $\dim_\bbQ \cH_\bd' < \infty$, we finish the proof if we show $\dim_\bbQ \cB_\bd' \leq \dim_\bbQ \cH_\bd'$.

For $k \in [1, r]$ and $i \leq j$, let
\[
x_{i, j}^{(k)} := x_i^{(k)} \cdots x_j^{(k)}.
\]
We also put $x_{i, i - 1}^{(k)} := 1$. Observe that $\cH_\bd'$ has a basis formed by the objects
\[
X_{i_1, j_1}^{(k_1)} \oplus X_{i_2, j_2}^{(k_2)} \oplus \cdots \oplus X_{i_n, j_n}^{(k_n)},
\]
$n \in \bbN$, such that
\[
\bdim (X_{i_1, j_1}^{(k_1)} \oplus X_{i_2, j_2}^{(k_2)} \oplus \cdots \oplus X_{i_n, j_n}^{(k_n)}) \leq \bd,
\]
$k_1 \geq k_2 \geq \cdots \geq k_n$, if $k_t = k_{t + 1}$, then $i_t \geq i_{t + 1}$, and if $k_t = k_{t + 1}$ and $i_t = i_{t + 1}$, then $j_t \geq j_{t + 1}$. Thus in order to show that $\dim_\bbQ \cB_\bd' \leq \dim_\bbQ \cH_\bd'$, it suffices to prove the following.

\begin{prop} \label{prop span}
Let $\bd \in \bbN^{(\bbZ)}$. Then $\cB_\bd'$ is spanned by the images of the monomials
\[
x_{i_1, j_1}^{(k_1)} x_{i_2, j_2}^{(k_2)} \cdots x_{i_n, j_n}^{(k_n)} \in \cA_\bd',
\]
$n \in \bbN$, such that $k_1 \geq k_2 \geq \cdots \geq k_n$, if $k_t = k_{t + 1}$, then $i_t \geq i_{t + 1}$, and if $k_t = k_{t + 1}$ and $i_t = i_{t + 1}$, then $j_t \geq j_{t + 1}$.
\end{prop}

By convention the product $x_{i_1, j_1}^{(k_1)} x_{i_2, j_2}^{(k_2)} \cdots x_{i_n, j_n}^{(k_n)}$ equals $1$ if $n = 0$. We fix $\bd$ for the rest of the subsection.

We introduce the following notation. If $\Phi, \Phi_1, \ldots, \Phi_l \in \cA_\bd'$ are monomials of the same degree vector, then we write
\[
\Phi \equiv \Phi_1 + \cdots + \Phi_l,
\]
if there exist $v_1, \ldots, v_l \in \bbQ$ and $\Psi \in \cA$ such that $\bdeg \Psi < \bdeg \Phi$ and
\[
\Phi + \cI' = v_1 \Phi_1 + \cdots + v_l \Phi_l + \Psi + \cI'.
\]
Note that $\equiv$ behaves like a congruence with respect to the multiplication, i.e.\ if $\Phi \equiv \Phi_1 + \cdots + \Phi_l$, and $\Psi'$ and $\Psi''$ are monomials, then
\[
\Psi' \Phi \Psi'' \equiv \Psi' \Phi_1 \Psi'' + \cdots + \Psi' \Phi_l \Psi''.
\]
Moreover $\equiv$ is transitive in the following sense: if $\Phi \equiv \Phi_1 + \cdots + \Phi_l$ and, for each $1 \leq i \leq l$, $\Phi_i \equiv \Phi_{i, 1} + \cdots + \Phi_{i, l_i}$, then
\[
\Phi \equiv \sum_{i = 1}^l \sum_{j = 1}^{l_i} \Phi_{i, j}.
\]

Using this notation we may reformulate the relations from Theorem~\ref{theo main} in the following way.

\begin{lemm} \label{lemm exchange}
Let $k \in [1, r]$ and $i \in \bbZ$. Then we have the following:
\begin{enumerate}

\item \label{exchange one}
$x_i^{(k)} x_j^{(l)} \equiv x_j^{(l)} x_i^{(k)}$, if $(l, j) \neq (k, i \pm 1)$;

\item \label{exchange two}
$x_i^{(k)} x_i^{(k)} x_{i + 1}^{(k)} \equiv x_i^{(k)} x_{i + 1}^{(k)} x_i^{(k)} + x_{i + 1}^{(k)} x_i^{(k)} x_i^{(k)}$;

\addtocounter{enumi}{-1}
\renewcommand{\theenumi}{\arabic{enumi}'}

\item \label{exchange two prim}
$x_i^{(k)} x_{i + 1}^{(k)} x_i^{(k)} \equiv x_i^{(k)} x_i^{(k)} x_{i + 1}^{(k)} + x_{i + 1}^{(k)} x_i^{(k)} x_i^{(k)}$;

\renewcommand{\theenumi}{\arabic{enumi}}

\item \label{exchange three}
$x_i^{(k)} x_{i + 1}^{(k)} x_{i + 1}^{(k)} \equiv x_{i + 1}^{(k)} x_i^{(k)} x_{i + 1}^{(k)} + x_{i + 1}^{(k)} x_{i + 1}^{(k)} x_i^{(k)}$;

\addtocounter{enumi}{-1}
\renewcommand{\theenumi}{\arabic{enumi}'}

\item \label{exchange three prim}
$x_{i + 1}^{(k)} x_i^{(k)} x_{i + 1}^{(k)} \equiv x_i^{(k)} x_{i + 1}^{(k)} x_{i + 1}^{(k)} + x_{i + 1}^{(k)} x_{i + 1}^{(k)} x_i^{(k)}$. \qed

\end{enumerate}
\end{lemm}

Note that~\eqref{exchange two prim} and~\eqref{exchange three prim} follow from the same relations as~\eqref{exchange two} and~\eqref{exchange three}, respectively.

For each degree vector $\bd$ we introduce the inverse lexicographic order on the set of monomials of the degree vector $\bd$, i.e. if
\[
\Phi_1 = x_{i_1}^{(k_1)} \cdots x_{i_n}^{(k_n)} \qquad \text{and} \qquad \Phi_2 = x_{j_1}^{(l_1)} \cdots x_{j_n}^{(l_n)}
\]
are monomials of degree $\bd$, then $\Phi_1 < \Phi_2$ if and only if there exists $1 \leq s \leq n$ such that $(k_1, i_1) = (l_1, j_1)$, \dots, $(k_{s - 1}, i_{s - 1}) = (l_{s - 1}, j_{s - 1})$, and either $k_s > l_s$ or $k_s = l_s$ and $i_s > j_s$. Again this order behaves well with respect to the multiplication, i.e.\ if $\Phi_1 < \Phi_2$, then $\Psi' \Phi_1 \Psi'' < \Psi' \Phi_2 \Psi''$, for any monomials $\Psi'$ and $\Psi''$.

Obviously, $\cB_\bd'$ is spanned by the images of the monomials $x_{i_1, j_1}^{(k_1)} x_{i_2, j_2}^{(k_2)} \cdots x_{i_n, j_n}^{(k_n)}$ of degree vector at most $\bd$. Moreover, if $i_2 = j_1 + 1$, then $x_{i_1, j_1}^{(k)} x_{i_2, j_2}^{(k)} = x_{i_1, j_2}^{(k)}$. Consequently, if order to show Proposition~\ref{prop span} (and hence Theorem~\ref{theo main}), it is enough to prove the following.

\begin{prop} \label{prop reduction}
Assume $k_1, k_2 \in [1, r]$, $i_1 \leq j_1$ and $i_2 \leq j_2$. If either $k_1 < k_2$, or $k_1 = k_2$ and $i_1 < i_2$ or $k_1 = k_2$, $i_1 = i_2$ and $j_1 < j_2$, then either $k_1 = k_2$ and $i_2 = j_1 + 1$ or there exists monomials $\Phi_1$, \ldots, $\Phi_n$ of degree vector $\bdeg (x_{i_1, j_1}^{(k_1)} x_{i_2, j_2}^{(k_2)})$ such that $\Phi_s < x_{i_1, j_1}^{(k_1)} x_{i_2, j_2}^{(k_2)}$, for each $1 \leq s \leq n$, and
\[
x_{i_1, j_1}^{(k_1)} x_{i_2, j_2}^{(k_2)} \equiv \Phi_1 + \cdots + \Phi_n.
\]
\end{prop}

The rest of the section is devoted to the proof of Proposition~\ref{prop reduction}. We have some cases to consider.

\vspace{1ex}

\textbf{Case 0$^\circ$}: $k_1 < k_2$.

In this case by an iterated use of Lemma~\ref{lemm exchange}\eqref{exchange one} we get
\[
x_{i_1, j_1}^{(k_1)} x_{i_2, j_2}^{(k_2)} \equiv x_{i_2, j_2}^{(k_2)} x_{i_1, j_1}^{(k_1)}.
\]
Moreover, $x_{i_2, j_2}^{(k_2)} x_{i_1, j_1}^{(k_1)} < x_{i_1, j_1}^{(k_1)} x_{i_2, j_2}^{(k_2)}$.

For the rest of the proof we assume $k_1 = k_2$ and denote this common value by $k$.

\vspace{1ex}

\textbf{Case 1$^\circ$}: $j_1 < i_2$.

In this case we may assume $i_2 > j_1 + 1$, hence we get
\[
x_{i_1, j_1}^{(k)} x_{i_2, j_2}^{(k)} \equiv x_{i_2, j_2}^{(k)} x_{i_1, j_1}^{(k)},
\]
again by an iterated use of Lemma~\ref{lemm exchange}\eqref{exchange one}. Moreover, $x_{i_2, j_2}^{(k)} x_{i_1, j_1}^{(k)} < x_{i_1, j_1}^{(k)} x_{i_2, j_2}^{(k)}$, since $i_2 > j_1 \geq i_1$.

\vspace{1ex}

\textbf{Case 2$^\circ$}: $j_1 = i_2$.

In this case we have either $i_1 < i_2 = j_1$ or $i_1 = i_2$, hence $i_2 = j_1 < j_2$. Assume first $i_1 < i_2 = j_1$. Then using Lemma~\ref{lemm exchange}\eqref{exchange three}, we have
\begin{multline*}
x_{i_1, j_1}^{(k)} x_{i_2, j_2}^{(k)} = x_{i_1, i_2 - 2}^{(k)} x_{i_2 - 1}^{(k)} x_{i_2}^{(k)} x_{i_2}^{(k)} x_{i_2 + 1, j_2}^{(k)}
\\
\equiv x_{i_1, i_2 - 2}^{(k)} x_{i_2}^{(k)} x_{i_2 - 1}^{(k)} x_{i_2}^{(k)} x_{i_2 + 1, j_2}^{(k)} + x_{i_1, i_2 - 2}^{(k)} x_{i_2}^{(k)} x_{i_2}^{(k)} x_{i_2 - 1}^{(k)} x_{i_2 + 1, j_2}^{(k)},
\end{multline*}
and the claim follows. Similarly, if $i_2 < j_2$, then
\[
x_{i_1, j_1}^{(k)} x_{i_2, j_2}^{(k)} \equiv x_{i_1, i_2 - 1}^{(k)} x_{i_2}^{(k)} x_{i_2 + 1}^{(k)} x_{i_2}^{(k)} x_{i_2 + 2, j_2}^{(k)} + x_{i_1, i_2 - 1}^{(k)} x_{i_2 + 1}^{(k)} x_{i_2}^{(k)} x_{i_2}^{(k)} x_{i_2 + 2, j_2}^{(k)}
\]
(here we use Lemma~\ref{lemm exchange}\eqref{exchange two}).

\vspace{1ex}

\textbf{Case 3$^\circ$}: $i_1 < i_2 < j_1$.

First, by an iterated use of Lemma~\ref{lemm exchange}\eqref{exchange one} we get
\[
x_{i_1, j_1}^{(k)} x_{i_2, j_2}^{(k)} \equiv x_{i_1, i_2 - 1}^{(k)} x_{i_2}^{(k)} x_{i_2 + 1}^{(k)} x_{i_2}^{(k)} x_{i_2 + 2, j_1}^{(k)} x_{i_2 + 1, j_2}^{(k)}.
\]
Next we use Lemma~\ref{lemm exchange}\eqref{exchange two prim} and get
\[
x_{i_1, j_1}^{(k)} x_{i_2, j_2}^{(k)} \equiv x_{i_1, i_2 - 1}^{(k)} x_{i_2 + 1}^{(k)} x_{i_2}^{(k)} x_{i_2}^{(k)} x_{i_2 + 2, j_1}^{(k)} x_{i_2 + 1, j_2}^{(k)} + x_{i_1, i_2}^{(k)} x_{i_2, j_1}^{(k)} x_{i_2 + 1, j_2}^{(k)}
\]
Finally we use Lemma~\ref{lemm exchange}\eqref{exchange two} (for $i = i_2 - 1$) and get
\begin{multline*}
x_{i_1, i_2}^{(k)} x_{i_2, j_1}^{(k)} x_{i_2 + 1, j_2}^{(k)} = x_{i_1, i_2 - 2}^{(k)} x_{i_2 - 1}^{(k)} x_{i_2}^{(k)} x_{i_2}^{(k)} x_{i_2 + 1, j_1}^{(k)} x_{i_2 + 1, j_2}^{(k)}
\\
\equiv x_{i_1, i_2 - 2}^{(k)} x_{i_2}^{(k)} x_{i_2 - 1}^{(k)} x_{i_2, j_1}^{(k)} x_{i_2 + 1, j_2}^{(k)} + x_{i_1, i_2 - 2}^{(k)} x_{i_2}^{(k)} x_{i_2}^{(k)} x_{i_2 - 1}^{(k)} x_{i_2 + 1, j_1}^{(k)} x_{i_2 + 1, j_2}^{(k)}.
\end{multline*}

\vspace{1ex}

\textbf{Case 4$^\circ$}: $i_1 = i_2$.

In this case $j_2 > j_1$. We start with the following lemma.

\begin{lemm} \label{lemm exchangeprim}
If $i_1 < s < j_2$, then there exist monomials $\Phi_1$, \ldots, $\Phi_n$ of degree vector $\bdeg (x_s^{(k)} x_{i_1, j_2}^{(k)})$ such that $\Phi_1, \ldots, \Phi_n < x_s^{(k)} x_{i_1, j_2}^{(k)}$ and
\[
x_s^{(k)} x_{i_1, j_2}^{(k)} \equiv x_{i_1, j_2}^{(k)} x_s^{(k)} + \Phi_1 + \cdots + \Phi_l.
\]
\end{lemm}

\begin{proof}
First, by an iterated use of Lemma~\ref{lemm exchange}\eqref{exchange one} we get
\[
x_s^{(k)} x_{i_1, j_2}^{(k)} \equiv x_{i_1, s - 2}^{(k)} x_s^{(k)} x_{s - 1, j_2}^{(k)}.
\]
Using Lemma~\ref{lemm exchange}\eqref{exchange three prim} we get
\[
x_{i_1, s - 2}^{(k)} x_s^{(k)} x_{s - 1, j_2}^{(k)} \equiv x_{i_1, s}^{(k)} x_{s, j_2}^{(k)} + x_{i_1, s - 2}^{(k)} x_s^{(k)} x_s^{(k)} x_{s - 1}^{(k)} x_{s + 1, j_2}^{(k)}.
\]
Now by an iterated use of Lemma~\ref{lemm exchange}\eqref{exchange one} we have
\[
x_{i_1, s - 2}^{(k)} x_s^{(k)} x_s^{(k)} x_{s - 1}^{(k)} x_{s + 1, j_2}^{(k)} \equiv x_s^{(k)} x_s^{(k)} x_{i_1, s - 2}^{(k)} x_{s - 1}^{(k)} x_{s + 1, j_2}^{(k)} < x_s^{(k)} x_{i_1, j_2}^{(k)}.
\]
On the other hand, using Lemma~\ref{lemm exchange}\eqref{exchange two} we obtain
\[
x_{i_1, s}^{(k)} x_{s, j_2}^{(k)} \equiv x_{i_1, s + 1}^{(k)} x_s^{(k)} x_{s + 2, j_2}^{(k)} + x_{i_1, s - 1}^{(k)} x_{s + 1}^{(k)} x_s^{(k)} x_s^{(k)} x_{s + 2, j_2}^{(k)}.
\]
Using iteratively Lemma~\ref{lemm exchange}\eqref{exchange one} again we get
\begin{gather*}
x_{i_1, s + 1}^{(k)} x_s^{(k)} x_{s + 2, j_2}^{(k)} \equiv x_{i_1, j_2}^{(k)} x_s^{(k)}
\\
\intertext{and}
x_{i_1, s - 1}^{(k)} x_{s + 1}^{(k)} x_s^{(k)} x_s^{(k)} x_{s + 2, j_2}^{(k)} \equiv x_{s + 1}^{(k)} x_{i_1, s}^{(k)} x_s^{(k)} x_{s + 2, j_2}^{(k)} < x_s^{(k)} x_{i_1, j_2}^{(k)},
\end{gather*}
which finishes the proof.
\end{proof}

Using iteratively Lemma~\ref{lemm exchangeprim} we obtain monomials $\Phi_1$, \ldots, $\Phi_l$ of degree vector $\bdeg (x_{i_1, j_1}^{(k)} x_{i_1, j_2}^{(k)})$ such that $\Phi_1, \ldots, \Phi_l < x_{i_1, j_1}^{(k)} x_{i_1, j_2}^{(k)}$ and
\[
x_{i_1, j_1}^{(k)} x_{i_1, j_2}^{(k)} \equiv x_{i_1}^{(k)} x_{i_1, j_2}^{(k)} x_{i_1 + 1, j_1}^{(k)} + \Phi_1 + \cdots + \Phi_l.
\]
Using Lemma~\ref{lemm exchange}\eqref{exchange two} we get
\[
x_{i_1}^{(k)} x_{i_1, j_2}^{(k)} x_{i_1 + 1, j_1}^{(k)} \equiv x_{i_1}^{(k)} x_{i_1 + 1}^{(k)} x_{i_1}^{(k)} x_{i_1 + 2, j_2}^{(k)} x_{i_1 + 1, j_1}^{(k)} + x_{i_1 + 1}^{(k)} x_{i_1}^{(k)} x_{i_1}^{(k)} x_{i_1 + 2, j_2}^{(k)} x_{i_1 + 1, j_1}^{(k)}.
\]
Finally, by an iterated use of Lemma~\ref{lemm exchange}\eqref{exchange one} we have
\[
x_{i_1}^{(k)} x_{i_1 + 1}^{(k)} x_{i_1}^{(k)} x_{i_1 + 2, j_2}^{(k)} x_{i_1 + 1, j_1}^{(k)} \equiv x_{i_1, j_2}^{(k)} x_{i_1 j_1}^{(k)}.
\]
This finishes the proof of Proposition~\ref{prop reduction}, and hence of Theorem~\ref{theo main}.

\section{Results. II} \label{sect res2}

In this section we describe the Hall algebra $\cH := \cH (\cC)$ associated to the category $\cC$ introduced in Section~\ref{sect category}.

\subsection{Relations}
We introduce the following notation:
\[
\mu_i^{(k, l)} := [X_0^{(k)}, Z_i^{(l)}], \qquad \nu_i^{(k, l)} := [Z_0^{(k)}, X_i^{(l)}], \qquad \kappa_i^{(k, l)} := [Z_0^{(k)}, Z_i^{(l)}].
\]
Moreover,
\begin{gather*}
\mu_i := \mu_i^{(1, 1)}, \qquad \mu_i' := \mu_{i - \delta_{1, r} m + 1}^{(2, 1)}, \qquad \nu_i := \nu_i^{(1, 1)}, \qquad \nu_i' := \nu_{i + \delta_{1, r} m - 1}^{(1, 2)},
\\
\intertext{and}
\kappa_i := \kappa_i^{(1, 1)}.
\end{gather*}
Similarly as for $\lambda_i^{(k, l)}$'s actual formulas depend on $r$ and $m$ and are cumbersome. We encourage the reader to write formulas for $\mu_i$, $\mu_i'$, $\nu_i$, $\nu_i'$ and $\kappa_i$. Note that
\[
\mu_{j - i}^{(k, l)} = [X_i^{(k)}, Z_j^{(l)}], \qquad \nu_{j - i}^{(k, l)} = [Z_i^{(k)}, X_j^{(l)}], \qquad \kappa_{j - i}^{(k, l)} = [Z_j^{(k)}, Z_i^{(l)}].
\]
Moreover,
\begin{gather*}
\mu_i = \mu_i^{(k, k)}, \qquad \mu_i' = \mu_{i - \delta_{k, r} m + 1}^{(k + 1, k)}, \qquad \nu_i = \nu_i^{(k, k)}, \qquad \nu_i' = \nu_{i + \delta_{k, r} m - 1}^{(k, k + 1)},
\\
\intertext{and}
\kappa_i = \kappa_i^{(k, k)}.
\end{gather*}

We calculate first the products involving $Z_i^{(k)}$'s. Recall that $X_{i, i - 1}^{(k)}$ denotes the zero object.

\begin{lemm} \label{lemm multZik}
Let $(k, i) \in [1, r] \times \bbZ$.
\begin{enumerate}



\item
If $(l, j) \neq (k, i)$ and either $l \neq k + 1$ or $l = k + 1$ and $j > i + \delta_{k, r} m$, then
\[
Z_i^{(k)} Z_j^{(l)} = \kappa_{j - i}^{(k, l)} (Z_i^{(k)} \oplus Z_j^{(l)}).
\]

\item
If $j \leq i + \delta_{k, r} m$ and either $r > 1$ or $r = 1$ and $j \neq i$, then
\[
Z_i^{(k)} Z_j^{(k + 1)} = \kappa_{j - i}^{(k, l)} (Z_i^{(k)} \oplus Z_j^{(k + 1)}) + X_{j, i + \delta_{k, r} m - 1}^{(k + 1)}.
\]

\item
If $(l, j) \neq (k, i - 1)$, then
\[
X_j^{(l)} Z_i^{(k)} = \mu_{i - j}^{(l, k)} (X_j^{(l)} \oplus Z_i^{(k)}).
\]
Moreover,
\[
X_{i - 1}^{(k)} Z_i^{(k)} = \mu_1 (X_{i - 1}^{(k)} \oplus Z_i^{(k)}) + \mu_1 Z_{i - 1}^{(k)}.
\]

\item
If $(l, j) \neq (k + 1, i + \delta_{k, r} m)$, then
\[
Z_i^{(k)} X_j^{(l)} = \nu_{j - i}^{(k, l)} (X_j^{(l)} \oplus Z_i^{(k)}).
\]
Moreover,
\[
Z_i^{(k)} X_{i + \delta_{k, r} m}^{(k + 1)} = \nu_1' (X_{i + \delta_{k, r} m}^{(k + 1)} \oplus Z_i^{(k)}) + \nu_1' Z_{i + 1}^{(k)}.
\]

\end{enumerate}

\end{lemm}

\begin{proof}
One has to mimic arguments from the proofs of Lemma~\ref{lemm possL} and~\ref{lemm FXX}. We leave details to the reader.
\end{proof}

Using the above we get the relations.

\begin{prop} \label{prop rel3}
Let $k \in [1, r]$.
\begin{enumerate}

\item
If $l \neq k - 1, k, k + 1$, then
\[
\tfrac{1}{\kappa_0^{(k, l)}} Z_0^{(k)} Z_0^{(l)} - \tfrac{1}{\kappa_0^{(l, k)}} Z_0^{(l)} Z_0^{(k)} = 0.
\]

\addtocounter{enumi}{-1}
\renewcommand{\theenumi}{\arabic{enumi}'}

\item
If $r \geq 3$, then
\[
\tfrac{1}{\kappa_0^{(k, k + 1)}} Z_0^{(k)} Z_0^{(k + 1)} - \tfrac{1}{\kappa_0^{(k + 1, k)}} Z_0^{(k + 1)} Z_0^{(k)} = \tfrac{1}{\kappa_0^{(k, k + 1)}} X_{0, \delta_{k, r} m - 1}^{(k + 1)}.
\]

\addtocounter{enumi}{-1}
\renewcommand{\theenumi}{\arabic{enumi}''}

\item
If $r = 2$, then
\[
\tfrac{1}{\kappa_0^{(1, 2)}} Z_0^{(1)} Z_0^{(2)} - \tfrac{1}{\kappa_0^{(2, 1)}} Z_0^{(2)} Z_0^{(1)} = \tfrac{1}{\kappa_0^{(1, 2)}} X_{0, -1}^{(2)} - \tfrac{1}{\kappa_0^{(2, 1)}} X_{0, m - 1}^{(1)}.
\]

\renewcommand{\theenumi}{\arabic{enumi}}

\item
If $(l, j) \neq (k, -1), (k + 1, \delta_{k, r} m)$, then
\[
\tfrac{1}{\mu_{-j}^{(l, k)}} X_j^{(l)} Z_0^{(k)} - \tfrac{1}{\nu_j^{(k, l)}} Z_0^{(k)} X_j^{(l)} = 0.
\]

\item \label{rel XZZ one}
If $r > 1$, then
\[
\tfrac{1}{\mu_1 \kappa_1} X_{-1}^{(k)} Z_0^{(k)} Z_0^{(k)} - (\tfrac{1}{\nu_{-1} \kappa_1} + \tfrac{1}{\mu_1 \kappa_{-1}}) Z_0^{(k)} X_{-1}^{(k)} Z_0^{(k)} + \tfrac{1}{\nu_{-1} \kappa_{-1}} Z_0^{(k)} Z_0^{(k)} X_{-1}^{(k)} = 0.
\]

\addtocounter{enumi}{-1}
\renewcommand{\theenumi}{\arabic{enumi}'}

\item
If $r = 1$, then
\begin{multline*}
\tfrac{1}{\mu_1 \kappa_1} X_{-1}^{(k)} Z_0^{(k)} Z_0^{(k)} - (\tfrac{1}{\nu_{-1} \kappa_1} + \tfrac{1}{\mu_1 \kappa_{-1}}) Z_0^{(k)} X_{-1}^{(k)} Z_0^{(k)}
\\
+ \tfrac{1}{\nu_{-1} \kappa_{-1}} Z_0^{(k)} Z_0^{(k)}X_{-1}^{(k)} = \tfrac{1}{\kappa_1} X_{0, \delta_{k, r} m - 2}^{(k)} - \tfrac{1}{\kappa_{-1}} X_{-1, \delta_{k, r} m - 1}^{(k)}.
\end{multline*}

\renewcommand{\theenumi}{\arabic{enumi}}

\item
If $r > 1$, then
\[
\tfrac{1}{\mu_{-1}' \kappa_{-1}} X_{\delta_{k, r} m}^{(k + 1)} Z_0^{(k)} Z_0^{(k)} - (\tfrac{1}{\nu_1' \kappa_{-1}} + \tfrac{1}{\mu_{-1}' \kappa_1}) Z_0^{(k)} X_{\delta_{k, r} m}^{(k + 1)} Z_0^{(k)} + \tfrac{1}{\nu_1' \kappa_1} Z_0^{(k)} Z_0^{(k)} X_{\delta_{k, r} m}^{(k + 1)} = 0.
\]

\addtocounter{enumi}{-1}
\renewcommand{\theenumi}{\arabic{enumi}'}

\item
If $r = 1$, then
\begin{multline*}
\tfrac{1}{\mu_{-1}' \kappa_{-1}} X_{\delta_{k, r} m}^{(k + 1)} Z_0^{(k)} Z_0^{(k)} - (\tfrac{1}{\nu_1' \kappa_{-1}} + \tfrac{1}{\mu_{-1}' \kappa_1}) Z_0^{(k)} X_{\delta_{k, r} m}^{(k + 1)} Z_0^{(k)}
\\
+ \tfrac{1}{\nu_1' \kappa_1} Z_0^{(k)} Z_0^{(k)} X_{\delta_{k, r} m}^{(k + 1)} = \tfrac{1}{\kappa_1} X_{1, \delta_{k, r} m - 1}^{(k)} - \tfrac{1}{\kappa_{-1}} X_{0, \delta_{k, r} m}^{(k)}.
\end{multline*}

\renewcommand{\theenumi}{\arabic{enumi}}

\item \label{rel XXZ one}
We have
\[
\tfrac{1}{\mu_0 \mu_1} X_{-1}^{(k)} X_{-1}^{(k)} Z_0^{(k)} - (\tfrac{1}{\mu_0 \nu_{-1}} + \tfrac{1}{\mu_1 \nu_0}) X_{-1}^{(k)} Z_0^{(k)} X_{-1}^{(k)} + \tfrac{1}{\nu_{-1} \nu_0} Z_0^{(k)} X_{-1}^{(k)} X_{-1}^{(k)} = 0.
\]

\item
We have
\begin{multline*}
\tfrac{1}{\mu_{-1} \mu_1} X_0^{(k)} X_{-1}^{(k)} Z_0^{(k)} - \tfrac{1}{\mu_{-1} \nu_{-1}} X_0^{(k)} Z_0^{(k)} X_{-1}^{(k)}
\\
- \tfrac{1}{\mu_1 \nu_1} X_{-1}^{(k)} Z_0^{(k)} X_0^{(k)} + \tfrac{1}{\nu_{-1} \nu_1} Z_0^{(k)} X_{-1}^{(k)} X_0^{(k)} = 0.
\end{multline*}

\item
We have
\begin{multline*}
\tfrac{1}{\mu_{-2}' \mu_1} X_{\delta_{k, r} m}^{(k + 1)} X_{-1}^{(k)} Z_0^{(k)} - \tfrac{1}{\mu_{-2}' \nu_{-1}} X_{\delta_{k, r} m}^{(k + 1)} Z_0^{(k)} X_{-1}^{(k)}
\\
- \tfrac{1}{\mu_1 \nu_2'} X_{-1}^{(k)} Z_0^{(k)} X_{\delta_{k, r} m}^{(k + 1)} + \tfrac{1}{\nu_{-1} \nu_2'} Z_0^{(k)} X_{-1}^{(k)} X_{\delta_{k, r} m}^{(k + 1)} = 0.
\end{multline*}

\item
We have
\begin{multline*}
\tfrac{1}{\mu_{-1}' \mu_1'} X_{\delta_{k, r} m - 1}^{(k + 1)} X_{\delta_{k, r} m}^{(k + 1)} Z_0^{(k)} - \tfrac{1}{\mu_1' \nu_1'} X_{\delta_{k, r} m - 1}^{(k + 1)} Z_0^{(k)} X_{\delta_{k, r} m}^{(k + 1)}
\\
- \tfrac{1}{\mu_{-1}' \nu_{-1}'} X_{\delta_{k, r} m}^{(k + 1)} Z_0^{(k)} X_{\delta_{k, r} m - 1}^{(k + 1)} + \tfrac{1}{\nu_{-1}' \nu_1'} Z_0^{(k)} X_{\delta_{k, r} m}^{(k + 1)} X_{\delta_{k, r} m - 1}^{(k + 1)} = 0.
\end{multline*}

\item \label{rel XXZ five}
We have
\begin{multline*}
\tfrac{1}{\mu_{-1}' \mu_0'} X_{\delta_{k, r} m}^{(k + 1)} X_{\delta_{k, r} m}^{(k + 1)} Z_0^{(k)}
\\
- (\tfrac{1}{\mu_0' \nu_1'} + \tfrac{1}{\mu_{-1}' \nu_0'}) X_{\delta_{k, r} m}^{(k + 1)} Z_0^{(k)} X_{\delta_{k, r} m}^{(k + 1)} + \tfrac{1}{\nu_0' \nu_1'} Z_0^{(k)} X_{\delta_{k, r} m}^{(k + 1)} X_{\delta_{k, r} m}^{(k + 1)} = 0.
\end{multline*}

\end{enumerate}
\end{prop}

\begin{proof}
The above formulas follow by direct calculations using Lemma~\ref{lemm multZik}. For~\eqref{rel XZZ one}--\eqref{rel XXZ five} one also has to use either
\begin{gather*}
\tfrac{1}{\mu_1} X_{-1}^{(k)} Z_0^{(k)} - \tfrac{1}{\nu_{-1}} Z_0^{(k)} X_{-1}^{(k)} = Z_{-1}^{(k)}
\\
\intertext{or}
\tfrac{1}{\mu_{-1}'} X_{\delta_{k, r} m}^{(k + 1)} Z_0^{(k)} - \tfrac{1}{\nu_1'} Z_0^{(k)} X_{\delta_{k, r} m}^{(k + 1)} = - Z_1^{(k)},
\end{gather*}
which also follow from Lemma~\ref{lemm multZik}.
\end{proof}

\subsection{Generators}

Let $\cA$ be the algebra generated by $x_i^{(k)}$, $z^{(k)}$, $k \in [1, r]$, $i \in \bbZ$. We also put
\[
\bdeg \Phi := (\bdeg_\cZ \Phi, \bdeg_\cX \Phi) \in \bbN^{[1, r]} \times \bbN^{([1, r] \times \bbZ)},
\]
where, for $k \in [1, r]$ and $i \in \bbZ$,
\[
(\bdeg_\cZ \Phi) (k) := \deg_{z^{(k)}} \Phi \qquad \text{and} \qquad (\bdeg_\cX \Phi) (k, i) := \deg_{x_i^{(k)}} \Phi.
\]
Recall that $(\bd_\cZ, \bd_\cX) \leq (\bd'_\cZ, \bd'_\cX)$, for $\bd_\cZ, \bd'_\cZ  \in \bbN^{[1, r]}$ and $\bd_\cX, \bd'_\cX \in \bbN^{([1, r] \times \bbZ)}$, if and only if either $\bd_\cZ \Phi < \bd'_\cZ \Phi$ or $\bd_\cZ \Phi = \bd'_\cZ \Phi$ and $\bd_\cX \Phi \leq \bd'_\cX \Phi$.

We have the following.

\begin{lemm} \label{lemm poly prim}
There exists polynomials $\Phi_i^{(k)} \in \cA$ such that $\bdeg \Phi_i^{(k)} \leq \bdim Z_i^{(k)}$ and
\[
\Phi_i^{(k)} ((X_j^{(l)}), (Z_0^{(l)})) = Z_i^{(k)},
\]
for all $k \in [1, r]$ and $i \in \bbZ$.
\end{lemm}

\begin{proof}
We prove the claim by induction on $|i|$. Note that $\Phi_0^{(k)} = z^{(k)}$. Thus assume $|i| > 0$. If $i < 0$, then it follows from Lemma~\ref{lemm multZik}, that
\[
Z_i^{(k)} = \tfrac{1}{\mu_1} X_i^{(k)} Z_{i + 1}^{(k)} - \tfrac{1}{\nu_{-1}} Z_{i + 1}^{(k)} X_i^{(k)}.
\]
Thus in this case the claim follows if we take
\[
\Phi_i^{(k)} := \tfrac{1}{\mu_1} x_i^{(k)} \Phi_{i + 1}^{(k)} - \tfrac{1}{\nu_{-1}} \Phi_{i + 1}^{(k)} x_i^{(k)}.
\]
If $i > 0$, then we proceed similarly.
\end{proof}

Using Lemmas~\ref{lemm poly} and \ref{lemm poly prim}, and Proposition~\ref{prop poly} we get the following.

\begin{prop} \label{prop poly prim}
For each $M \in \cC$ there exits a polynomial $\Phi_M \in \cA$ such that $\bdeg \Phi_M \leq \bdim M$ and
\[
\Phi_M ((X_i^{(k)}), (Z_0^{(k)})) = M. \eqno \qed
\]
\end{prop}

Consequently, we we obtain the following.

\begin{coro} \label{coro gen prim}
The algebra $\cH$ is generated by $X_i^{(k)}$, $Z_0^{(k)}$, $k \in [1, r]$, $i \in \bbZ$. \qed
\end{coro}

\subsection{Presentation}

Recall from Lemma~\ref{lemm poly} that if $k \in [1, r]$ and $i \leq j$, then there exists $\Phi_{i, j}^{(k)} \in \cA$ such that $\Phi_{i, j}^{(k)} ((X_s^{(l)}), (Z_0^{(l)})) = X_{i, j}^{(k)}$ and $\bdeg \Phi_{i, j}^{(k)} \leq \bdim X_{i, j}^{(k)}$. Moreover, we put $\Phi_{i, i - 1}^{(k)} := 0$.

This subsection is devoted to the proof of the following.

\begin{theo} \label{theo main prim}
Let $\cA$ be the algebra generated by $x_i^{(k)}$, $z^{(k)}$, $k \in [1, r]$, $i \in \bbZ$.  Let $\cI$ be the ideal in $\cA$ generated by the following elements:
\begin{enumerate}

\item
$\tfrac{1}{\lambda_{j - i}^{(k, l)}} x_i^{(k)} x_j^{(l)} - \tfrac{1}{\lambda_{i - j}^{(l, k)}} x_j^{(l)} x_i^{(k)}$, $(k, i), (l, j) \in [1, r] \times \bbZ$, $(l, j) \neq (k, i), (k, i + 1), (k, i - 1), (k + 1, i + \delta_{k, r} m), (k - 1, i - \delta_{k, 1} m)$,

\addtocounter{enumi}{-1}
\renewcommand{\theenumi}{\arabic{enumi}'}

\item
$\tfrac{1}{\lambda_{\delta_{k, r} m}^{(k, k + 1)}} x_i^{(k)} x_{i + \delta_{k, r} m}^{(k + 1)} - \tfrac{1}{\lambda_{- \delta_{k, r} m}^{(k + 1, k)}} x_{i + \delta_{k, r} m}^{(k + 1)} x_i^{(k)} - \tfrac{\frakq}{\frakq - 1}$, for $(k, i) \in [1, r] \times \bbZ$, if $(r, m) \neq (1, 1)$,

\renewcommand{\theenumi}{\arabic{enumi}}

\item
$\tfrac{1}{\lambda_1^2} x_i^{(k)} x_i^{(k)} x_{i + 1}^{(k)} - \tfrac{1}{\lambda_1 \lambda_{-1}} (\frakq + 1) x_i^{(k)} x_{i + 1}^{(k)} x_i^{(k)} + \tfrac{1}{\lambda_{-1}^2} \frakq x_{i + 1}^{(k)} x_i^{(k)} x_i^{(k)} - \delta_{r m, 1} \frakq (\frakq + 1) x_i^{(k)}$, for $(k, i) \in [1, r] \times \bbZ$,

\item
$\tfrac{1}{\lambda_1^2} x_i^{(k)} x_{i + 1}^{(k)} x_{i + 1}^{(k)} - \tfrac{1}{\lambda_1 \lambda_{-1}} (\frakq + 1) x_{i + 1}^{(k)} x_i^{(k)} x_{i + 1}^{(k)} + \tfrac{1}{\lambda_{-1}^2} \frakq x_{i + 1}^{(k)} x_{i + 1}^{(k)} x_i^{(k)} - \delta_{r m, 1} \frakq (\frakq + 1) x_{i + 1}^{(k)}$, for $(k, i) \in [1, r] \times \bbZ$.

\item
$\tfrac{1}{\kappa_0^{(k, l)}} z^{(k)} z^{(l)} - \tfrac{1}{\kappa_0^{(l, k)}} z^{(l)} z^{(k)}$, $k, l \in [1, r]$, $l \neq k - 1, k, k + 1$.

\addtocounter{enumi}{-1}
\renewcommand{\theenumi}{\arabic{enumi}'}

\item
$\tfrac{1}{\kappa_0^{(k, k + 1)}} z^{(k)} z^{(k + 1)} - \tfrac{1}{\kappa_0^{(k + 1, k)}} z^{(k + 1)} z^{(k)} - \tfrac{1}{\kappa_0^{(k, k + 1)}} \Phi_{0, \delta_{k, r} m - 1}^{(k + 1)} + \delta_{r, 2} \tfrac{1}{\kappa_0^{(k + 1, k)}} \Phi_{0, \delta_{k + 1, r} m - 1}^{(k)}$, $k \in [1, r]$, $r \geq 2$.

\renewcommand{\theenumi}{\arabic{enumi}}

\item
$\tfrac{1}{\mu_{-j}^{(l, k)}} x_j^{(l)} z^{(k)} - \tfrac{1}{\nu_j^{(k, l)}} z^{(k)} x_j^{(l)}$, $k, l \in [1, r]$, $j \in \bbZ$, $(l, j) \neq (k, -1), (k + 1, \delta_{k, r} m)$.

\item
$\tfrac{1}{\mu_1 \kappa_1} x_{-1}^{(k)} z^{(k)} z^{(k)} - (\tfrac{1}{\nu_{-1} \kappa_1} + \tfrac{1}{\mu_1 \kappa_{-1}}) z^{(k)} x_{-1}^{(k)} z^{(k)} + \tfrac{1}{\nu_{-1} \kappa_{-1}} z^{(k)} z^{(k)} x_{-1}^{(k)} - \delta_{r, 1} \tfrac{1}{\kappa_1} \Phi_{0, \delta_{k, r} m - 2}^{(k)} + \delta_{r, 1} \tfrac{1}{\kappa_{-1}} \Phi_{-1, \delta_{k, r} m - 1}^{(k)}$, $k \in [1, r]$.

\item
$\tfrac{1}{\mu_{-1}' \kappa_{-1}} x_{\delta_{k, r} m}^{(k + 1)} z^{(k)} z^{(k)} - (\tfrac{1}{\nu_1' \kappa_{-1}} + \tfrac{1}{\mu_{-1}' \kappa_1}) z^{(k)} x_{\delta_{k, r} m}^{(k + 1)} z^{(k)} + \tfrac{1}{\nu_1' \kappa_1} z^{(k)} z^{(k)} x_{\delta_{k, r} m}^{(k + 1)} - \delta_{r, 1} \tfrac{1}{\kappa_1} x_{1, \delta_{k, r} m - 1}^{(k)} + \delta_{r, 1} \tfrac{1}{\kappa_{-1}} x_{0, \delta_{k, r} m}^{(k)}$, $k \in [1, r]$.

\item
$\tfrac{1}{\mu_0 \mu_1} x_{-1}^{(k)} x_{-1}^{(k)} z^{(k)} - (\tfrac{1}{\mu_0 \nu_{-1}} + \tfrac{1}{\mu_1 \nu_0}) x_{-1}^{(k)} z^{(k)} x_{-1}^{(k)} + \tfrac{1}{\nu_{-1} \nu_0} z^{(k)} x_{-1}^{(k)} x_{-1}^{(k)}$, $k \in [1, r]$.

\item
$\tfrac{1}{\mu_{-1} \mu_1} x_0^{(k)} x_{-1}^{(k)} z^{(k)} - \tfrac{1}{\mu_{-1} \nu_{-1}} x_0^{(k)} z^{(k)} x_{-1}^{(k)} - \tfrac{1}{\mu_1 \nu_1} x_{-1}^{(k)} z^{(k)} x_0^{(k)} + \tfrac{1}{\nu_{-1} \nu_1} z^{(k)} x_{-1}^{(k)} x_0^{(k)}$, $k \in [1, r]$.

\item
$\tfrac{1}{\mu_{-2}' \mu_1} x_{\delta_{k, r} m}^{(k + 1)} x_{-1}^{(k)} z^{(k)} - \tfrac{1}{\mu_{-2}' \nu_{-1}} x_{\delta_{k, r} m}^{(k + 1)} z^{(k)} x_{-1}^{(k)} - \tfrac{1}{\mu_1 \nu_2'} x_{-1}^{(k)} z^{(k)} x_{\delta_{k, r} m}^{(k + 1)} + \tfrac{1}{\nu_{-1} \nu_2'} z^{(k)} x_{-1}^{(k)} x_{\delta_{k, r} m}^{(k + 1)}$, $k \in [1, r]$.

\item
$\tfrac{1}{\mu_{-1}' \mu_1'} x_{\delta_{k, r} m - 1}^{(k + 1)} x_{\delta_{k, r} m}^{(k + 1)} z^{(k)} - \tfrac{1}{\mu_1' \nu_1'} x_{\delta_{k, r} m - 1}^{(k + 1)} z^{(k)} x_{\delta_{k, r} m}^{(k + 1)} - \tfrac{1}{\mu_{-1}' \nu_{-1}'} x_{\delta_{k, r} m}^{(k + 1)} z^{(k)} x_{\delta_{k, r} m - 1}^{(k + 1)} + \tfrac{1}{\nu_{-1}' \nu_1'} z^{(k)} x_{\delta_{k, r} m}^{(k + 1)} x_{\delta_{k, r} m - 1}^{(k + 1)}$, $k \in [1, r]$.

\item
$\tfrac{1}{\mu_{-1}' \mu_0'} x_{\delta_{k, r} m}^{(k + 1)} x_{\delta_{k, r} m}^{(k + 1)} z^{(k)} - (\tfrac{1}{\mu_0' \nu_1'} + \tfrac{1}{\mu_{-1}' \nu_0'}) x_{\delta_{k, r} m}^{(k + 1)} z^{(k)} x_{\delta_{k, r} m}^{(k + 1)} + \tfrac{1}{\nu_0' \nu_1'} z^{(k)} x_{\delta_{k, r} m}^{(k + 1)} x_{\delta_{k, r} m}^{(k + 1)}$, $k \in [1, r]$.

\end{enumerate}
If $\cB := \cA / \cI$, then the homomorphism $\xi \colon \cA \to \cH$ given by
\[
\xi (x_i^{(k)}) := X_i^{(k)} \qquad \text{and} \qquad \xi (z^{(k)}) := Z_0^{(k)},
\]
$k \in [1, r]$, $i \in \bbZ$, induces an isomorphism $\ol{\xi} \colon \cB \to \cH$.
\end{theo}

We proceed similarly as in the proof of Theorem~\ref{theo main}. Again we know from Propositions~\ref{prop rel1}, \ref{prop rel2} and~\ref{prop rel3} that $\ol{\xi}$ is a well-defined homomorphism, which is an epimorphism by Corollary~\ref{coro gen prim}. In the rest of the subsection we show that $\ol{\xi}$ is a monomorphism.

For $\bd \in \bbN^{[1, r]} \times \bbN^{([1, r] \times \bbZ)}$, let $\cA_\bd$ be the subspace of $\cA$ spanned by the monomials $\Phi$ such that $\bdeg \Phi \leq \bd$. If $\Phi, \Phi_1, \ldots, \Phi_l \in \cA_\bd$ are monomials of the same degree vector, then we write
\[
\Phi \equiv \Phi_1 + \cdots + \Phi_l,
\]
if there exist $v_1, \ldots, v_l \in \bbQ$ and $\Psi \in \cA$ such that $\bdeg \Psi < \bdeg \Phi$ and
\[
\Phi + \cI = v_1 \Phi_1 + \cdots + v_l \Phi_l + \Psi + \cI.
\]
Using this notation we may reformulate the relations from Theorem~\ref{theo main prim} in the following way.

\begin{lemm} \label{lemm change}
Let $k \in [1, r]$ and $i \in \bbZ$. Then we have the following:
\begin{enumerate}

\item \label{change one}
$x_i^{(k)} x_j^{(l)} \equiv x_j^{(l)} x_i^{(k)}$, if $(l, j) \neq (k, i \pm 1)$;

\item \label{change two}
$x_i^{(k)} x_i^{(k)} x_{i + 1}^{(k)} \equiv x_i^{(k)} x_{i + 1}^{(k)} x_i^{(k)} + x_{i + 1}^{(k)} x_i^{(k)} x_i^{(k)}$;

\addtocounter{enumi}{-1}
\renewcommand{\theenumi}{\arabic{enumi}'}

\item \label{change two prim}
$x_i^{(k)} x_{i + 1}^{(k)} x_i^{(k)} \equiv x_i^{(k)} x_i^{(k)} x_{i + 1}^{(k)} + x_{i + 1}^{(k)} x_i^{(k)} x_i^{(k)}$;

\renewcommand{\theenumi}{\arabic{enumi}}

\item \label{change three}
$x_i^{(k)} x_{i + 1}^{(k)} x_{i + 1}^{(k)} \equiv x_{i + 1}^{(k)} x_i^{(k)} x_{i + 1}^{(k)} + x_{i + 1}^{(k)} x_{i + 1}^{(k)} x_i^{(k)}$;

\addtocounter{enumi}{-1}
\renewcommand{\theenumi}{\arabic{enumi}'}

\item \label{change three prim}
$x_{i + 1}^{(k)} x_i^{(k)} x_{i + 1}^{(k)} \equiv x_i^{(k)} x_{i + 1}^{(k)} x_{i + 1}^{(k)} + x_{i + 1}^{(k)} x_{i + 1}^{(k)} x_i^{(k)}$;

\addtocounter{enumi}{-1}
\renewcommand{\theenumi}{\arabic{enumi}''}

\item \label{change three bis}
$x_{i + 1}^{(k)} x_{i + 1}^{(k)} x_i^{(k)} \equiv x_i^{(k)} x_{i + 1}^{(k)} x_{i + 1}^{(k)} + x_{i + 1}^{(k)} x_i^{(k)} x_{i + 1}^{(k)}$;

\renewcommand{\theenumi}{\arabic{enumi}}

\item \label{change four}
$z^{(k)} z^{(l)} \equiv z^{(l)} z^{(k)}$, for any $l$;

\item \label{change five}
$x_j^{(l)} z^{(k)} \equiv z^{(k)} x_j^{(l)}$, if $(l, j) \neq (k, -1), (k + 1, \delta_{k, r} m)$;

\addtocounter{enumi}{-1}
\renewcommand{\theenumi}{\arabic{enumi}'}

\item \label{change five prim}
$z^{(k)} x_j^{(l)} \equiv x_j^{(l)} z^{(k)}$, if $(l, j) \neq (k, -1), (k + 1, \delta_{k, r} m)$;

\renewcommand{\theenumi}{\arabic{enumi}}

\item \label{change six}
$x_{-1}^{(k)} z^{(k)} z^{(k)} \equiv z^{(k)} x_{-1}^{(k)} z^{(k)} + z^{(k)} z^{(k)} x_{-1}^{(k)}$;

\addtocounter{enumi}{-1}
\renewcommand{\theenumi}{\arabic{enumi}'}

\item \label{change six prim}
$z^{(k)} x_{-1}^{(k)} z^{(k)} \equiv x_{-1}^{(k)} z^{(k)} z^{(k)} + z^{(k)} z^{(k)} x_{-1}^{(k)}$;

\renewcommand{\theenumi}{\arabic{enumi}}

\item \label{change seven}
$x_{\delta_{k, r} m}^{(k + 1)} z^{(k)} z^{(k)} \equiv z^{(k)} x_{\delta_{k, r} m}^{(k + 1)} z^{(k)} + z^{(k)} z^{(k)} x_{\delta_{k, r} m}^{(k + 1)}$;

\addtocounter{enumi}{-1}
\renewcommand{\theenumi}{\arabic{enumi}'}

\item \label{change seven prim}
$z^{(k)} x_{\delta_{k, r} m}^{(k + 1)} z^{(k)} \equiv x_{\delta_{k, r} m}^{(k + 1)} z^{(k)} z^{(k)} + z^{(k)} z^{(k)} x_{\delta_{k, r} m}^{(k + 1)}$;

\renewcommand{\theenumi}{\arabic{enumi}}

\item \label{change eight}
$x_{-1}^{(k)} x_{-1}^{(k)} z^{(k)} \equiv x_{-1}^{(k)} z^{(k)} x_{-1}^{(k)} + z^{(k)} x_{-1}^{(k)} x_{-1}^{(k)}$;

\item \label{change nine}
$x_0^{(k)} x_{-1}^{(k)} z^{(k)} \equiv x_0^{(k)} z^{(k)} x_{-1}^{(k)} + x_{-1}^{(k)} z^{(k)} x_0^{(k)} + z^{(k)} x_{-1}^{(k)} x_0^{(k)}$;

\item \label{change ten}
$x_{\delta_{k, r} m}^{(k + 1)} x_{-1}^{(k)} z^{(k)} \equiv x_{\delta_{k, r} m}^{(k + 1)} z^{(k)} x_{-1}^{(k)} + x_{-1}^{(k)} z^{(k)} x_{\delta_{k, r} m}^{(k + 1)} + z^{(k)} x_{-1}^{(k)} x_{\delta_{k, r} m}^{(k + 1)}$;

\item \label{change eleven}
$x_{\delta_{k, r} m - 1}^{(k + 1)} x_{\delta_{k, r} m}^{(k + 1)} z^{(k)} \equiv x_{\delta_{k, r} m - 1}^{(k + 1)} z^{(k)} x_{\delta_{k, r} m}^{(k + 1)} + x_{\delta_{k, r} m}^{(k + 1)} z^{(k)} x_{\delta_{k, r} m - 1}^{(k + 1)} + z^{(k)} x_{\delta_{k, r} m}^{(k + 1)} x_{\delta_{k, r} m - 1}^{(k + 1)}$;

\item \label{change twelve}
$x_{\delta_{k, r} m}^{(k + 1)} x_{\delta_{k, r} m}^{(k + 1)} z^{(k)} \equiv x_{\delta_{k, r} m}^{(k + 1)} z^{(k)} x_{\delta_{k, r} m}^{(k + 1)} + z^{(k)} x_{\delta_{k, r} m}^{(k + 1)} x_{\delta_{k, r} m}^{(k + 1)}$.

\end{enumerate}
\end{lemm}

Recall, that for $k \in [1, r]$ and $i \leq j$,
\[
x_{i, j}^{(k)} := x_i^{(k)} \cdots x_j^{(k)}.
\]
Moreover, $x_{i, i - 1}^{(k)} := 1$. Next, for $k \in [1, r]$ and $i \in \bbZ$, we put
\[
z_i^{(k)} :=
\begin{cases}
x_i^{(k)} \cdots x_{-1}^{(k)} z^{(k)} & \text{if $i < 0$},
\\
z^{(k)} & \text{if $i = 0$},
\\
x_{i + \delta_{k, r} m - 1}^{(k + 1)} \cdots x_{\delta_{k, r} m}^{(k + 1)} z^{(k)} & \text{if $i > 0$}.
\end{cases}
\]

Given $\bd \in \bbN^{[1, r]} \times \bbN^{([1, r] \times \bbZ)}$, we introduce an order on the set of monomials of the degree vector $\bd$ in the following way. Let
\[
\Phi_1 = a_1 \cdots a_n \qquad \text{and} \qquad \Phi_2 = b_1 \cdots b_n
\]
be two monomials of degree vector $\bd$, where
\[
a_1, \ldots, a_n, b_1, \ldots, b_n \in \{ x_i^{(k)} : \text{$k \in [1, r]$, $i \in \bbZ$} \} \cup \{ z^{(k)} : \text{$k \in [1, r]$} \}.
\]
Let
\begin{gather*}
\{ s_1 < \cdots < s_l \} := \{ s \in [1, n] : \text{$a_s = z^{(k)}$ for some $k \in [1, r]$} \}
\\
\intertext{and}
\{ t_1 < \cdots < t_l \} := \{ t \in [1, n] : \text{$b_t = z^{(k)}$ for some $k \in [1, r]$} \}.
\end{gather*}
Moreover,
\begin{gather*}
\{ s_1' < \cdots < s_{n - l}' \} := [1, n] \setminus \{ s_1 < \cdots < s_l \}
\\
\intertext{and}
\{ t_1' < \cdots < t_{n - l}' \} := [1, n] \setminus \{ t_1 < \cdots < t_l \}.
\end{gather*}
Then we write $\Phi_1 < \Phi_2$ if:
\begin{enumerate}

\item
either, there exists $j \in [1, l]$ such that $a_{s_1} = b_{t_1}$, \ldots, $a_{s_{j - 1}} = b_{t_{j - 1}}$, and $a_{s_j} = z^{(k_1)}$ and $b_{t_j} = z^{(k_2)}$, for $k_1 > k_2$, i.e.\ $a_{s_1} \cdots a_{s_l}$ is lexicographically bigger then $b_{t_1} \cdots b_{t_l}$,

\item
or, $a_{s_1} \cdots a_{s_l} = b_{t_1} \cdots b_{t_l}$, and there exists $j \in [1, l]$ such that $s_1 = t_1$, \ldots, $s_{j - 1} = t_{j - 1}$, and $s_j < t_j$,

\item
or, $a_{s_1} \cdots a_{s_l} = b_{t_1} \cdots b_{t_l}$, $\{ s_1, \ldots, s_l \} = \{ t_1, \ldots, t_l \}$, and $a_{s_1'} \cdots a_{s_{n - l}'} < b_{t_1'} \cdots b_{t_{n - l}'}$ in the sense of the relation defined in subsection~\ref{sub pres X}.

\end{enumerate}

The following will be crucial.

\begin{prop} \label{prop change}
\begin{enumerate}

\item \label{prop change one}
Assume $k_1, k_2 \in [1, r]$ and $i_1, i_2 \in \bbZ$. If either $k_1 < k_2$, or $k_1 = k_2$ and $|i_1| > |i_2|$, or $k_1 = k_2$, $|i_1| = |i_2|$, and $i_1 < i_2$, then there exist monomials $\Phi_1$, \ldots, $\Phi_n$ of degree vector $\bdim (z_{i_1}^{(k_1)} z_{i_2}^{(k_2)})$ such that $\Phi_s < z_{i_1}^{(k_1)} z_{i_2}^{(k_2)}$, for each $1 \leq s \leq n$, and
\[
z_{i_1}^{(k_1)} z_{i_2}^{(k_2)} \equiv \Phi_1 + \cdots + \Phi_n.
\]

\item \label{prop change two}
Assume $k_1, k_2 \in [1, r]$, $i_1 \leq j_1$, and $i_2 \in \bbZ$. Then
\begin{enumerate}

\item \label{prop change twoa}
either $k_1 = k_2$ and $j_1 = i_2 - 1 < 0$,

\item \label{prop change twob}
or $k_1 = k_2 + 1$ and $i_1 = j_1 = i_2 + \delta_{k, r} m \geq \delta_{k, r} m$,

\item
or there exist monomials $\Phi_1$, \ldots, $\Phi_n$ of degree vector $\bdim (x_{i_1, j_1}^{(k_1)} z_{i_2}^{(k_2)})$ such that $\Phi_s < x_{i_1, j_1}^{(k_1)} z_{i_2}^{(k_2)}$, for each $1 \leq s \leq n$, and
\[
x_{i_1, j_1}^{(k_1)} z_{i_2}^{(k_2)} \equiv \Phi_1 + \cdots + \Phi_n.
\]

\end{enumerate}
\end{enumerate}
\end{prop}

\begin{proof}
\eqref{prop change one}~Write
\[
z_{i_1}^{(k_1)} = a_{|i_1|} \cdots a_1 z^{(k_1)} \qquad \text{and} \qquad z_{i_2}^{(k_2)} = b_{|i_2|} \cdots b_1 z^{(k_2)},
\]
for $a_1, \ldots, a_{|i_1|}, b_1, \cdots, b_{|i_2|} \in \{ x_i^{(l)} : \text{$l \in [1, r]$, $i \in \bbZ$} \}$. Similarly as in the proof of Proposition~\ref{prop reduction}, we have some cases to consider.

\vspace{1ex}

\textbf{Case 0$^\circ$}: $k_1 < k_2$.

Since $k_1 < k_2$, $b_s \neq x_{-1}^{(k_1)}, x_{\delta_{k_1, r} m}^{(k_1 + 1)}$, for each $s \in [1, |i_2|]$. Consequently, using iteratively Lemma~\ref{lemm change}\eqref{change five prim} and then~Lemma~\ref{lemm change}\eqref{change four}, we get
\[
z_{i_1}^{(k_1)} z_{i_2}^{(k_2)} \equiv a_{|i_1|} \cdots a_1 b_{|i_2|} \cdots b_1 z^{(k_1)} z^{(k_2)} \equiv a_{|i_1|} \cdots a_1 b_{|i_2|} \cdots b_1 z^{(k_2)} z^{(k_1)},
\]
Since
\[
a_{|i_1|} \cdots a_1 b_{|i_2|} \cdots b_1 z^{(k_2)} z^{(k_1)} < z_{i_1}^{(k_1)} z_{i_2}^{(k_2)},
\]
the claim follows in this case.

For the rest of the proof we assume $k_1 = k_2$ and denote this common value by $k$. Observe that under this assumption $i_1 \neq 0$.

\vspace{1ex}

\textbf{Case 1$^\circ$}: $i_2 = 0$.

Since $i_1 \neq 0$, using either Lemma~\ref{lemm change}\eqref{change six} (if $i_1 < 0$) or Lemma~\ref{lemm change}\eqref{change seven} (if $i_1 > 0$), we get
\[
z_{i_1}^{(k)} z_{i_2}^{(k)} \equiv a_{|i_1|} \cdots a_2 z^{(k)} a_1 z^{(k)} + a_{|i_1|} \cdots a_2 z^{(k)} z^{(k)} a_1,
\]
hence the claim follows.

\vspace{1ex}

\textbf{Case 2$^\circ$}: $i_2 < 0$.

In this case $b_1 = x_{-1}^{(k)}$. By an iterated use of Lemma~\ref{lemm change}\eqref{change five prim} we get
\[
z_{i_1}^{(k)} z_{i_2}^{(k)} \equiv a_{|i_1|} \cdots a_1 b_{|i_2|} \cdots b_2 z^{(k)} b_1 z^{(k)}.
\]
Now we use Lemma~\ref{lemm change}\eqref{change six prim} and obtain
\[
z_{i_1}^{(k)} z_{i_2}^{(k)} \equiv a_{|i_1|} \cdots a_1 b_{|i_2|} \cdots b_1 z^{(k)} z^{(k)} + a_{|i_1|} \cdots a_1 b_{|i_2|} \cdots b_2 z^{(k)} z^{(k)} b_1.
\]
Moreover, by iterated use of Lemma~\ref{lemm change}\eqref{change five} we have
\[
a_{|i_1|} \cdots a_1 b_{|i_2|} \cdots b_2 z^{(k)} z^{(k)} b_1 \equiv a_{|i_1|} \cdots a_1 z^{(k)} z^{(k)} b_{|i_2|} \cdots b_1 < z_{i_1}^{(k)} z_{i_2}^{(k)}.
\]
Thus we only have to deal with $a_{|i_1|} \cdots a_1 b_{|i_2|} \cdots b_1 z^{(k)} z^{(k)}$.

Assume first that $i_1 > 0$. We may use iteratively Lemma~\ref{lemm change}\eqref{change one} and get
\[
a_{|i_1|} \cdots a_1 b_{|i_2|} \cdots b_1 z^{(k)} z^{(k)} \equiv b_{|i_2|} \cdots b_1 a_{|i_1|} \cdots a_1 z^{(k)} z^{(k)}.
\]
Since $a_1 = x_{\delta_{k, r} m}^{(k + 1)}$, we use Lemma~\ref{lemm change}\eqref{change seven} and get
\begin{multline*}
b_{|i_2|} \cdots b_1 a_{|i_1|} \cdots a_1 z^{(k)} z^{(k)} \equiv b_{|i_2|} \cdots b_1 a_{|i_1|} \cdots a_2 z^{(k)} a_1 z^{(k)}
\\
+ b_{|i_2|} \cdots b_1 a_{|i_1|} \cdots a_2 z^{(k)} z^{(k)} a_1.
\end{multline*}
Finally, by an iterated use of Lemma~\ref{lemm change}\eqref{change five} we have
\begin{gather*}
b_{|i_2|} \cdots b_1 a_{|i_1|} \cdots a_2 z^{(k)} a_1 z^{(k)} \equiv b_{|i_2|} \cdots b_1 z^{(k)} a_{|i_1|} \cdots a_1 z^{(k)} < z_{i_1}^{(k)} z_{i_2}^{(k)}
\\
\intertext{and}
b_{|i_2|} \cdots b_1 a_{|i_1|} \cdots a_2 z^{(k)} z^{(k)} a_1 \equiv b_{|i_2|} \cdots b_1 z^{(k)} z^{(k)} a_{|i_1|} \cdots a_1 < z_{i_1}^{(k)} z_{i_2}^{(k)},
\end{gather*}
where for the inequalities we use that $|i_2| < |i_1|$.

Now assume $i_1 < 0$, hence in particular $i_1 < i_2$. In this case Proposition~\ref{prop span} implies that
\[
a_{|i_1|} \cdots a_1 b_{|i_2|} \cdots b_1 \equiv \Phi_1 + \cdots + \Phi_n,
\]
where, for each $j$,
\begin{equation}
\label{eq dim vect}
\bdim_\cX \Phi_j = \bdim_\cX (x_{i_1, -1}^{(k)} x_{i_2, -1}^{(k)})
\end{equation}
and
\[
\Phi_j = x_{s_1, t_1}^{(k)} x_{s_2, t_2}^{(k)} \cdots x_{s_l, t_l}^{(k)},
\]
$l \in \bbN$, such that either $s_d > s_{d + 1}$ or $s_d = s_{d + 1}$ and $t_d \geq t_{d + 1}$. Fix $j \in [1, l]$. There exist $1 \leq d_1 < d_2 \leq l$ such that $t_{d_1} = -1 = t_{d_2}$. If $d_1 > 1$, then $\bdim_{x_{s_1}^{(k)}} \Phi \geq 3$, since $s_{d_2} \leq s_{d_1} \leq s_1 \leq -1 = t_{d_1} = t_{d_2}$. This contradicts~\eqref{eq dim vect}, hence $d_1 = 1$. Similarly, $s_1 > i_1$, since otherwise $\deg_{x_{i_1}^{(k)}} \Phi_j \geq 2$, which again contradicts~\eqref{eq dim vect}, as $i_1 < i_2$. Now, using iteratively Lemma~\ref{lemm change}\eqref{change five} and Lemma~\ref{lemm change}\eqref{change six} (in an appropriate order), we get
\[
\Phi_j z^{(k)} z^{(k)} \equiv z_{s_1}^{(k)} z^{(k)} x_{s_2, t_2}^{(k)} \cdots x_{s_l, t_l}^{(k)} + z_{s_1}^{(k)} x_{s_2, t_2}^{(k)} \cdots x_{s_{d_2}, t_{d_2}}^{(k)} z^{(k)} x_{s_{d_2 + 1}, t_{d_2 + 1}}^{(k)} \cdots x_{s_l, t_l}^{(k)}.
\]
Since $s_1 > i_1$,
\[
z_{s_1}^{(k)} z^{(k)} x_{s_2, t_2}^{(k)} \cdots x_{s_l, t_l}^{(k)}, z_{s_1}^{(k)} x_{s_2, t_2}^{(k)} \cdots x_{s_{d_2}, t_{d_2}}^{(k)} z^{(k)} x_{s_{d_2 + 1}, t_{d_2 + 1}}^{(k)} \cdots x_{s_l, t_l}^{(k)} < z_{i_1}^{(k)} z_{i_2}^{(k)},
\]
hence the claim follows in this case.

\vspace{1ex}

\textbf{Case 3$^\circ$}: $i_2 > 0$.

This case is dual to Case~2$^\circ$. The only difference is that when $i_1 < 0$, we only have $|i_2| \leq |i_1|$, but we also use $i_2 > i_1$, if $|i_2| = |i_1|$.

\vspace{1ex}

\eqref{prop change two}
Again we need to consider some cases.

\vspace{1ex}

\textbf{Case 0$^\circ$}: $k_2 \neq k_1, k_1 - 1$ (in particular, $r \geq 3$).

In this case, by an iterated use of Lemma~\ref{lemm change}\eqref{change one} and Lemma~\ref{lemm change}\eqref{change five} we get
\[
x_{i_1, j_1}^{(k_1)} z_{i_2}^{(k_2)} \equiv z_{i_2}^{(k_2)}x_{i_1, j_1}^{(k_1)} < x_{i_1, j_1}^{(k_1)} z_{i_2}^{(k_2)},
\]
hence the claim follows in this case.

\vspace{1ex}

\textbf{Case 1$^\circ$}: $k_1 = k_2$. We denote this common value by $k$.

There are some subcases in this case.

\vspace{0.5ex}

\textbf{Subcase 1.0$^\circ$}: $i_2 = 0$.

Condition~\eqref{prop change twoa} implies that $j_1 \neq -1$. Similarly, condition~\eqref{prop change twob} means that if $r = 1$, then either $i_1 < j_1$ or $j_1 \neq \delta_{k, r} m$. Consequently, if either $r \neq 1$ or $j_1 \neq \delta_{k, r} m$, we use Lemma~\ref{lemm change}\eqref{change five} and get
\[
x_{i_1, j_1}^{(k)} z_{i_2}^{(k)} = x_{i_1, j_1 - 1}^{(k)} x_{j_1}^{(k)} z^{(k)} \equiv x_{i_1, j_1 - 1}^{(k)} z^{(k)} x_{j_1}^{(k)} < x_{i_1, j_1}^{(k)} z_{i_2}^{(k)}.
\]
On the other hand, if $r = 1$ and $j_1 = \delta_{k, r} m$, then $i_1 < j_1$, hence Lemma~\ref{lemm change}\eqref{change eleven} gives
\[
x_{i_1, j_1}^{(k)} z_{i_2}^{(k)} \equiv x_{i_1, j_1 - 2}^{(k)} x_{j_1 - 1}^{(k)} z^{(k)} x_{j_1}^{(k)} + x_{i_1, j_1 - 2}^{(k)} x_{j_1}^{(k)} z^{(k)} x_{j_1 - 1}^{(k)} + x_{i_1, j_1 - 2}^{(k)} z^{(k)} x_{j_1}^{(k)} x_{j_1- 1}^{(k)},
\]
which finishes proof in this case.

\vspace{0.5ex}

\textbf{Subcase 1.1$^\circ$}: $i_2 < 0$.

Note that by condition~\eqref{prop change twoa} $j_1 \neq i_2 - 1$.

Assume either $j_1 < i_2 - 1$ or $j_1 > 0$. Moreover, if $r = 1$, we also assume $j_1 \neq \delta_{k, r} m$. Then an iterated use of Lemma~\ref{lemm change}\eqref{change one} and~\ref{lemm change}\eqref{change five} gives
\[
x_{i_1, j_1}^{(k)} z_{i_2}^{(k)} \equiv x_{i_1, j_1 - 1}^{(k)} z_{i_2}^{(k)} x_{j_1}^{(k)} < x_{i_1, j_1}^{(k)} z_{i_2}^{(k)}.
\]

Next assume that either $j_1 = 0$ or $r = 1$ and $j_1 = \delta_{k, r} m$. Using iteratively Lemma~\ref{lemm change}\eqref{change one} we get
\[
x_{i_1, j_1}^{(k)} z_{i_2}^{(k)} \equiv x_{i_1, j_1 - 1}^{(k)} x_{i_2, -2}^{(k)} x_{j_1}^{(k)} x_{-1}^{(k)} z^{(k)}.
\]
Now we use either Lemma~\ref{lemm change}\eqref{change nine} (if $j_1 = 0$) or Lemma~\ref{lemm change}\eqref{change ten} (if $r = 1$ and $j_1 = \delta_{k, r} m$) and get
\begin{multline*}
x_{i_1, j_1}^{(k)} z_{i_2}^{(k)} \equiv x_{i_1, j_1 - 1}^{(k)} x_{i_2, -2}^{(k)} x_{j_1}^{(k)} z^{(k)} x_{-1}^{(k)}
\\
+ x_{i_1, j_1 - 1}^{(k)} x_{i_2, -2}^{(k)} x_{-1}^{(k)} z^{(k)} x_{j_1}^{(k)} + x_{i_1, j_1 - 1}^{(k)} x_{i_2, -2}^{(k)} z^{(k)} x_{-1}^{(k)} x_{j_1}^{(k)},
\end{multline*}
hence the claim follows in this case.

Now assume $j_1 = i_2$. If $i_2 = -1$, then we use Lemma~\ref{lemm change}\eqref{change eight} and get
\[
x_{i_1, j_1}^{(k)} z_{i_2}^{(k)} \equiv x_{i_1, j_1 - 1}^{(k)} x_{-1}^{(k)} z^{(k)} x_{-1}^{(k)} + x_{i_1, j_1 - 1}^{(k)} z^{(k)} x_{-1}^{(k)} x_{-1}^{(k)}.
\]
On the other hand, if $i_2 < -1$, then using Lemma~\ref{lemm change}\eqref{change two} we get
\[
x_{i_1, j_1}^{(k)} z_{i_2}^{(k)} \equiv x_{i_1, j_1 - 1}^{(k)} x_{j_1}^{(k)} x_{j_1 + 1}^{(k)} x_{j_1}^{(k)} z_{i_2 + 2}^{(k)} + x_{i_1, j_1 - 1}^{(k)} x_{j_1 + 1}^{(k)} x_{j_1}^{(k)} x_{j_1}^{(k)} z_{i_2 + 2}^{(k)},
\]
hence the claim follows in this case.

Finally assume $j_1 \in [i_2 + 1, -1]$ (in particular, $i_2 < -1$). We use (if $j_1 > i_2 + 1$) iteratively Lemma~\ref{lemm change}\eqref{change one} and get
\[
x_{i_1, j_1}^{(k)} z_{i_2}^{(k)} \equiv x_{i_1, j_1 - 1}^{(k)} x_{i_2, j_1 - 2}^{(k)} x_{j_1}^{(k)} x_{j_1 - 1}^{(k)} x_{j_1}^{(k)} z_{j_1 + 1}^{(k)}.
\]
Next we use Lemma~\ref{lemm change}\eqref{change three prim} and get
\[
x_{i_1, j_1}^{(k)} z_{i_2}^{(k)} \equiv x_{i_1, j_1 - 1}^{(k)} x_{i_2, j_2 - 2}^{(k)} x_{j_2 - 1}^{(k)} x_{j_2}^{(k)} x_{j_2}^{(k)} z_{j_2 + 1}^{(k)} + x_{i_1, j_1 - 1}^{(k)} x_{i_2, j_2 - 2}^{(k)} x_{j_2}^{(k)} x_{j_2}^{(k)} x_{j_2 - 1}^{(k)} z_{j_2 + 1}^{(k)}.
\]
Using iteratively Lemma~\ref{lemm change}\eqref{change one} and Lemma~\ref{lemm change}\eqref{change five} we have
\[
x_{i_1, j_1 - 1}^{(k)} x_{i_2, j_2 - 2}^{(k)} x_{j_2}^{(k)} x_{j_2}^{(k)} x_{j_2 - 1}^{(k)} z_{j_2 + 1}^{(k)} \equiv x_{i_1, j_1 - 1}^{(k)} x_{i_2, j_2 - 2}^{(k)} x_{j_2}^{(k)} x_{j_2}^{(k)} z_{j_2 + 1}^{(k)} x_{j_2 - 1}^{(k)} < x_{i_1, j_1}^{(k)} z_{i_2}^{(k)}.
\]
Now, if $j_2 = -1$, then we use Lemma~\ref{lemm change}\eqref{change eight} and get
\[
x_{i_1, j_1 - 1}^{(k)} x_{i_2, j_2 - 2}^{(k)} x_{j_2 - 1}^{(k)} x_{j_2}^{(k)} x_{j_2}^{(k)} z_{j_2 + 1}^{(k)} \equiv x_{i_1, j_1 - 1}^{(k)} z_{i_2}^{(k)} x_{-1}^{(k)} + x_{i_1, j_1 - 1}^{(k)} x_{i_2, j_2 - 1}^{(k)} z_{j_2 + 1}^{(k)} x_{-1}^{(k)} x_{-1}^{(k)}.
\]
On the other hand, if $j_2 < -1$, then we use Lemma~\ref{lemm change}\eqref{change two} and obtain
\begin{multline*}
x_{i_1, j_1 - 1}^{(k)} x_{i_2, j_2 - 2}^{(k)} x_{j_2 - 1}^{(k)} x_{j_2}^{(k)} x_{j_2}^{(k)} z_{j_2 + 1}^{(k)} \equiv x_{i_1, j_1 - 1}^{(k)} x_{i_2, j_2 + 1}^{(k)} x_{j_2}^{(k)} z_{j_2 + 2}^{(k)}
\\
+ x_{i_1, j_1 - 1}^{(k)} x_{i_2, j_2 - 1}^{(k)} x_{j_2 + 1}^{(k)} x_{j_2}^{(k)} x_{j_2}^{(k)} z_{j_2 + 2}^{(k)}.
\end{multline*}
Furthermore, we iteratively use Lemma~\ref{lemm change}\eqref{change one} and Lemma~\ref{lemm change}\eqref{change five} and have
\begin{gather*}
x_{i_1, j_1 - 1}^{(k)} x_{i_2, j_2 + 1}^{(k)} x_{j_2}^{(k)} z_{j_2 + 2}^{(k)} \equiv x_{i_1, j_1 - 1}^{(k)} z_{i_2}^{(k)} x_{j_2}^{(k)}
\\
\intertext{and}
x_{i_1, j_1 - 1}^{(k)} x_{i_2, j_2 - 1}^{(k)} x_{j_2 + 1}^{(k)} x_{j_2}^{(k)} x_{j_2}^{(k)} z_{j_2 + 2}^{(k)} \equiv x_{i_1, j_1 - 1}^{(k)} x_{i_2, j_2 - 1}^{(k)} z_{j_2 + 1}^{(k)} x_{j_2}^{(k)} x_{j_2}^{(k)},
\end{gather*}
hence the claim follows in this case.
\vspace{0.5ex}

\textbf{Subcase 1.2$^\circ$}: $i_2 > 0$.

In this case condition~\eqref{prop change twob} implies that, if $r = 1$, then either $i_1 < j_1$ or $j_1 \neq i_2 + \delta_{k, r} m$.

First assume $r > 1$. In this case, if $j_1 \neq -1$, then we iteratively use Lemma~\ref{lemm change}\eqref{change one} and Lemma~\ref{lemm change}\eqref{change five} and obtain
\[
x_{i_1, j_1}^{(k)} z_{i_2}^{(k)} \equiv x_{i_1, j_1 - 1}^{(k)} z_{i_2}^{(k)} x_{j_1}^{(k)} < x_{i_1, j_1}^{(k)} z_{i_2}^{(k)}.
\]
On the other hand, if $j_1 = -1$, then we iteratively use Lemma~\ref{lemm change}\eqref{change one} again and get
\[
x_{i_1, j_1}^{(k)} z_{i_2}^{(k)} \equiv x_{i_1, j_1 - 1}^{(k)} x_{i_2 + \delta_{k, r} m - 1}^{(k + 1)} \cdots x_{\delta_{k, r} m + 1}^{(k + 1)} x_{\delta_{k, r} m}^{(k + 1)} x_{-1}^{(k)} z^{(k)}.
\]
Now we use Lemma~\ref{lemm change}\eqref{change ten} and get
\begin{align*}
\lefteqn{x_{i_1, j_1 - 1}^{(k)} x_{i_2 + \delta_{k, r} m - 1}^{(k + 1)} \cdots x_{\delta_{k, r} m + 1}^{(k + 1)} x_{\delta_{k, r} m}^{(k + 1)} x_{-1}^{(k)} z^{(k)}} \qquad &
\\
& \equiv x_{i_1, j_1 - 1}^{(k)} x_{i_2 + \delta_{k, r} m - 1}^{(k + 1)} \cdots x_{\delta_{k, r} m + 1}^{(k + 1)} x_{\delta_{k, r} m}^{(k + 1)} z^{(k)} x_{-1}^{(k)}
\\
& + x_{i_1, j_1 - 1}^{(k)} x_{i_2 + \delta_{k, r} m - 1}^{(k + 1)} \cdots x_{\delta_{k, r} m + 1}^{(k + 1)} x_{-1}^{(k)} z^{(k)} x_{\delta_{k, r} m}^{(k + 1)}
\\
& + x_{i_1, j_1 - 1}^{(k)} x_{i_2 + \delta_{k, r} m - 1}^{(k + 1)} \cdots x_{\delta_{k, r} m + 1}^{(k + 1)} z^{(k)} x_{-1}^{(k)} x_{\delta_{k, r} m}^{(k + 1)}.
\end{align*}
Thus for the rest of the proof of Subcase~1.2$^\circ$ we assume $r = 1$, thus in particular $k + 1 = k$.

Under the assumption $r = 1$, if $j_1 < i_2 + \delta_{k, r} m - 2$, then we iteratively use Lemma~\ref{lemm change}\eqref{change one} and get
\[
x_{i_1, j_1}^{(k)} z_{i_2}^{(k)} \equiv x_{i_2 + \delta_{k, r} m - 1}^{(k)} x_{i_1, j_1}^{(k)} z_{i_2 - 1}^{(k)} < x_{i_1, j_1}^{(k)} z_{i_2}^{(k)}.
\]
Similarly, if $j_1 > i_2 + \delta_{k, r} m$, then we iteratively use Lemma~\ref{lemm change}\eqref{change one} and Lemma~\ref{lemm change}\eqref{change five} and obtain
\[
x_{i_1, j_1}^{(k)} z_{i_2}^{(k)} \equiv x_{i_1, j_1 - 1}^{(k)} z_{i_2}^{(k)} x_{j_1}^{(k)} < x_{i_1, j_1}^{(k)} z_{i_2}^{(k)}.
\]

Now assume $j_1 = i_2 + \delta_{k, r} m - 2$. If additionally $i_2 = 1$, then we use Lemma~\ref{lemm change}\eqref{change eleven} and get
\[
x_{i_1, j_1}^{(k)} z_{i_2}^{(k)} \equiv x_{i_1, j_1 - 1}^{(k)} x_{j_1}^{(k)} z^{(k)} x_{j_1 + 1}^{(k)} + x_{i_1, j_1 - 1}^{(k)} x_{j_1 + 1}^{(k)} z^{(k)} x_{j_1}^{(k)} + x_{i_1, j_1 - 1}^{(k)} z^{(k)} x_{j_1 + 1}^{(k)} x_{j_1}^{(k)}.
\]
On the other hand, if $i_2 > 1$, then we use Lemma~\ref{lemm change}\eqref{change two prim} and get
\[
x_{i_1, j_1}^{(k)} z_{i_2}^{(k)} \equiv x_{i_1, j_1 - 1}^{(k)} x_{j_1 + 1}^{(k)} x_{j_1}^{(k)} x_{j_1}^{(k)} z_{i_2 - 2}^{(k)} + x_{i_1, j_1 - 1}^{(k)} x_{j_1}^{(k)} x_{j_1}^{(k)} x_{j_1 + 1}^{(k)} z_{i_2 - 2}^{(k)}.
\]
Observe that
\[
x_{i_1, j_1 - 1}^{(k)} x_{j_1 + 1}^{(k)} x_{j_1}^{(k)} x_{j_1}^{(k)} z_{i_2 - 2}^{(k)} < x_{i_1, j_1}^{(k)} z_{i_2}^{(k)}.
\]
On the other hand, by an iterated use of Lemma~\ref{lemm change}\eqref{change one} and Lemma~\ref{lemm change}\eqref{change five} we get
\[
x_{i_1, j_1 - 1}^{(k)} x_{j_1}^{(k)} x_{j_1}^{(k)} x_{j_1 + 1}^{(k)} z_{i_2 - 2}^{(k)} \equiv x_{i_1, j_1}^{(k)} z_{i_2 - 1}^{(k)} x_{j_1 + 1}^{(k)} < x_{i_1, j_1}^{(k)} z_{i_2}^{(k)}.
\]

Next assume $j_1 = i_2 + \delta_{k, r} m - 1$. In this case, if $i_2 = 1$, then we use Lemma~\ref{lemm change}\eqref{change twelve} and get
\[
x_{i_1, j_1}^{(k)} z_{i_2}^{(k)} \equiv x_{i_1, j_1 - 1}^{(k)} x_{j_1}^{(k)} z^{(k)} x_{j_1}^{(k)} + x_{i_1, j_1 - 1}^{(k)} z^{(k)} x_{j_1}^{(k)} x_{j_1}^{(k)}.
\]
On the other hand, if $i_2 > 1$, then we use Lemma~\ref{lemm change}\eqref{change three bis} and get
\[
x_{i_1, j_1}^{(k)} z_{i_2}^{(k)} \equiv x_{i_1, j_1 - 1}^{(k)} x_{j_1}^{(k)} x_{j_1 - 1}^{(k)} x_{j_1}^{(k)} z_{i_2 - 2}^{(k)} + x_{i_1, j_1 - 1}^{(k)} x_{j_1 - 1}^{(k)} x_{j_1}^{(k)} x_{j_1}^{(k)} z_{i_2 - 2}^{(k)}.
\]
Moreover, if we iteratively use Lemma~\ref{lemm change}\eqref{change one} and Lemma~\ref{lemm change}\eqref{change five}, then
\begin{gather*}
x_{i_1, j_1 - 1}^{(k)} x_{j_1}^{(k)} x_{j_1 - 1}^{(k)} x_{j_1}^{(k)} z_{i_2 - 2}^{(k)} \equiv x_{i_1, j_1}^{(k)} z_{i_2 - 1}^{(k)} x_{j_1}^{(k)} < x_{i_1, j_1}^{(k)} z_{i_2}^{(k)}
\\
\intertext{and}
x_{i_1, j_1 - 1}^{(k)} x_{j_1 - 1}^{(k)} x_{j_1}^{(k)} x_{j_1}^{(k)} z_{i_2 - 2}^{(k)} \equiv x_{i_1, j_1 - 1}^{(k)} z_{i_2 - 1}^{(k)} x_{j_1}^{(k)} x_{j_1}^{(k)} < x_{i_1, j_1}^{(k)} z_{i_2}^{(k)}.
\end{gather*}

Finally assume $j_1 = i_2 + \delta_{k, r} m$. In this case $i_1 < j_1$ by condition~\eqref{prop change twob}, hence we may apply Lemma~\ref{lemm change}\eqref{change two prim} and get
\[
x_{i_1, j_1}^{(k)} z_{i_2}^{(k)} \equiv x_{i_1, j_1 - 2}^{(k)} x_{j_1 - 1}^{(k)} x_{j_1 - 1}^{(k)} x_{j_1}^{(k)} z_{i_2 - 1}^{(k)} + x_{i_1, j_1 - 2}^{(k)} x_{j_1}^{(k)} x_{j_1 - 1}^{(k)} x_{j_1 - 1}^{(k)} z_{i_2 - 1}^{(k)}.
\]
Observe that
\[
x_{i_1, j_1 - 2}^{(k)} x_{j_1}^{(k)} x_{j_1 - 1}^{(k)} x_{j_1 - 1}^{(k)} z_{i_2 - 1}^{(k)} < x_{i_1, j_1}^{(k)} z_{i_2}^{(k)}.
\]
On the other hand, if we iteratively apply Lemma~\ref{lemm change}\eqref{change one} and Lemma~\ref{lemm change}\eqref{change five}, then we get
\[
x_{i_1, j_1 - 2}^{(k)} x_{j_1 - 1}^{(k)} x_{j_1 - 1}^{(k)} x_{j_1}^{(k)} z_{i_2 - 1}^{(k)} \equiv x_{i_1, j_1 - 1}^{(k)} z_{i_2}^{(k)} x_{j_1}^{(k)} < x_{i_1, j_1}^{(k)} z_{i_2}^{(k)}.
\]

\vspace{1ex}

\textbf{Case 2$^\circ$}: $k_1 \neq k_2 = k_1 - 1$ (in particular, $r \geq 2$). Put $k := k_2$, thus $k_1 = k + 1$.

If either $i_2 \leq 0$ and $j_1 \neq \delta_{k, r} m$, or $i_2 > 0$ and $j_1 \not \in [\delta_{k, r} m - 1, i_2 + \delta_{k, r} m]$, then by an iterated use of Lemma~\ref{lemm change}\eqref{change one} and Lemma~\ref{lemm change}\eqref{change five} we get
\[
x_{i_1, j_1}^{(k + 1)} z_{i_2}^{(k)} \equiv x_{i_1, j_1 - 1}^{(k + 1)} z_{i_2}^{(k)} x_{j_1}^{(k + 1)} < x_{i_1, j_1}^{(k + 1)} z_{i_2}^{(k)}.
\]
Next, if $i_2 < 0$ and $j_1 = \delta_{k, r} m$, then we iteratively use Lemma~\ref{lemm change}\eqref{change one} and get
\[
x_{i_1, j_1}^{(k + 1)} z_{i_2}^{(k)} \equiv x_{i_1, j_1 - 1}^{(k + 1)} x_{i_2, -2}^{(k)} x_{j_1}^{(k + 1)} x_{-1}^{(k)} z^{(k)}.
\]
Furthermore, we apply Lemma~\ref{lemm change}\eqref{change ten} and get
\begin{multline*}
x_{i_1, j_1}^{(k + 1)} z_{i_2}^{(k)} \equiv x_{i_1, j_1 - 1}^{(k + 1)} x_{i_2, -2}^{(k)} x_{j_1}^{(k + 1)} z^{(k)} x_{-1}^{(k)}
\\
+ x_{i_1, j_1 - 1}^{(k + 1)} z_{i_2}^{(k)} x_{j_1}^{(k + 1)} + x_{i_1, j_1 - 1}^{(k + 1)} x_{i_2, -2}^{(k)} z^{(k)} x_{-1}^{(k)} x_{j_1}^{(k + 1)}.
\end{multline*}
On the other hand, if $i_2 = 0$ and $j_1 = \delta_{k, r} m$, then $i_1 < j_1$ by condition~\eqref{prop change twob}, thus we may use Lemma~\ref{lemm change}\eqref{change eleven} and get
\[
x_{i_1, j_1}^{(k + 1)} z_{i_2}^{(k)} \equiv x_{i_1, j_1 - 1}^{(k + 1)} z^{(k)} x_{j_1}^{(k + 1)} + x_{i_1, j_1 - 2}^{(k + 1)} x_{j_1}^{(k + 1)} z^{(k)} x_{j_1 - 1}^{(k + 1)} + x_{i_1, j_1 - 2}^{(k + 1)} z^{(k)} x_{j_1}^{(k + 1)} x_{j_1 - 1}^{(k + 1)}.
\]
Consequently, we may assume $i_2 > 0$ and $j_1 \in [\delta_{k, r} m - 1, i_2 + \delta_{k, r} m]$.

First assume $j_1 \in [\delta_{k, r} m - 1, i_2 + \delta_{k, r} m - 3]$. Using Lemma~\ref{lemm change}\eqref{change one} we get
\[
x_{i_1, j_1}^{(k + 1)} z_{i_2}^{(k)} \equiv x_{i_1, j_1 - 1}^{(k + 1)} x_{i_2 + \delta_{k, r} m - 1}^{(k + 1)} x_{j_1}^{(k + 1)} z_{i_2 - 1}^{(k)} < x_{i_1, j_1}^{(k + 1)} z_{i_2}^{(k)},
\]
hence the claim follows in this case.

Next assume $j_1 = i_2 + \delta_{k, r} m - 2$. If $i_2 = 1$, i.e.\ $j_1 = \delta_{k, r} m - 1$, then using Lemma~\ref{lemm change}\eqref{change eleven} we have
\begin{multline*}
x_{i_1, j_1}^{(k + 1)} z_{i_2}^{(k)} \equiv x_{i_1, j_1 - 1}^{(k + 1)} x_{\delta_{k, r} m - 1}^{(k + 1)} z^{(k)} x_{\delta_{k, r} m}^{(k + 1)}
\\
+ x_{i_1, j_1 - 1}^{(k + 1)} x_{\delta_{k, r} m}^{(k + 1)} z^{(k)} x_{\delta_{k, r} m - 1}^{(k + 1)} + x_{i_1, j_1 - 1}^{(k + 1)} z^{(k)} x_{\delta_{k, r} m}^{(k + 1)} x_{\delta_{k, r} m - 1}^{(k + 1)},
\end{multline*}
thus the claim follows. On the other hand, if $i_2 > 1$, then we use Lemma~\ref{lemm change}\eqref{change two prim} and get
\[
x_{i_1, j_1}^{(k + 1)} z_{i_2}^{(k)} \equiv x_{i_1, j_1 - 1}^{(k + 1)} x_{j_1}^{(k + 1)} x_{j_1}^{(k + 1)} x_{j_1 + 1}^{(k + 1)} z_{i_2 - 2}^{(k)} + x_{i_1, j_1 - 1}^{(k + 1)} x_{j_1 + 1}^{(k + 1)} x_{j_1}^{(k + 1)} x_{j_1}^{(k + 1)} z_{i_2 - 2}^{(k)}.
\]
Observe that
\[
x_{i_1, j_1 - 1}^{(k + 1)} x_{j_1 + 1}^{(k + 1)} x_{j_1}^{(k + 1)} x_{j_1}^{(k + 1)} z_{i_2 - 2}^{(k)} < x_{i_1, j_1}^{(k + 1)} z_{i_2}^{(k)}.
\]
Moreover,
\[
x_{i_1, j_1 - 1}^{(k + 1)} x_{j_1}^{(k + 1)} x_{j_1}^{(k + 1)} x_{j_1 + 1}^{(k + 1)} z_{i_2 - 2}^{(k)} \equiv x_{i_1, j_1}^{(k + 1)} x_{j_1}^{(k + 1)} z_{i_2 - 2}^{(k)} x_{j_1 + 1}^{(k + 1)}
\]
by an iterated use of Lemma~\ref{lemm change}\eqref{change one} and Lemma~\ref{lemm change}\eqref{change five}, and this finishes the proof in this case.

Now assume $j_1 = i_2 + \delta_{k, r} m - 1$. If $i_2 = 1$, i.e.\ $j_1 = \delta_{k, r} m$, then we use Lemma~\ref{lemm change}\eqref{change twelve} and get
\[
x_{i_1, j_1}^{(k + 1)} z_{i_2}^{(k)} \equiv x_{i_1, j_1 - 1}^{(k + 1)} x_{j_1}^{(k + 1)} z^{(k)} x_{j_1}^{(k + 1)} + x_{i_1, j_1 - 1}^{(k + 1)}  z^{(k)}x_{j_1}^{(k + 1)} x_{j_1}^{(k + 1)},
\]
thus the claim follows. If $i_2 > 1$, then we use Lemma~\ref{lemm change}\eqref{change three bis} and have
\[
x_{i_1, j_1}^{(k + 1)} z_{i_2}^{(k)} \equiv x_{i_1, j_1 - 1}^{(k + 1)} x_{j_1}^{(k + 1)} x_{j_1 - 1}^{(k + 1)} x_{j_1}^{(k + 1)} z_{i_2 - 2}^{(k)} + x_{i_1, j_1 - 1}^{(k + 1)} x_{j_1 - 1}^{(k + 1)} x_{j_1}^{(k + 1)} x_{j_1}^{(k + 1)} z_{i_2 - 2}^{(k)}.
\]
Furthermore,
\begin{gather*}
x_{i_1, j_1 - 1}^{(k + 1)} x_{j_1}^{(k + 1)} x_{j_1 - 1}^{(k + 1)} x_{j_1}^{(k + 1)} z_{i_2 - 2}^{(k)} \equiv x_{i_1, j_1}^{(k + 1)} z_{i_2 - 1}^{(k)} x_{j_1}^{(k + 1)}
\\
\intertext{and}
x_{i_1, j_1 - 1}^{(k + 1)} x_{j_1 - 1}^{(k + 1)} x_{j_1}^{(k + 1)} x_{j_1}^{(k + 1)} z_{i_2 - 2}^{(k)} \equiv x_{i_1, j_1 - 1}^{(k + 1)} z_{i_2 - 1}^{(k)} x_{j_1}^{(k + 1)} x_{j_1}^{(k + 1)}
\end{gather*}
by an iterated use of Lemma~\ref{lemm change}\eqref{change one} and Lemma~\ref{lemm change}\eqref{change five}.

Finally assume $j_1 = i_2 + \delta_{k, r} m$. In this case we may assume that $i_1 < j_1$ by condition~\eqref{prop change twob}. Consequently, if we use Lemma~\ref{lemm change}\eqref{change two prim}, then we get
\[
x_{i_1, j_1}^{(k + 1)} z_{i_2}^{(k)} \equiv x_{i_1, j_1 - 2}^{(k + 1)} x_{j_1 - 1}^{(k + 1)} x_{j_1 - 1}^{(k + 1)} x_{j_1}^{(k + 1)} z_{i_2 - 1}^{(k)} + x_{i_1, j_1 - 2}^{(k + 1)} x_{j_1}^{(k + 1)} x_{j_1 - 1}^{(k + 1)} x_{j_1 - 1}^{(k + 1)} z_{i_2 - 1}^{(k)}.
\]
Now
\[
x_{i_1, j_1 - 2}^{(k + 1)} x_{j_1}^{(k + 1)} x_{j_1 - 1}^{(k + 1)} x_{j_1 - 1}^{(k + 1)} z_{i_2 - 1}^{(k)} < x_{i_1, j_1}^{(k + 1)} z_{i_2}^{(k)}.
\]
On the other hand,
\[
x_{i_1, j_1 - 2}^{(k + 1)} x_{j_1 - 1}^{(k + 1)} x_{j_1 - 1}^{(k + 1)} x_{j_1}^{(k + 1)} z_{i_2 - 1}^{(k)} \equiv x_{i_1, j_1 - 1}^{(k + 1)} z_{i_2}^{(k)} x_{j_1}^{(k + 1)}
\]
by an iterated use of Lemma~\ref{lemm change}\eqref{change one} and Lemma~\ref{lemm change}\eqref{change five}, and this finishes the proof of Case~$2^\circ$.
\end{proof}

Let $\cB_\bd$ be the image of $\cA_\bd$ under the canonical epimorphism $\cA \to \cB$. Observe that if $k_1 = k_2$ and $j_1 = i_2 - 1 < 0$, then $x_{i_1, j_1}^{(k_1)} z_{i_2}^{(k_2)} = z_{i_1}^{(k_2)}$. Similarly, if $k_1 = k_2 + 1$ and $i_1 = j_1 = i_2 + \delta_{k, r} m \geq \delta_{k, r} m$, then $x_{i_1, j_1}^{(k_1)} z_{i_2}^{(k_2)} = z_{i_2 + 1}^{(k_2)}$. Thus we have the following consequence of Propositions~\ref{prop change} and~\ref{prop reduction}.

\begin{coro} \label{coro span}
$\cB_\bd$ is spanned by the images of the monomials
\[
z_{i_1}^{(k_1)} z_{i_2}^{(k_2)} \cdots z_{i_d}^{(k_d)} x_{s_1, t_1}^{(l_1)} x_{s_2, t_2}^{(l_2)} \cdots x_{s_n, t_n}^{(l_n)} \in \cA_\bd,
\]
$d, n \in \bbN$, such that
\begin{enumerate}

\item
$k_1 \geq k_2 \geq \cdots \geq k_d$,

\item
if $k_j = k_{j + 1}$, then either $|i_j| < |i_{j + 1}|$ or $|i_j| = |i_{j + 1}|$ and $i_j > i_{j + 1}$,

\item
$l_1 \geq l_2 \geq \cdots \geq l_n$,

\item
if $l_j = l_{j + 1}$, then either $s_j > s_{j + 1}$ or $s_j = s_{j + 1}$ and $t_j \geq t_{j + 1}$. \qed

\end{enumerate}
\end{coro}

Let $\cH_\bd$ be the subspace of $\cH$ spanned by the $M$'s such that $\bdim M \leq \bd$. Obviously, $\cH_\bd$ has a basis formed by the objects
\[
Z_{i_1}^{(k_1)} \oplus Z_{i_2}^{(k_2)} \oplus \cdots \oplus Z_{i_d}^{(k_d)} \oplus X_{s_1, t_1}^{(l_1)} \oplus X_{s_2, t_2}^{(l_2)} \cdots \oplus X_{s_n, t_n}^{(l_n)} \in \cH_\bd,
\]
such that the conditions from Corollary~\ref{coro span} are satisfied. Thus we have a natural bijection between the above spanning sets of $\cB_\bd$ and $\cH_\bd$. However these spaces are in general infinite dimensional, hence this does not imply immediately that $\ol{\xi}$ is a monomorphism.

If $\bd = (\bd_1, \bd_2)$, for $\bd_1 \in \bbN^{[1, r]}$ and $\bd_2 \in \bbN^{([1, r] \times \bbZ)}$, then we write $\bd_\cZ := \bd_1$ and $\bd_\cX := \bd_2$. For $d \in \bbN$, let $\cB_d$ be the sum of all $\cB_\bd$ such that $|(\bd_\cZ)| \leq d$. Moreover, if $\bd' \in \bbN^{([1, r] \times \bbZ)}$, then $\cB_{d, \bd'}$ is the sum of $\cB_\bd$ such that either $|(\bd_\cZ)| < d$ or $|(\bd_\cZ)| = d$ and $\bd_\cX \leq \bd'$. We define $\cH_d$ and $\cH_{d, \bd'}$ analogously.

Observe that Corollary~\ref{coro dim vect C} implies that $\xi$ induces maps $\ol{\xi}_d \colon \cB_d \to \cH_d$ and $\ol{\xi}_{d, \bd'} \colon \cB_{d, \bd'} \to \cH_{d, \bd'}$. Consequently, $\xi$ induces maps $\ol{\ol{\xi}}_d \colon \cB_d / \cB_{d - 1} \to \cH_d / \cH_{d - 1}$ and $\ol{\ol{\xi}}_{d, \bd'} \colon \cB_{d, \bd'}  / \cB_{d - 1} \to \cH_{d, \bd'} / \cH_{d - 1}$, where $\cB_{-1} := 0$ and $\cH_{-1} := 0$.

The next step in the proof is the following.

\begin{lemm} \label{lemm mono}
$\ol{\ol{\xi}}_{d, \bd}$ is a monomorphism.
\end{lemm}

\begin{proof}
Corollary~\ref{coro span} implies that $\cB_{d, \bd'} / \cB_{d - 1}$ is spanned by the classes of the images of the monomials
\[
\Phi = z_{i_1}^{(k_1)} z_{i_2}^{(k_2)} \cdots z_{i_d}^{(k_d)} x_{s_1, t_1}^{(l_1)} x_{s_2, t_2}^{(l_2)} \cdots x_{s_n, t_n}^{(l_n)},
\]
$n \in \bbN$, such that $\bdeg_\cX \Phi \leq \bd'$,
\begin{enumerate}

\item
$k_1 \geq k_2 \geq \cdots \geq k_d$,

\item
if $k_j = k_{j + 1}$, then either $|i_j| < |i_{j + 1}|$ or $|i_j| = |i_{j + 1}|$ and $i_j > i_{j + 1}$,

\item
$l_1 \geq l_2 \geq \cdots \geq l_n$,

\item
if $l_j = l_{j + 1}$, then either $s_j > s_{j + 1}$ or $s_j = s_{j + 1}$ and $t_j \geq t_{j + 1}$.

\end{enumerate}
Similarly, $\cH_{d, \bd} / \cH_{d - 1}$ has a basis consisting of the objects
\[
M = Z_{i_1}^{(k_1)} \oplus Z_{i_2}^{(k_2)} \oplus \cdots \oplus Z_{i_d}^{(k_d)} \oplus X_{s_1, t_1}^{(l_1)} \oplus X_{s_2, t_2}^{(l_2)} \cdots \oplus X_{s_n, t_n}^{(l_n)},
\]
such that $\bdim_\cX M \leq \bd'$ and the above conditions are satisfied. Thus
\[
\dim_\bbQ \cB_{d, \bd'} / \cB_{d - 1} \leq \dim_\bbQ \cH_{d, \bd'} / \cH_{d - 1}.
\]
Since these dimensions are finite and $\ol{\ol{\xi}}_{d, \bd}$ is an epimorphism by Proposition~\ref{prop poly prim}, the claim follows.
\end{proof}

We obtain the following consequence.

\begin{coro} \label{coro mono}
$\ol{\ol{\xi}}_d$ is a monomorphism.
\end{coro}

\begin{proof}
The claim follows from Lemma~\ref{lemm mono}, since
\[
\cB_d / \cB_{d - 1} = \bigcup_{\bd' \in \bbN^{([1, r] \times \bbZ)}} \cB_{d, \bd'} / \cB_{d - 1}. \qedhere
\]
\end{proof}

In order to apply Corollary~\ref{coro mono} we need a general observation.

\begin{prop} \label{prop crit mono}
Let $\zeta \colon V \to W$ be a linear map. Assume that, for each $d \in \bbN$, we have subspaces $V_d \subseteq V$ and $W_d \subseteq W$ such that the following conditions are satisfied:
\begin{enumerate}

\item
$V_d \subseteq V_{d + 1}$ and $W_d \subseteq W_{d + 1}$, for each $d \in \bbN$,

\item
$V = \bigcup_{d \in \bbN} V_d$ and $W = \bigcup_{d \in \bbN} W_d$,

\item
$f (V_d) \subseteq W_d$, for each $d \in \bbN$.

\end{enumerate}
If, for each $d \in \bbN$, the induced map $V_d / V_{d - 1} \to W_d / W_{d - 1}$ is a monomorphism, where $V_{-1} := 0$ and $W_{-1} := 0$, then $\zeta$ is a monomorphism.
\end{prop}

\begin{proof}
By induction on $d$ we prove that $f |_{V_d}$ is a monomorphism. The claim is obvious for $d = 0$.  Thus assume $v \in V_d$, for $d > 0$, and $f (v) = 0$. Then in particular, $f (v) \in W_{d - 1}$. Since the induced map $V_d / V_{d - 1} \to W_d / W_{d - 1}$ is a monomorphism, $v \in V_{d - 1}$. Consequently, $v = 0$, as by the inductive hypothesis $f |_{V_{d - 1}}$ is a monomorphism.
\end{proof}

We finish now the proof the Theorem~\ref{theo main prim}. As pointed out after the formulation of Theorem~\ref{theo main prim} we have to show that $\ol{\xi}$ is a monomorphism. However this follows from Proposition~\ref{prop crit mono} and Corollary~\ref{coro mono}.

\section{An application to one-cycle gentle algebras} \label{sect gentle}

By a gentle bound quiver we mean a pair $(Q, R)$, where $Q$ is a connected quiver and $R$ is a set of paths of length $2$ in $Q$, such that the following conditions are satisfied:
\begin{enumerate}

\item
for each $x \in Q_0$, there are at most two $\alpha \in Q_1$ such that $s \alpha = x$ (i.e.\ $\alpha$ starts at $x$),

\item
for each $x \in Q_0$, there are at most two $\alpha \in Q_1$ such that $t \alpha = x$ (i.e.\ $\alpha$ terminates at $x$),

\item
for each $\alpha \in Q_1$, there is at most one $\alpha' \in Q_1$ such that $s \alpha' = t \alpha$ and $\alpha' \alpha \not \in R$, and at most one $\alpha' \in Q_1$ such that $t \alpha' = s \alpha$ and $\alpha \alpha' \not \in R$,

\item
for each $\alpha \in Q_1$, there is at most one $\alpha' \in Q_1$ such that $s \alpha' = t \alpha$ and $\alpha' \alpha \in R$, and at most one $\alpha' \in Q_1$ such that $t \alpha' = s \alpha$ and $\alpha \alpha' \in R$.

\end{enumerate}
A gentle bound quiver $(Q, R)$ is called one-cycle if $|Q_0| = |Q_1|$. A finite dimensional algebra $\Lambda$ is called a one-cycle gentle algebra, if $\Lambda$ is Morita equivalent to the path algebra of a one-cycle gentle bound quiver.

A derived classification of the one-cycle gentle algebras has been obtained in~\cites{AssemSkowronski, BobinskiGeissSkowronski}. In order to present this classification we introduce some notation.

First, let $p \in \bbN_+$ and $q \in \bbN$. We denote by $Q (p, q)$ the quiver
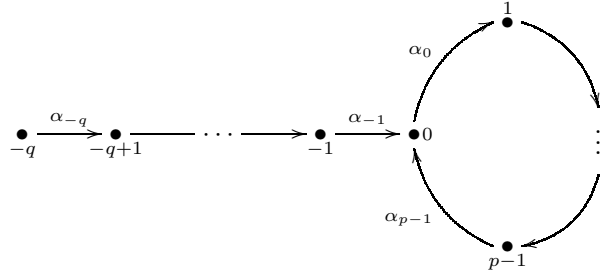
\[
\vcenter{\xymatrix{%
& & & & & \vertexD{1} \ar@/^/[]+R;[rd]+U &
\\
\vertexU{-q} \ar[r]^{\alpha_{-q}} & \vertexU{- q + 1} \ar@{-}[r] &
\cdots \ar[r] & \vertexU{-1} \ar[r]^{\alpha_{-1}} & \vertexL{0}
\ar@/^/[]+U;[ru]+L^{\alpha_0} & & \raisebox{-0.5em}[1em][1em]{\vdots} \ar@/^/[]+D;[ld]+R
\\
& & & & & \vertexU{p - 1} \ar@/^/[]+L;[lu]+D^{\alpha_{p - 1}} &}}
\]
If $r \in [1, p]$, then we denote by $\Lambda (p, q, r)$ the path algebra of the quiver $Q (p, q)$ bound by relations
\[
\alpha_0 \alpha_{p - 1}, \; \alpha_{p - 1} \alpha_{p - 2}, \; \ldots, \; \alpha_{p - r + 1} \alpha_{p - r}.
\]
Similarly, if $p, q \in \bbN_+$, then $\Lambda (p, q)$ is the path algebra of the quiver
\[
\vcenter{\xymatrix{%
& \vertexD{1} \ar@/_/[]+L;[ld]+U_{\alpha_0} & \ar[l] \cdots & \vertexD{p - 1} \ar[l]
\\
\vertexR{0} & & & & \vertexL{-q} \ar@/_/[]+U;[lu]+R_{\alpha_p} \ar@/^/[]+D;[ld]+R^{\alpha_{-q}}
\\
& \vertexU{-1} \ar@/^/[]+L;[lu]+D^{\alpha_{-1}} & \ar[l] \cdots & \vertexU{-q + 1} \ar[l]
}}
\]

The following is a consequence of~\cites{AssemSkowronski, Vossieck}.

\begin{theo} \label{theo onecycle}
Let $\Lambda$ be a gentle one-cycle algebra. Then there exists $p \in \bbN_+$ and $q \in \bbN$ such that either $\Lambda$ is derived equivalent to $\Lambda (p, q, r)$, for some $r \in [1, p]$, or $q > 0$ and $\Lambda$ is derived equivalent to $\Lambda (p, q)$. Moreover, $\gldim \Lambda = \infty$ if and only if $\Lambda$ is derived equivalent to $\Lambda (p, q, p)$, for some $p \in \bbN_+$ and $q \in \bbN$.  \qed
\end{theo}

Now let $\Lambda$ be a gentle one-cycle algebra of infinite global dimension. Our aim to describe the Hall algebras
\[
\cH (\cK^b (\proj \Lambda)) \qquad \text{and} \qquad \cH (\cD^b (\mod \Lambda)).
\]
Theorem~\ref{theo onecycle} implies that we may assume $\Lambda = \Lambda (p, q, p)$, for some $p \in \bbN_+$ and $q \in \bbN$. The categories $\cK^b (\proj \Lambda)$ and $\cD^b (\proj \Lambda)$ are known (see for example~\cites{Bobinski, BroomheadPauksztelloPloogI}). Namely, $\cK^b (\proj \Lambda)$ is equivalent to the category $\cX (p, p + q)$ described in Section~\ref{sect category}, while $\cD^b (\mod \Lambda)$ is equivalent to $\cC (p, p + q)$. Consequently, we get the following, which is the main result of the paper.

\begin{theo}
Let $\Lambda$ be an algebra derived equivalent to $\Lambda (p, q, p)$, for $p \in \bbN_+$ and $q \in \bbN$. Then Theorem~\ref{theo main} with $r = p$ and $m = p + q$ describes the Hall algebra $\cH (\cK^b (\proj \Lambda))$. Similarly, Theorem~\ref{theo main prim} with $r = p$ and $m = p + q$ describes the Hall algebra $\cH (\cD^b (\mod \Lambda))$. \qed
\end{theo}

\bibsection

\end{document}